\documentclass[a4paper, twoside]{scrartcl}
\usepackage{amsfonts,amssymb,amsmath,amsthm,amstext,amssymb,mathtools}
\usepackage{subcaption,float,color,xcolor,graphicx,enumitem}
\usepackage{dsfont}  % for indicator-one
\usepackage[normalem]{ulem} % cross out text using \sout{text}
\usepackage[authoryear,sort,round]{natbib}  % for bibliography
\usepackage{fullpage} % Package to use full page
\usepackage{hyperref,nicefrac}
\usepackage[noabbrev,capitalise,nosort,nameinlink]{cleveref}

\crefname{assumption}{Assumption}{Assumptions}

\usepackage{mathrsfs}
\newcommand{\Borel}[1]{\mathscr{B}(#1)}
\newcommand*{\cball}[2]{B_{#2}(#1)}
\newcommand*{\cballn}[3]{B_{#2}^{#3}(#1)}
\newcommand*{\oball}[2]{\mathring{B}_{#2}(#1)}
\newcommand*{\dd}{\Delta}
\newcommand*{\defeq}{\coloneqq}
\newcommand*{\defterm}{\emph}%{\textbf}
\newcommand*{\diam}{\mathop{\textup{diam}}\nolimits}
\newcommand*{\incomp}{\mathrel{\Vert}}

\newcommand*{\Naturals}{\mathbb{N}}
\newcommand*{\Rationals}{\mathbb{Q}}
\newcommand*{\one}{\mathds{1}}
\newcommand*{\prob}[1]{\mathscr{P}(#1)}

\newcommand*{\quark}{\setbox0\hbox{$x$}\hbox to\wd0{\hss$\cdot$\hss}}
\newcommand*{\rd}{\mathrm{d}}
\newcommand*{\Reals}{\mathbb{R}}

\DeclareMathOperator*{\supp}{supp}

\newcommand*{\cprime}{\bgroup\ensuremath'\!\egroup}

\newcommand*{\even}{\textup{e}}
\newcommand*{\odd}{\textup{o}}

\renewcommand*{\geq}{\geqslant}

\renewcommand*{\leq}{\leqslant}

\renewcommand*{\preceq}{\preccurlyeq}
\renewcommand*{\succeq}{\succcurlyeq}

\newcommand*{\crcdf}[2]{\mu(\cball{#1}{#2})}

\newcommand*{\liprec}{\preceq_0^{\liminf}}
\newcommand*{\lsprec}{\preceq_0^{\limsup}}
\newcommand*{\Liprec}{\preceq_0^{\Li}}
\newcommand*{\Lsprec}{\preceq_0^{\Ls}}

\newcommand*{\lsasymp}{\asymp_0^{\limsup}}

\DeclareMathOperator*{\Li}{Li}
\DeclareMathOperator*{\Ls}{Ls}

\newcommand*{\indexof}[2]{k(#1,#2)}
\newcommand*{\massof}[1]{m(#1)}
\newcommand*{\dyadicirratexp}[1]{\beta_2(#1)}
\newcommand*{\dyadicmeas}[2]{\varphi_2(#1, #2)}
\newcommand*{\irratmeas}[2]{\varphi(#1, #2)}
\newcommand*{\irratexp}[1]{\beta(#1)}

\newcommand*{\upclosure}{\mathop{\uparrow}\nolimits}

\newcommand*{\todo}[1]{\bgroup\color{red}#1\egroup}

\theoremstyle{plain}
\newtheorem{theorem}{\sffamily Theorem}[section]
\newtheorem{proposition}[theorem]{\sffamily Proposition}
\newtheorem{lemma}[theorem]{\sffamily Lemma}
\newtheorem{corollary}[theorem]{\sffamily Corollary}
\theoremstyle{definition}
\newtheorem{definition}[theorem]{\sffamily Definition}

\newtheorem{example}[theorem]{\sffamily Example}

\newtheorem{remark}[theorem]{\sffamily Remark}

\newcommand{\absval}[1]{\lvert #1 \rvert}

\newcommand{\norm}[1]{\lVert #1 \rVert}
\newcommand{\set}[2]{\{ #1 \mid #2 \}}

\newcommand{\bigset}[2]{\bigl\{ #1 \,\big\vert\, #2 \bigr\}}
\newcommand{\Absval}[1]{\left\vert #1 \right\vert}

\newcommand{\Set}[2]{\left\{ #1 \,\middle\vert\, #2 \right\}}

\numberwithin{equation}{section}
\numberwithin{figure}{section}
\numberwithin{table}{section}

\usepackage{etex}
\usepackage[a4paper, top=88pt, bottom=88pt, left=72pt, right=72pt, headsep=16pt, footskip=28pt]{geometry}
\usepackage{hyperref}
\hypersetup{
	colorlinks,
	linkcolor=blue,
	citecolor=blue,
	urlcolor=blue,
}
\newcommand*{\arXiv}[1]{\bgroup\color{blue}\href{https://arxiv.org/abs/#1}{arXiv:#1}\egroup}
\newcommand*{\doi}[1]{\bgroup\color{blue}\href{https://doi.org/#1}{doi:#1}\egroup}
\newcommand*{\email}[1]{\bgroup\color{blue}\href{mailto:#1}{#1}\egroup}
\renewcommand*{\url}[1]{\bgroup\color{blue}\href{#1}{#1}\egroup}
\usepackage{enumitem, moreenum}
\setlist[enumerate]{nosep}
\setlist[itemize]{nosep}
\usepackage{mleftright} \mleftright

\usepackage{xpatch}
\newcommand{\proofheadfont}{\bfseries\sffamily}
\xpatchcmd{\proof}{\itshape}{\proofheadfont}{}{}

\usepackage[labelfont={sf,bf}]{caption}
\usepackage{scrlayer-scrpage, xhfill}
\automark[section]{section}
\setkomafont{pageheadfoot}{\normalcolor\sffamily}
\setkomafont{pagenumber}{\normalfont\normalsize\sffamily}
\clearpairofpagestyles
\let\oldtitle\title
\renewcommand{\title}[1]{\oldtitle{#1}\newcommand{\theshorttitle}{#1}}
\newcommand{\shorttitle}[1]{\renewcommand{\theshorttitle}{#1}}
\let\oldauthor\author
\renewcommand{\author}[1]{\oldauthor{#1}\newcommand{\theshortauthor}{#1}}
\newcommand{\shortauthor}[1]{\renewcommand{\theshortauthor}{#1}}
\cohead{\xrfill[0.525ex]{0.6pt}~\theshorttitle~\xrfill[0.525ex]{0.6pt}}
\cehead{\xrfill[0.525ex]{0.6pt}~\theshortauthor~\xrfill[0.525ex]{0.6pt}}
\newcommand{\theabstract}[1]{\par\bgroup\noindent\textbf{\textsf{Abstract.}} #1\egroup}
\newcommand{\thekeywords}[1]{\par\smallskip\bgroup\noindent\textbf{\textsf{Keywords.}}\newcommand{\and}{ $\bullet$ } #1\egroup}
\newcommand{\themsc}[1]{\par\smallskip\bgroup\noindent\textbf{\textsf{2020 Mathematics Subject Classification.}}\newcommand{\and}{ $\bullet$ } #1\egroup}
\newcommand{\theshortmsc}[1]{\par\smallskip\bgroup\noindent\textbf{\textsf{2020 MSC.}}\newcommand{\and}{ $\bullet$ } #1\egroup}

\newcommand*{\affilref}[1]{\ref{affiliation#1}}
\newcommand*{\affiliation}[3]{
	\footnotetext[#1]{\label{affiliation#2}#3}
}

\title{An order-theoretic perspective on modes\\and maximum a posteriori estimation in\\ Bayesian inverse problems}
\shorttitle{An order-theoretic perspective on modes and MAP estimation}
\author{%
	Hefin Lambley\textsuperscript{\affilref{Warwick}}%
	\and%
	T.~J.~Sullivan\textsuperscript{\affilref{Warwick},\affilref{Turing}}%
}
\shortauthor{H.~Lambley and T.~J.~Sullivan}
\date{\today}

\begin{document}

\maketitle
\affiliation{1}{Warwick}{Mathematics Institute and School of Engineering, University of Warwick, Coventry, CV4 7AL, United Kingdom\newline (\email{hefin.lambley@warwick.ac.uk}, \email{t.j.sullivan@warwick.ac.uk})}
\affiliation{2}{Turing}{Alan Turing Institute, 96 Euston Road, London, NW1 2DB, United Kingdom}

\begin{abstract}\small
	\theabstract{It is often desirable to summarise a probability measure on a space $X$ in terms of a mode, or MAP estimator, i.e.\ a point of maximum probability.
Such points can be rigorously defined using masses of metric balls in the small-radius limit.
However, the theory is not entirely straightforward:
the literature contains multiple notions of mode and various examples of pathological measures that have no mode in any sense.
Since the masses of balls induce natural orderings on the points of $X$, this article aims to shed light on some of the problems in non-parametric MAP estimation by taking an order-theoretic perspective, which appears to be a new one in the inverse problems community.
This point of view opens up attractive proof strategies based upon the Cantor and Kuratowski intersection theorems;
it also reveals that many of the pathologies arise from the distinction between greatest and maximal elements of an order, and from the existence of incomparable elements of $X$, which we show can be dense in $X$, even for an absolutely continuous measure on $X = \mathbb{R}$.
}
	\thekeywords{Bayesian inverse problems%
\and%
local behaviour of measures%
\and%
maximum a posteriori estimation%
\and%
modes of probability measures%
\and%
orders on metric spaces%
}
	\theshortmsc{06F99% Order, lattices, ordered algebraic structures > Ordered structures > None of the above, but in this section
\and%
28A75% Measure and integration > Classical measure theory > Length, area, volume, other geometric measure theory
\and%
28C15% Measure and integration > Set functions and measures on spaces with additional structure > Set functions and measures on topological spaces
\and%
60B05% Probability theory and stochastic processes > Probability theory on algebraic and topological structures > Probability measures on topological spaces
\and%
62F10% Statistics > Parametric inference > Point estimation
\and%
62R20% Statistics > Statistics on algebraic and topological structures > Statistics on metric spaces
}
\end{abstract}

\section{Introduction}
\label{sec:introduction}

In diverse applications such as statistical inference and the analysis of transition paths of random dynamical systems it is desirable to summarise a complicated probability measure $\mu$ on a space $X$ by a single distinguished point $x^{\star} \in X$ that is, in some sense, a ``point of maximum probability'' under $\mu$ --- i.e.\ a \emph{mode} or, in the Bayesian context, a \emph{maximum a post\-eriori} (\emph{MAP}) \emph{estimator}.
Many optimisation-based approaches to inverse problems (e.g.\ Tikhonov regularisation of the misfit) aim to calculate or approximate modes, at least heuristically understood.
Over the last decade, it has become common to define modes in terms of masses of metric balls in the limit as the ball radius tends to zero, since this makes sense even when $X$ is a very general --- possibly infinite-dimensional --- space, as is often the case for modern inference problems \citep{Stuart2010}.

However, this ``small balls'' theory of modes is not entirely straightforward.
There are various definitions --- e.g.\ the strong mode of \citet{DashtiLawStuartVoss2013}, the generalised strong mode of \citet{Clason2019GeneralizedModes}, the weak mode of \citet{HelinBurger2015} --- with various subtle distinctions among them.
Even the existence theory for modes is not entirely straightforward:
there are already examples in the literature, and this article will supply further examples, of relatively simple probability measures that have no mode.
It can even be the case that the average of two disjointly supported unimodal probability measures may have no mode.

The purpose of this article is to formulate the notion of a mode in an order-theoretic manner and thereby to clarify some of these pathologies in the theory of modes.
We claim that this is a natural step to take in view of the heuristic understanding of modes as ``most probable points''.

With an order-theoretic point of view, many of the difficulties can be seen to arise from the distinction between greatest and maximal elements of a preordered set $(X, \preceq)$ when the preorder $\preceq$ is not total, i.e.\ when there exist \defterm{incomparable} $x, x' \in X$ for which neither $x \preceq x'$ nor $x \succeq x'$ holds.
Simply put, a greatest element must dominate every other element of $X$, whereas a maximal element need only dominate those with which it is comparable;
for a total preorder, maximal and greatest elements coincide.
Motivated by the needs of inverse problems theory, current notions of modes correspond to greatest elements.
However, many preorders lack maximal elements, and even those that have maximal elements may lack greatest elements;
this is exactly the situation of the examples discussed in \Cref{thm:oscillation_example,thm:countable_dense_antichain}.
Thus, one might argue that current notions of mode are order-theoretically ``too strong'', and perhaps maximal elements should be considered as modes, but possibly these are ``too weak'' for the needs of applications communities.
We hope that the present article will stimulate discussion on this point.

\paragraph{Outline of the paper.}
The rest of this paper is structured as follows:

\Cref{sec:notation} sets out basic notation for the rest of the paper, including a brief recap of necessary concepts from functional analysis, measure theory, and order theory.

\Cref{sec:related} gives an overview of related work in this area, in particular the ``small balls'' approach to defining MAP estimators for non-parametric statistical inverse problems.

\Cref{sec:positive-radius_preorder} introduces and analyses the total preorder $\preceq_{r}$ on $X$ induced by the $\mu$-measures of metric balls of fixed radius $r > 0$.
Because the preorder $\preceq_{r}$ is total, its maximal elements are also greatest, and can be seen as approximate ``radius-$r$ modes'' for $\mu$.
We are able to provide several criteria for the existence of such radius-$r$ modes $x_{r}^{\star}$ (\Cref{thm:r_greatest}) as well as examples of measures that admit none (\Cref{eg:no_radius_1_mode,eg:no_radius_r_mode}).
As a prelude to the next section, we also consider the convergence of $x_{r}^{\star}$ as $r \to 0$ (\Cref{thm:limits_of_r-modes,thm:AMFs_and_strong_modes}).

In \Cref{sec:limiting_preorders} we attempt to take the limit as $r \to 0$ of the preorders $\preceq_{r}$ to define a preorder $\preceq_{0}$ whose greatest elements will be weak modes of $\mu$.
However, because the preorder $\preceq_{0}$ is not total, the distinction between greatest and maximal elements becomes important.
Incomparable maximal elements are particularly troubling because their maximality means that one would like to think of them as candidate modes, yet their incomparability means that one cannot actually say which is ``most probable'' and hence a bona fide mode, as in \Cref{thm:oscillation_example}.
We show that antichains (collections of mutually incomparable elements) can be topologically dense in $X$ even when $\mu$ is absolutely continuous with respect to Lebesgue measure on $X \subseteq \Reals$ (\Cref{thm:countable_dense_antichain}).
We also show that measures with a continuous Lebesgue density may have incomparable elements, but never incomparable \emph{maximal} elements (\Cref{eg:countable_space_antichain}, \Cref{prop:essentially_total_examples}).

\Cref{sec:conclusion} gives some closing remarks, while technical supporting results can be found in \Cref{sec:technical}, and \Cref{sec:alternative_small-radius_preorders} discusses some alternatives to the limiting preorder $\preceq_{0}$ of \Cref{sec:limiting_preorders} and illustrates their shortcomings.

\clearpage

\section{Problem setting and notation}
\label{sec:notation}

\subsection{Spaces of interest}

Throughout, unless noted otherwise, $X$ will be a metric space with metric $d$;
we write $\Borel{X}$ for its Borel $\sigma$-algebra, i.e.\ the one generated by the closed balls $\cball{x}{r} \defeq \set{x' \in X}{d(x, x') \leq r}$, $x \in X$, $r \geq 0$;
we also write $\oball{x}{r} \defeq \set{x' \in X}{d(x, x') < r}$ for the corresponding open ball.
We will often assume that $X$ is complete and separable, and occasionally we will specialise to the case of $X$ being a separable Banach or Hilbert space.

\subsection{Measures of non-compactness and intersection theorems}

Our approach in \Cref{sec:positive-radius_preorder} will make much use of measures of non-compactness and intersection theorems;
see \citet[Sections~7.5--7.8]{MalkowskyRakocevic2019} for a thorough treatment of these concepts and their properties.

Briefly, given $A \subseteq X$, its \defterm{separation} (or \defterm{Istr\u{a}\c{t}escu}) \defterm{measure of non-compactness} is
\begin{equation}
	\label{eq:separation_measure_nc}
	\gamma(A) \defeq \inf \Set{ r \geq 0 }{ \text{there is no $(x_{n})_{n \in \Naturals} \subseteq A$ with } \inf_{\substack{ m, n \in \Naturals \\ m \neq n }} d(x_{m}, x_{n}) \geq r } .
\end{equation}
This is an increasing function with respect to inclusion of sets, is finite precisely when $A$ is bounded, and is zero precisely when $A$ is pre-compact.
The function $\gamma$ is bi-Lipschitz equivalent with several other measures of non-compactness such as the set (or Kuratowski) measure of non-compactness and the ball (or Hausdorff) measure of non-compactness.

\begin{theorem}[Generalised intersection theorem]
	\label{thm:intersection_theorem}
	Let $(A_{n})_{n \in \Naturals}$ be a decreasingly nested sequence of non-empty, closed subsets of a topological space $X$ and let $A \defeq \bigcap_{n \in \Naturals} A_{n}$.
	\begin{enumerate}[label=(\alph*)]
		\item
		\label{item:intersection_theorem_Cantor_1}
		(Cantor)
		If each $A_{n}$ is compact, then $A$ is non-empty and compact.

		\item
		\label{item:intersection_theorem_Cantor_2}
		(Cantor)
		If $X$ is a complete metric space and $\diam(A_{n}) \to 0$ as $n \to \infty$, then $A$ is a singleton.

		\item
		\label{item:intersection_theorem_Kuratowski}
		(Kuratowski)
		If $X$ is a complete metric space and $\gamma(A_{n}) \to 0$ as $n \to \infty$, then $A$ is non-empty and compact.
	\end{enumerate}
\end{theorem}

\subsection{Measure-theoretic concepts}

Given a metric space $X$, $\prob{X}$ denotes the set of all probability measures on $\Borel{X}$.
Absolute continuity of $\mu$ with respect to $\nu$ is denoted $\mu \ll \nu$.
The \defterm{topological support} of $\mu \in \prob{X}$ is
\begin{equation}
	\label{eq:supp_mu}
	\supp(\mu) \defeq \set{ x \in X }{ \text{for all $r > 0$, } \crcdf{x}{r} > 0 } ,
\end{equation}
which is always a closed subset of $X$, and is non-empty when $X$ is separable (or, equivalently, second countable or Lindel\"of) \citep[Theorem~12.14]{AliprantisBorder2006}.

The $n$-dimensional Lebesgue measure on $\Reals^{n}$ will be denoted $\lambda^{n}$.

The quantity $\crcdf{x}{r}$ will play a major role in this work, especially when thought of as a function of $r > 0$ for various choices of $x \in X$;
we shall call the map $r \mapsto \crcdf{x}{r}$ the \defterm{radial cumulative distribution function} (RCDF) and some of its key properties are given in \Cref{lem:RCDF} and \Cref{cor:RCDF_spherically_nonatomic}.

\subsection{Order-theoretic concepts}

We summarise here some basic terms from order theory;
for a comprehensive introduction to order theory, see e.g.\ \citet{Davey2022Introduction}.

In the course of this work, the set $X$ will be equipped with various \emph{preorders} $\preceq$, i.e.\ relations satisfying both
\begin{enumerate}[label=(\alph*)]
	\item \emph{reflexivity}:
	for all $x \in X$, $x \preceq x$;
	and
	\item \emph{transitivity}:
	for all $x, y, z \in X$,
	if both $x \preceq y$ and $y \preceq z$, then $x \preceq z$.
\end{enumerate}
For any such preorder, we will write $x \asymp x'$ if both $x \preceq x'$ and $x \succeq x'$ hold true, in which case $x$ and $x'$ are called \emph{equivalent}\footnote{A preorder $\preceq$ is called a \emph{partial order} if it is \emph{antisymmetric}, i.e.\ if $x \asymp x' \implies x = x'$, but almost none of the preorders that we consider will actually be partial orders.} in the preorder;
we write $x \prec x'$ if $x \preceq x'$ but $x \not\succeq x'$.

If at least one of $x \preceq x'$ and $x \succeq x'$ holds true, then we call $x$ and $x'$ \emph{comparable};
if neither holds, then we call them \emph{incomparable} and write $x \incomp x'$.
A preorder $\preceq$ is \emph{total} or \emph{linear} if there are no incomparable elements.
A subset of $X$ on which $\preceq$ is total is called a \emph{chain}, and a subset for which every two distinct elements are incomparable is called an \emph{antichain}.

We highlight and contrast two notions of a ``biggest'' element for a preorder:

\begin{definition}
	\label{defn:greatest_maximal_element}
	Let $X$ be a set equipped with a preorder $\preceq$.
	\begin{enumerate}[label=(\alph*)]
		\item $g \in X$ is a \defterm{greatest element} if, for every $x \in X$, $g \succeq x$.
		\item $m \in X$ is a \defterm{maximal element} if, whenever $x \in X$ is such that $m \preceq x$, it follows that $m \succeq x$ (and hence $m \asymp x$).
		Equivalently, $m$ is maximal if there is no $x \in X$ with $x \succ m$.
		\item $u \in X$ is an \defterm{upper bound} for $A \subseteq X$ if, for all $x \in A$, $u \succeq x$.
	\end{enumerate}
\end{definition}

Note in particular that a greatest element is also a maximal element, but it must additionally be comparable to (and dominate) every element of $X$.
On the other hand, a maximal element is only required to dominate those elements of $X$ with which it is comparable, and those elements could constitute a rather small subset of $X$.

The most famous statement about the existence of maximal elements is \emph{Zorn's lemma}:
under the Axiom of Choice, if $(X, \preceq)$ is a preordered space in which every chain $Y \subseteq X$ has an upper bound, then $X$ has at least one maximal element.
However, Zorn's lemma says nothing about the existence of greatest elements.

We write $\upclosure Y \defeq \set{ x \in X }{ x \succeq y \text{ for some } y \in Y }$ for the \defterm{upward closure} of $Y \subseteq X$, and further write, for $y \in X$, $\upclosure y \defeq \upclosure \{ y \} = \set{ x \in X }{ x \succeq y }$, so that $\upclosure Y = \bigcup_{y \in Y} \upclosure y$.

Finally, since many of the preorders we consider will be parametrised by radius $r \geq 0$, we will write $\preceq_{r}$ for the preorder, $\incomp_{r}$ for the induced relation of incomparability, $\upclosure_{r} Y$ for the upward closure of $Y$ with respect to $\preceq_{r}$, etc.

\section{Overview of related work}
\label{sec:related}

Modes, loosely understood as points of maximum probability, arise in many areas of pure and applied mathematics.
Two application domains where modes are particularly prominent are the analysis of the transition paths of random dynamical systems and the Bayesian approach to inverse problems.

The random dynamical systems setting is exemplified by mathematical models of chemical reactions using diffusion processes.
One is typically interested in the (rare) transitions of the process from one energy well or metastable state to another, and in particular one wishes to understand the transition paths that a diffusion process is most likely to take.
This amounts to a study of the modes of the law $\mu$ of the diffusion process on the associated path space $X$;
e.g.\ for a molecule consisting of $n$ atoms in three-dimensional space, $X = C([0, T]; \Reals^{3 n})$.
The modes of $\mu$ are understood as \emph{minimum-action paths}, and the behaviour of $\mu$ near the mode is quantified using Freidlin--Wentzell theory or large deviations theory \citep{DemboZeitouni1998,ERenVandenEijnden2004,FreidlinWentzell1998}.

In the Bayesian approach to inverse problems \citep{KaipioSomersalo2005, Stuart2010}, the reconstruction of an $X$-valued parameter of interest from observed $Y$-valued data is expressed in the form of a probability measure $\mu \in \prob{X}$, the \emph{posterior distribution}.
In many modern inverse problems, particularly those coupled to partial differential equations, the space $X$ is an infinite-dimensional function space or a high-dimensional discretisation of such a space, e.g.\ via a system of finite elements.

The posterior measure $\mu$ arises from three ingredients:
a \emph{prior measure} $\mu_{0} \in \prob{X}$, which encodes (subjective) beliefs about the parameter that are held in advance of knowing the observation mechanism or the specific data that are observed;
a \emph{likelihood model}, i.e.\ a family of probability measures $L(\quark|x) \in \prob{Y}$, one for each $x \in X$, which models how observed data would be expected to arise if the parameter value $x$ were the truth;
and a specific observed instance of the data, a point $y \in Y$.
Strictly speaking, the posterior measure $\mu$ is defined as the disintegration (conditional distribution) of the joint measure $\nu (\rd x, \rd y) \defeq L(\rd y|x) \mu_{0}(\rd x) \in \prob{X \times Y}$ along the $y$-fibre \citep{ChangPollard1997}.
For simplicity, however, we often concentrate on the case that $\mu$ has a density  with respect to $\mu_{0}$ given by Bayes' formula,
\begin{equation}
	\label{eq:Bayes}
	\mu (\rd x) = \frac{ \exp(- \Phi(x; y)) \, \mu_{0}(\rd x) }{ \int_{X} \exp(- \Phi(x'; y)) \, \mu_{0}(\rd x') } ,
\end{equation}
where $\Phi \colon X \times Y \to \Reals$ is called the \emph{potential}.
In simple settings with $\dim Y < \infty$, the Lebesgue probability density of $L(\quark|x)$ is proportional to $\exp( - \Phi(x; \quark))$ and $\Phi$ can be interpreted as a non-negative misfit functional.
The case of infinite-dimensional data, $\dim Y = \infty$, is considerably more subtle and does not generally admit a density for $\mu$ with respect to $\mu_{0}$ as in \eqref{eq:Bayes};
see e.g.\ \citet[Remark~3.8]{Stuart2010} and \citet[Remark~9]{Lasanen2012_I}.

Since the full posterior distribution $\mu$ can be a rather intractable object, it is often desirable to have access to a convenient point summary:
the two principal such point estimators are the \emph{conditional mean estimator} (i.e.\ the mean of $\mu$) and a \emph{maximum a posteriori estimator} (i.e.\ a mode, or point of maximum probability, for $\mu$), and here we focus on this second approach.
Heuristically, at least when $X = \Reals^{d}$, a MAP estimator is just an essential maximiser of the Lebesgue density of $\mu$, i.e.\ a minimiser of the sum of $\Phi(\quark; y)$ and the negative logarithm of the Lebesgue density of $\mu_{0}$.
However, this definition is not effective if we have no access to Lebesgue densities;
in particular, it makes no sense when $\dim X = \infty$ \citep[e.g.][]{Sudakov1959}.

To handle the general infinite-dimensional case, various definitions of modes / MAP estimators have been advanced over recent years, and we summarise them here.\footnote{The definitions of \citet{DashtiLawStuartVoss2013}, \citet{HelinBurger2015}, and \citet{Clason2019GeneralizedModes} were all stated in the case of a separable Banach space $X$, but they generalise easily to the metric setting, as given here.
Also, their definitions used open rather than closed balls.}
One approach \citep{DuerrBach1978} is to understand a mode of the path measure $\mu$ of a diffusion process as a minimiser of the \defterm{Onsager--Machlup (OM) functional} $I_{\mu}$ of $\mu$, which is defined by the relation
\begin{equation}
	\label{eq:Onsager--Machlup}
	\lim_{r \to 0} \frac{\crcdf{x}{r}}{\crcdf{x'}{r}} = \frac{\exp(-I_{\mu}(x))}{\exp(-I_{\mu}(x'))}
	\quad
	\text{for $x, x' \in X$.}
\end{equation}
In some sense, $I_{\mu}$ is a formal negative log-density for $\mu$, but it is in general only a partially defined extended-real-valued function.
For example, the OM functional of a Gaussian measure on a Hilbert space is finite only on the Cameron--Martin space.
The rigorous interpretation of modes as minimisers of $I_{\mu}$ requires considerable care, especially since in some cases it is not even possible to assign $+\infty$ as an exceptional value for $I_{\mu}$:
the ratio in \eqref{eq:Onsager--Machlup} may oscillate and fail to converge as $r \to 0$.

A \defterm{strong mode} of $\mu$ was defined by \citet{DashtiLawStuartVoss2013} to be any $x^{\star} \in X$ such that
\begin{align}
	\label{eq:strong_mode}
	\lim_{r \to 0} \frac{\crcdf{x^{\star}}{r}}{M_{r}} = 1 , \\
	\label{eq:M_r}
	M_{r} \defeq \sup_{x \in X} \crcdf{x}{r} .
\end{align}
(By \Cref{cor:no_unbounded_sequence_approximates_M_r}, separability of $X$ ensures that $\supp(\mu) \neq \varnothing$ and $M_{r} > 0$.)
Any strong mode must lie in $\supp(\mu)$, and the ratio in \eqref{eq:strong_mode} is at most 1 for every choice of $x^{\star} \in X$, so
\begin{equation}
	\label{eq:strong_mode_equivalent}
	\text{$x^{\star}$ is a strong mode}
	\iff
	\liminf_{r \to 0} \frac{\crcdf{x^{\star}}{r}}{M_{r}} \geq 1
	\iff
	\limsup_{r \to 0} \frac{M_{r}}{\crcdf{x^{\star}}{r}} \leq 1 .
\end{equation}
However, \citet{Clason2019GeneralizedModes} observed that even elementary absolutely continuous measures on $\Reals$ such as $\mu(E) \defeq \int_{E \cap [-1, 1]} \absval{x} \, \rd x$ do not have strong modes, even though the Lebesgue density of $\mu$ is clearly maximised at $\pm 1$.
Therefore, they call $x^{\star} \in X$ a \defterm{generalised strong mode} if, for every positive null sequence $(r_{n})_{n \in \Naturals}$, there exists a sequence $(x_{n})_{n \in \Naturals}$ converging to $x^{\star}$ such that
\begin{equation}
	\label{eq:generalised_strong_mode}
	\lim_{n \to \infty} \frac{\crcdf{x_{n}}{r_{n}}}{M_{r_{n}}} = 1 .
\end{equation}

Motivated by \eqref{eq:strong_mode_equivalent}, $x^{\star} \in \supp(\mu) \subseteq X$ is called a \defterm{weak mode} \citep{HelinBurger2015} if\footnote{In fact, \citet{HelinBurger2015} used ``$\lim$'' in place of ``$\limsup$'' in \eqref{eq:weak_mode}, implicitly assuming the existence of the limit.
However, as \citet{AyanbayevKlebanovLieSullivan2022_I} observe, this yields an unsatisfying definition because it excludes the case in which the ratio oscillates, while remaining bounded away from unity, from being a weak mode.
The desirable implication ``strong mode $\implies$ weak mode'' fails for the original ``$\lim$'' version of the definition, but holds for the ``$\limsup$'' version.}
\begin{equation}
	\label{eq:weak_mode}
	\limsup_{r \to 0} \frac{\crcdf{x'}{r}}{\crcdf{x^{\star}}{r}} \leq 1
	\text{ for all } x' \in X .
\end{equation}
As a point of terminology, \citet{HelinBurger2015} were primarily interested in the restricted case that $x' \in x^{\star} + E$, where $x^{\star} \in E$ and $E$ is a topologically dense linear subspace of a Banach space $X$, and \citet{LieSullivan2018} later called this case an \defterm{$E$-weak mode}.
Conversely, \citet{AyanbayevKlebanovLieSullivan2022_I} call $x^{\star}$ satisfying \eqref{eq:weak_mode} a \defterm{global weak mode}.
Since we are only going to consider global weak modes, we can simply call them \defterm{weak modes} without any ambiguity.

Under the assumption that the OM functional $I_{\mu}$ of $\mu$ is real-valued on $\varnothing \neq E \subseteq X$ and
\begin{equation}
	\label{eq:property_M}
	\text{for some $x \in E$ and all $x' \in X \setminus E$,} \quad \lim_{r \to 0} \frac{\crcdf{x'}{r}}{\crcdf{x}{r}} = 0 ,
\end{equation}
which \citet[Definition~3.1]{AyanbayevKlebanovLieSullivan2022_I} call \emph{property $M(\mu, E)$}, $I_{\mu}$ can be regarded as having the value $+ \infty$ on $X \setminus E$ and weak modes are precisely minimisers of this extended version of $I_{\mu}$.
This enabled \citet{AyanbayevKlebanovLieSullivan2022_I, AyanbayevKlebanovLieSullivan2022_II} to establish a stability and convergence theory for weak modes in terms of the $\Gamma$-convergence and equicoercivity of the associated OM functionals.\footnote{Frustratingly, while there are some situations in which strong modes can be characterised as minimisers of Onsager--Machlup functionals \citep{AgapiouBurgerDashtiHelin2018,DashtiLawStuartVoss2013}, there are also situations in which this correspondence breaks down, even when property $M(\mu, E)$ holds \citep[Example~B.5]{AyanbayevKlebanovLieSullivan2022_I}.}
Furthermore, as we show below in \Cref{lem:GWM_greatest_maximal}, weak modes are exactly the greatest elements of a natural preorder $\preceq_{0}$ on $X$, namely the one induced by the limiting ratios of masses of balls in the small-radius limit (\Cref{defn:analytic_small-radius_preorder}).

There are also local versions of the strong and weak modes \citep{AgapiouBurgerDashtiHelin2018}, in which $x^{\star}$ is only compared to points in a sufficiently small ball $\cball{x^{\star}}{\delta}$, $\delta > 0$, analogous to local maximisers of the Lebesgue probability density function / local minimisers of the OM functional.

For $\mu$ of the form \eqref{eq:Bayes} with $\mu_{0}$ Gaussian and $X$ Hilbert, \citet{DashtiLawStuartVoss2013} proved that $\mu$ has a strong mode by studying maximisers of the radius-$r$ ball mass $x \mapsto \crcdf{x}{r}$ for fixed $r > 0$ --- which we call radius-$r$ modes in \Cref{sec:positive-radius_preorder} --- and arguing that a sequence of such maximisers must converge to a strong mode.
The arguments of \citet{DashtiLawStuartVoss2013} assume the existence of radius-$r$ modes without proof; in \Cref{subsection:existence_of_radius_r}, we prove results on the existence of radius-$r$ modes in various settings but also provide examples that have no such radius-$r$ modes.
Despite the contributions of \citet{DashtiLawStuartVoss2013}, \citet{Kretschmann2019}, and \citet{KlebanovWacker2022} among others --- and our own offerings --- a surprising amount is still unknown about the existence of radius-$r$ modes, let alone weak and strong modes, even for ``nicely'' reweighted Gaussian measures on Banach spaces.

\section{The positive-radius preorder}
\label{sec:positive-radius_preorder}

\subsection{Definition and basic properties}

A probability measure on a metric space $X$ induces a family of preorders on $X$, one for each positive radius, in a very straightforward way:

\begin{definition}[Positive-radius preorder]
	\label{defn:positive_radius_preorder}
	Let $X$ be a metric space and let $\mu \in \prob{X}$.
	For each $r > 0$, define a relation $\preceq_{r}$ on $X$ by
	\begin{align}
		\label{eq:preceq_r}
		x \preceq_{r} x'
		& \iff \crcdf{x}{r} \leq \crcdf{x'}{r} .
	\end{align}
\end{definition}

It is almost trivial to verify that $\preceq_{r}$ satisfies the axioms for a preorder.
We will write $x \asymp_{r} x'$ if both $x \preceq_{r} x'$ and $x \succeq_{r} x'$ hold, and $x \incomp_{r} x'$ if neither $x \preceq_{r} x'$ nor $x \succeq_{r} x'$ hold.
In fact, though, incomparability never arises for this preorder:
totality of the usual order $\leq$ on $\Reals$ implies totality of $\preceq_{r}$ on $X$.
Totality implies that the maximal and greatest elements of $X$ with respect to $\preceq_{r}$ coincide (\Cref{lem:radius_r_mode}), which simplifies the discussion considerably.

Upward closures with respect to $\preceq_{r}$ are notably well behaved.
In particular, \Cref{lem:upper_closure_is_closed_and_bounded}\ref{lem:upper_closure_is_closed_and_bounded_2} says that the relation $\preceq_{r}$ is \emph{upper semicontinuous} \citep[p.44]{AliprantisBorder2006}.

\begin{lemma}[Closedness, boundedness, and non-compactness of upward closures]
	\label{lem:upper_closure_is_closed_and_bounded}
	Let $X$ be a metric space, let $\mu \in \prob{X}$, and fix $r > 0$.
	\begin{enumerate}[label=(\alph*)]
		\item \label{lem:upper_closure_is_closed_and_bounded_1}
		For each $t \geq 0$, $\set{ x' \in X }{ \crcdf{x'}{r} \geq t }$ is closed.
		\item \label{lem:upper_closure_is_closed_and_bounded_2}
		For each $x \in X$, $\upclosure_{r} x \defeq \set{ x' \in X }{ x' \succeq_{r} x }$ is closed.
		\item \label{lem:upper_closure_is_closed_and_bounded_3}
		For each $t > 0$, $\set{ x' \in X }{ \crcdf{x'}{r} \geq t }$ is bounded, with separation measure of non-compactness $\gamma( \set{ x' \in X }{ \crcdf{x'}{r} \geq t } ) \leq 2 r$.
		\item \label{lem:upper_closure_is_closed_and_bounded_4}
		For each $x \in X$ with $\crcdf{x}{r} > 0$, $\upclosure_{r} x$ is bounded with $\gamma(\upclosure_{r} x) \leq 2 r$.
	\end{enumerate}
\end{lemma}

\begin{proof}
	Claim \ref{lem:upper_closure_is_closed_and_bounded_1} is immediate from the upper semicontinuity of the map $x \mapsto \crcdf{x}{r}$ (\Cref{lem:RCDF}\ref{lem:RCDF_in_x}), and \ref{lem:upper_closure_is_closed_and_bounded_2} is a special case of claim \ref{lem:upper_closure_is_closed_and_bounded_1}.

	Now fix $t > 0$ and suppose for a contradiction that $(x_{n})_{n \in \Naturals}$ is an unbounded sequence in $\set{ x' \in X }{ \crcdf{x'}{r} \geq t }$.
	By passing to a subsequence if necessary, we may assume that $d(x_{n}, x_{n'}) > 2 r$ for all distinct $n, n' \in \Naturals$.
	We thus obtain the contradiction that
	\begin{equation*}
		1 = \mu(X) \geq \mu \left( \biguplus_{n \in \Naturals} \cball{x_{n}}{r} \right) = \sum_{n \in \Naturals} \crcdf{x_{n}}{r} \geq \sum_{n \in \Naturals} t = \infty .
	\end{equation*}
	This shows that $\set{ x' \in X }{ \crcdf{x'}{r} \geq t }$ must be bounded and also that it admits no infinite subset with separation $2 r$, thus establishing \ref{lem:upper_closure_is_closed_and_bounded_3}, of which \ref{lem:upper_closure_is_closed_and_bounded_4} is a special case.
\end{proof}

\subsection{Existence and absence of greatest elements}
\label{subsection:existence_of_radius_r}

Our first aim is to establish existence of greatest elements for $\preceq_{r}$, which we also call \emph{radius-$r$ modes}.
Such points can be seen as approximate modes\footnote{The intuition that radius-$r$ modes are approximate modes must be treated sceptically.
For example, consider $\mu \in \prob{X}$ with bimodal continuous Lebesgue density $\rho(x) \propto \max \{ 0, 1 - 4 (x - 1)^{2} \} + \max \{ 0, 1 - 4 (x + 1)^{2} \}$, for which a radius-$1$ mode is located at $0$, which is neither a maximiser of $\rho$ nor even in $\supp(\mu)$.} with respect to the positive radius / spatial resolution $r$;
only in the next section will we attempt to take the limit as $r \searrow 0$.

\Cref{lem:radius_r_mode} now gives several equivalent conditions for a point to be a radius-$r$ mode.
The intersection criterion \ref{item:radius_r_mode_seq_intersection} will prove especially helpful in what follows, in the sense that we establish existence of radius-$r$ modes by showing that intersections of this type are non-empty.

\begin{lemma}[Characterisation of radius-$r$ modes]
	\label{lem:radius_r_mode}
	Let $X$ be any metric space, let $\mu \in \prob{X}$, and let $r > 0$.
	As in \eqref{eq:M_r}, let $M_{r} \defeq \sup_{x \in X} \crcdf{x}{r}$.
	Then the following are equivalent and if one (and hence any) holds, then $x_{r}^{\star} \in X$ is called a \defterm{radius-$r$ mode}:
	\begin{enumerate}[label=(\alph*)]
		\item
		\label{item:radius_r_mode_maximal}
		$x_{r}^{\star}$ is a $\preceq_{r}$-maximal element;

		\item
		\label{item:radius_r_mode_greatest}
		$x_{r}^{\star}$ is a $\preceq_{r}$-greatest element;
		
		\item
		\label{item:radius_r_mode_full_intersection}
		$x_{r}^{\star} \in \bigcap_{x \in X} \upclosure_{r} x$;
		
		\item
		\label{item:radius_r_mode_seq_intersection}
		$x_{r}^{\star} \in \bigcap_{n \in \Naturals} \upclosure_{r} x_{n}$ for some sequence $(x_{n})_{n \in \Naturals} \subseteq X$ with $\crcdf{x_{n}}{r} \nearrow M_{r}$ as $n \to \infty$;

		\item
		\label{item:radius_r_mode_Mr}
		$\crcdf{x_{r}^{\star}}{r} = M_{r}$.
	\end{enumerate}
\end{lemma}

\begin{proof}
	\mbox{(\ref{item:radius_r_mode_maximal}$\iff$\ref{item:radius_r_mode_greatest})}$\quad$
	This equivalence holds because $\preceq_{r}$ is a total preorder.
	
	\noindent\mbox{(\ref{item:radius_r_mode_greatest}$\iff$\ref{item:radius_r_mode_full_intersection})}
	This equivalence is simply a restatement of the definition of being greatest.
	
	\noindent\mbox{(\ref{item:radius_r_mode_full_intersection}$\implies$\ref{item:radius_r_mode_seq_intersection})}
	This implication is obvious, since $\bigcap_{n \in \Naturals} \upclosure_{r} x_{n} \supseteq \bigcap_{x \in X} \upclosure_{r} x$.
	
	\noindent\mbox{(\ref{item:radius_r_mode_seq_intersection}$\implies$\ref{item:radius_r_mode_Mr})}
	Let $(x_{n})_{n \in \Naturals}$ be a sequence in $X$ with $\crcdf{x_{n}}{r} \nearrow M_{r}$ and let $x_{r}^{\star} \in \bigcap_{n \in \Naturals} \upclosure_{r} x_{n}$.
	Then, for each $n$, $\crcdf{x_{r}^{\star}}{r} \geq \crcdf{x_{n}}{r}$, and taking the limit as $n \to \infty$ shows that $\crcdf{x_{r}^{\star}}{r} \geq M_{r}$.
	The definition of $M_r$ implies that $\crcdf{x_{r}^{\star}}{r} \leq M_r$, and so $\crcdf{x_{r}^{\star}}{r} = M_r$.
	
	\noindent\mbox{(\ref{item:radius_r_mode_Mr}$\implies$\ref{item:radius_r_mode_greatest})}
	Suppose that $x_{r}^{\star}$ has $\crcdf{x_{r}^{\star}}{r} = M_r$.
	Then, for any $x \in X$, $\crcdf{x}{r} \leq M_{r}$, i.e.\ $x \preceq_{r} x_{r}^{\star}$.
	Thus, $x_{r}^{\star}$ is $\preceq_{r}$-greatest.
\end{proof}

A very simple existence result for radius-$r$ modes is the following:

\begin{proposition}[Existence of radius-$r$ modes in compact spaces]
	\label{prop:compact_yields_r-mode}
	Let $X$ be a compact metric space, let $\mu \in \prob{X}$, and let $r > 0$.
	Then $\preceq_{r}$ has at least one radius-$r$ mode $x_{r}^{\star} \in X$.
\end{proposition}

\begin{proof}
	This is a special case of \Cref{thm:r_greatest}\ref{item:r_greatest_HB}, and also follows from \citet[Theorem~2.44]{AliprantisBorder2006}, but a self-contained proof is given by observing that the map $\crcdf{\quark}{r} \colon X \to [0, 1]$ is upper semicontinuous (\Cref{lem:RCDF}\ref{lem:RCDF_in_x}) and hence has at least one global maximiser $x_{r}^{\star}$ in the compact space $X$.
\end{proof}

We now adopt a very different approach to establishing the existence of radius-$r$ modes, one based on applying intersection theorems to upward closures with respect to $\preceq_{r}$.
We begin with a very general lemma;
when \Cref{lem:r_greatest_T_compact} is used in practice, $\mathcal{T}$ will often be the metric topology, but another useful case is the weak topology of a Banach space.

\begin{lemma}
	\label{lem:r_greatest_T_compact}
	Let $X$ be a separable metric space, $\mu \in \prob{X}$, and $r > 0$.
	Suppose that $\mathcal{T}$ is a topology on $X$ such that, for some sequence $(x_{n})_{n \in \Naturals} \subset X$ with $\crcdf{x_{n}}{r} \nearrow M_{r} > 0$, $\upclosure_{r} x_{n}$ is $\mathcal{T}$-closed and $\mathcal{T}$-compact for all sufficiently large $n$.
	Then the set $\mathfrak{M}_{r}$ of radius-$r$ modes for $\mu$ is non-empty, $\mathcal{T}$-compact, and $\mathfrak{M}_{r} = \bigcap_{n \in \Naturals} \upclosure_{r} x_n$.
\end{lemma}

\begin{proof}
	Separability of $X$ implies that $M_{r} > 0$.
	Let $(x_{n})_{n \in \Naturals}$ be such that $\crcdf{x_{n}}{r} \nearrow M_{r}$ as $n \to \infty$.
	The sets $\upclosure_{r} x_{n}$ are non-empty;
	since the sequence $\bigl(\crcdf{x_{n}}{r}\bigr)_{n \in \Naturals}$ is increasing, $\upclosure_{r} x_{n + 1} \subseteq \upclosure_{r} x_{n}$ for each $n$, i.e.\ they are decreasingly nested;
	by hypothesis, for sufficiently large $n$, they are also $\mathcal{T}$-closed and $\mathcal{T}$-compact.
	Therefore, by Cantor's intersection theorem (\Cref{thm:intersection_theorem}\ref{item:intersection_theorem_Cantor_1}), their intersection is non-empty and $\mathcal{T}$-compact.
	This intersection is precisely the set $\mathfrak{M}_{r}$ of radius-$r$ modes, as already shown by \Cref{lem:radius_r_mode}.
\end{proof}

\begin{theorem}[Existence of radius-$r$ modes]
	\label{thm:r_greatest}
	Let $X$ be a separable metric space, $\mu \in \prob{X}$, and $r > 0$.
	Let $\mathfrak{M}_{r}$ denote the set of radius-$r$ modes for $\mu$.
	\begin{enumerate}[label=(\alph*)]
		\item
		\label{item:r_greatest_HB}
		Suppose that $X$ has the Heine--Borel property, i.e.\ that every closed and bounded subset of $X$ is compact.
		Then $\mathfrak{M}_{r}$ is non-empty and compact.

		\item
		\label{item:r_greatest_doubling}
		Suppose that $X$ is complete and that $\mu$ is a doubling measure, i.e.\ there exists a constant $C > 0$ such that $\crcdf{x}{2r} \leq C \crcdf{x}{r}$ for all $x \in X$ and $r > 0$.
		Then $\mathfrak{M}_{r}$ is non-empty and compact.

		\item
		\label{item:r_greatest_lower_bound}
		Suppose that $X$ is complete and there exists a point $o \in X$ and a function $f \colon (0, \infty)^{2} \to (0, \infty)$ such that
		\begin{equation}
			\label{eq:r_greatest_lower_bound}
			\text{for all $x \in \cball{o}{R}$,}
			\quad
			\crcdf{x}{\delta} \geq f(\delta, R) > 0 .
		\end{equation}
		Then $\mathfrak{M}_{r}$ is non-empty and compact.

		\item
		\label{item:r_greatest_vanishing_MNC}
		Suppose that $X$ is complete and that there exists $(x_{n})_{n \in \Naturals}$ with $\crcdf{x_{n}}{r} \nearrow M_{r}$ and $\gamma(\upclosure_{r} x_{n}) \to 0$ as $n \to \infty$.
		Then $\mathfrak{M}_{r}$ is non-empty and compact.

		\item
		\label{item:r_greatest_vanishing_diameter}
		Suppose that $X$ is complete and that there exists $(x_{n})_{n \in \Naturals}$ with $\crcdf{x_{n}}{r} \nearrow M_{r}$ and $\diam(\upclosure_{r} x_{n}) \to 0$ as $n \to \infty$.
		Then $\mathfrak{M}_{r}$ is a singleton.

		\item
		\label{item:r_greatest_weakly_compact}
		Suppose that $X$ is a Banach space and that there exists $(x_{n})_{n \in \Naturals}$ with $\crcdf{x_{n}}{r} \nearrow M_{r}$, and that $\upclosure_{r} x_{n}$ is weakly compact for all sufficiently large $n$.
		Then $\mathfrak{M}_{r}$ is non-empty and weakly compact.

		\item
		\label{item:r_greatest_convex}
		Suppose that $X$ is a reflexive Banach space and that there exists $(x_{n})_{n \in \Naturals}$ with $\crcdf{x_{n}}{r} \nearrow M_{r}$ and that $\upclosure_{r} x_{n}$ is convex for all sufficiently large $n$.
		Then $\mathfrak{M}_{r}$ is non-empty, weakly compact, and convex.
	\end{enumerate}
\end{theorem}

\begin{proof}
	\begin{enumerate}[label=(\alph*)]
		\item
		\Cref{lem:upper_closure_is_closed_and_bounded} ensures that every upward closure $\upclosure_{r} x$, $x \in X$, is closed and bounded in the metric topology on $X$.
		The claim now follows from \Cref{lem:r_greatest_T_compact}.

		\item
		By \citet[Proposition~3.1]{BjoernBjoern2011}, any complete metric space with a doubling measure has the Heine--Borel property.
		The claim now follows from \ref{item:r_greatest_HB}.

		\item
		Let $R > 0$ and $\delta > 0$ be arbitrary.
		The lower bound \eqref{eq:r_greatest_lower_bound} implies that there cannot be an infinite set of pairwise-disjoint balls $\cball{x_{n}}{\delta}$, $n \in \Naturals$, with centres $x_{n} \in \cball{o}{R}$ since, if there were, then we would obtain the contradiction
		\begin{equation*}
			1 \geq \crcdf{o}{R + \delta} \geq
			\mu \left( \biguplus_{n \in \Naturals} \cball{x_{n}}{\delta} \right) =
			\sum_{n \in \Naturals} \crcdf{x_{n}}{\delta} \geq
			\sum_{n \in \Naturals} f(\delta, R) = \infty .
		\end{equation*}
		Since $\delta > 0$ was arbitrary, $\gamma(\cball{o}{R}) = 0$, i.e.\ $\cball{o}{R}$ is compact.
		Now, given any closed and bounded set $A \subseteq X$, choose $R > 0$ large enough that $A \subseteq \cball{o}{R}$ to see that $A$ must be compact.
		Therefore, $X$ has the Heine--Borel property.
		The claim now follows from \ref{item:r_greatest_HB}.

		\item
		The claim follows from Kuratowski's intersection theorem (\Cref{thm:intersection_theorem}\ref{item:intersection_theorem_Kuratowski}).

		\item
		As already observed, by \Cref{lem:upper_closure_is_closed_and_bounded}, each upward closure $\upclosure_{r} x_{n}$ is both closed and bounded in the metric topology and they are decreasingly nested.
		Since $X$ is complete, Cantor's intersection theorem (\Cref{thm:intersection_theorem}\ref{item:intersection_theorem_Cantor_2}) yields that $\mathfrak{M}_{r} = \bigcap_{n \in \Naturals} \upclosure_{r} x_{n} = \{ x_{r}^{\star} \}$ for some $x_{r}^{\star} \in X$.

		\item
		This is simply \Cref{lem:r_greatest_T_compact} in the special case that $\mathcal{T}$ is the weak topology of the separable Banach space $X$.

		\item
		Each closed, bounded, and convex subset of the separable, reflexive Banach space $X$ is necessarily weakly compact, and so the claim follows from \ref{item:r_greatest_weakly_compact}.
		\qedhere
	\end{enumerate}
\end{proof}

\Cref{thm:r_greatest} is by no means universally applicable, and indeed there are measures that have no radius-$r$ modes, as the next two examples show.

\begin{example}[An atomic measure with no radius-$r$ mode for $1 \leq r < 2$]
	\label{eg:no_radius_1_mode}
	Let $X = \Naturals$ be equipped with the following variant of the discrete metric:
	\begin{equation}
		\label{eq:funny_discrete}
		\dd(k, \ell) \defeq
		\begin{cases}
			0 , & \text{if $k = \ell$,} \\
			2 , & \text{if $\min \{ k, \ell \}$ is odd and $\max \{ k, \ell \} = \min \{ k, \ell \} + 1$,} \\
			1 , & \text{otherwise.}
		\end{cases}
	\end{equation}
	In the space $(X, \dd)$, distinct points are a unit distance apart, with the exception of each odd number and its successor, which are doubly spaced.
	Equip this space with the measure $\mu \defeq \sum_{k \in \Naturals} 2^{-k} \delta_{k} \in \prob{X}$, where $\delta_{k}$ is the unit Dirac measure centred at $k$.
	For arbitrary $k \in \Naturals$,
	\begin{align}
		\label{eg:no_radius_1_mode_odd}
		\crcdf{2 k - 1}{1} & = \mu(X \setminus \{ 2 k \}) = 1 - 2^{- 2 k} , \\
		\label{eg:no_radius_1_mode_even}
		\crcdf{2 k}{1} & = \mu(X \setminus \{ 2 k - 1 \}) = 1 - 2^{- (2 k - 1)} .
	\end{align}
	Both \eqref{eg:no_radius_1_mode_odd} and \eqref{eg:no_radius_1_mode_even} show that $M_{1} = 1$;
	\eqref{eg:no_radius_1_mode_odd} shows that no odd number is a radius-$1$ mode;
	\eqref{eg:no_radius_1_mode_even} shows that no even number is a radius-$1$ mode.
	Thus, $\mu$ has no radius-$1$ mode at all.

	Similar arguments also show that $\mu$ has no radius-$r$ mode for $1 \leq r < 2$;
	for $r \geq 2$, every point of $X$ is a radius-$r$ mode;
	for $0 < r < 1$, the point $1 \in X$ is the unique radius-$r$ mode.

	It is interesting to relate this example to \Cref{thm:r_greatest}.
	In this setting, for each $k \in X$, $\upclosure_{1} k \supseteq \{ k, k + 2, k + 4, \dots \}$.
	This set is non-compact with $\gamma(\upclosure_{1} k) \geq 1$, since it contains an infinite $1$-separated sequence.
	Thus, neither \Cref{thm:r_greatest}\ref{item:r_greatest_HB} nor \ref{item:r_greatest_vanishing_MNC} can apply.
	Also, although the space $(X, \dd)$ is complete,\footnote{Just as in the case of the discrete metric, in this space, the properties of being a Cauchy sequence, being a convergent sequence, and being eventually constant all coincide.} \Cref{thm:r_greatest}\ref{item:r_greatest_vanishing_diameter} does not apply because $\diam(\upclosure_{1} k) \geq 1$.
\end{example}

\begin{example}[A non-atomic measure with no radius-$r$ mode for any $0 < r < \nicefrac{1}{8}$.]
	\label{eg:no_radius_r_mode}
	Building on the ideas of \Cref{eg:no_radius_1_mode}, consider the space
	\begin{align}
		\label{eq:no_radius_r_mode_X}
		X & \defeq \Set{ (\xi, k, m) \in \Reals \times \Naturals^{2} }{ \absval{ \xi } \leq 2^{- k - m - 1} } ,
	\end{align}
	equipped with the metric $d$ and probability measure $\mu$ given by
	\begin{align}
		\label{eq:no_radius_r_mode_d}
		d \bigl( (\xi, k, m) , (\eta, \ell, n) \bigr) & \defeq
		\begin{cases}
			2 , & \text{if $m \neq n$,} \\
			2^{-m} \dd (k, \ell) , & \text{if $m = n$ and $k \neq \ell$,} \\
			\absval{ \xi - \eta } , & \text{if $m = n$ and $k = \ell$,}
		\end{cases} \\
		\label{eq:no_radius_r_mode_mu}
		\mu \left( \biguplus_{k, m \in \Naturals} \bigl( E_{k, m} \times \{ k \} \times \{ m \} \bigr) \right) & \defeq \frac{1}{Z} \sum_{k, m \in \Naturals} \sigma^{-m} \lambda^{1} \bigl( E_{k, m} \cap [-2^{- k - m - 1}, 2^{- k - m - 1}] \bigr),
	\end{align}
	for $E_{k, m} \in \Borel{\Reals}$, where $\lambda^{1}$ is one-dimensional Lebesgue measure, $\nicefrac{1}{2} < \sigma < 1$ is a scaling parameter, and the normalisation constant is $Z \defeq \sum_{m \in \Naturals} (2 \sigma)^{-m} \in (1, \infty)$.

	Now let $0 < r < \nicefrac{1}{8}$ be arbitrary and let $n \in \Naturals$ be uniquely determined by $2^{-n} \leq r < 2^{-n + 1}$.
	We now determine $M_{r}$ and whether or not it can be attained by the masses of balls $\cball{x}{r}$, where $x = (\xi, k, m)$ has $m = n$, $m < n$, or $m > n$ respectively.
	Note that, since $r < 2$, the first case of \eqref{eq:no_radius_r_mode_d} implies that $\cball{x}{r} \subseteq \Reals \times \Naturals \times \{ m \}$.
	\begin{enumerate}[label=(\roman*)]
		\item
		First suppose that $m = n$.
		For odd $k \in \Naturals$,
		\begin{equation*}
			\mu(\cball{x}{r}) = \mu(\Reals \times (\Naturals \setminus \{ k + 1 \}) \times \{ m \}) = \frac{(2\sigma)^{-m}}{Z} (1 - 2^{-(k + 1)}) .
		\end{equation*}
		Taking the limit as $k \to \infty$ shows that $M_{r} \geq \frac{(2\sigma)^{-n}}{Z}$ but that no such ball realises this supremal mass.
		The case of even $k$ is similar, just as in \Cref{eg:no_radius_1_mode}.

		\item
		For $m > n$, since $\cball{x}{r} \subseteq \Reals \times \Naturals \times \{ m \}$, it follows that $x$ is not a radius-$r$ mode because
		\begin{equation*}
				\mu(\cball{x}{r}) \leq \frac{(2 \sigma)^{-m}}{Z} < \frac{(2\sigma)^{-n}}{Z} \leq M_r.
		\end{equation*}

		\item
		If $m < n$, then $r < 2^{-n} < 2^{-m}$, and so the second case of \eqref{eq:no_radius_r_mode_d} ensures that $\cball{x}{r} \subseteq \Reals \times \{ k \} \times \{ m \}$.
		The mass of such a ball is maximised by the case $\xi = 0$, $k = 1$, $m = n - 1$, in which case the ball (which is a single line segment) has mass
		\begin{align*}
			\crcdf{x}{r}
			& = \frac{\sigma^{-m}}{Z} \lambda^{1} \bigl( [-r, r] \cap [ -2^{-k - m - 1}, 2^{-k - m - 1} ] \bigr) \\
			& \leq \sigma \frac{\sigma^{-n}}{Z} \lambda^{1} \bigl( [-2^{-n}, 2^{-n}] \cap [ -2^{-n - 1}, 2^{-n - 1} ] \bigr) \\
			& = \sigma \frac{(2\sigma)^{-n}}{Z}
			< \frac{(2\sigma)^{-n}}{Z},
		\end{align*}
		where the last inequality follows from the fact that $\sigma < 1$.
	\end{enumerate}
	Hence, $M_{r} = \frac{(2 \sigma)^{-n}}{Z}$ but $\crcdf{x}{r} < M_{r}$ for all $x \in X$, i.e.\ $\mu$ has no radius-$r$ mode.
\end{example}

Thus, while \Cref{thm:r_greatest} on the existence of radius-$r$ modes covers a variety of well-behaved spaces and measures, \Cref{eg:no_radius_1_mode,eg:no_radius_r_mode} show that existence cannot be guaranteed for general spaces and measures.

Before moving on, we mention one interesting intermediate case, motivated by applications to inverse problems, namely countable product measures on weighted $\ell^{p}$ spaces (and isometric linear images of such spaces).
This is a broad class that includes Gaussian, Besov \citep{DashtiHarrisStuart2012,LassasSaksmanSiltanen2009}, and Cauchy measures \citep{Sullivan2017}.
It turns out that reweightings of such measures always have radius-$r$ modes, i.e.\ Bayesian posteriors with such measures as priors always have radius-$r$ MAP estimators.
We defer the precise statements to \Cref{thm:lp_alpha_radius_r_modes,cor:Gaussian_Besov_Cauchy_radius_r_modes} in \Cref{sec:radius_r_modes_sequence_spaces} because they do not have a particularly order-theoretic flavour.

Also, while we do prove the existence of radius-$r$ modes, we do not consider taking limits as $r \to 0$ to obtain true MAP estimators for such posteriors.
The main difficulty here lies in proving that such a family $(x_{r}^{\star})_{r > 0}$ is bounded, so that a weakly convergent subsequence can be extracted;
this is not true in general, so one must argue using properties of the prior and likelihood (e.g.\ when the prior is Gaussian or Besov).

Indeed, the whole question of taking limits of radius-$r$ modes is a sensitive one, and is the topic of the next section.

\subsection{Convergence of greatest and near-greatest elements}
\label{subsection:convergence_of_radius_r}

If radius-$r$ modes $x_{r}^{\star}$ do exist for each $r > 0$, it is then natural to ask whether sequences of radius-$r$ modes can approximate true modes, e.g.\ strong or weak modes.
This approach is used by \citet[Theorem~3.5]{DashtiLawStuartVoss2013} to obtain strong modes for Bayesian posteriors arising from Gaussian priors.

However, we have seen that existence of radius-$r$ modes can be difficult to prove, and in some cases no radius-$r$ modes exist (\Cref{eg:no_radius_1_mode,eg:no_radius_r_mode}).
Taking limits of radius-$r$ modes is also problematic for more general measures: the limit need not be a strong or weak mode, and not every mode can be represented as the limit of radius-$r$ modes.
Thus, one cannot hope to use the approach of taking limits of radius-$r$ modes to find true modes if there is no correspondence between modes and limits of radius-$r$ modes.
Nevertheless, we show that some of the difficulties can be overcome using \defterm{asymptotic maximising families} (AMFs) as proposed by \citet{KlebanovWacker2022}.

To illustrate the problem and motivate the introduction of AMFs, we first give an example of a measure with a bounded and continuous Lebesgue density possessing a mode that cannot be represented as the limit of radius-$r$ modes. 
The problem here is that balls around the points ${\pm 1}$ have asymptotically equivalent mass, but each ball around ${+1}$ has slightly more mass than the corresponding ball around ${-1}$; as a result, ${+1}$ ``hides'' the other mode ${-1}$.

\begin{figure}[t]
	\centering
	\begin{subfigure}{0.49\textwidth}
		\centering
		\includegraphics[width=\textwidth]{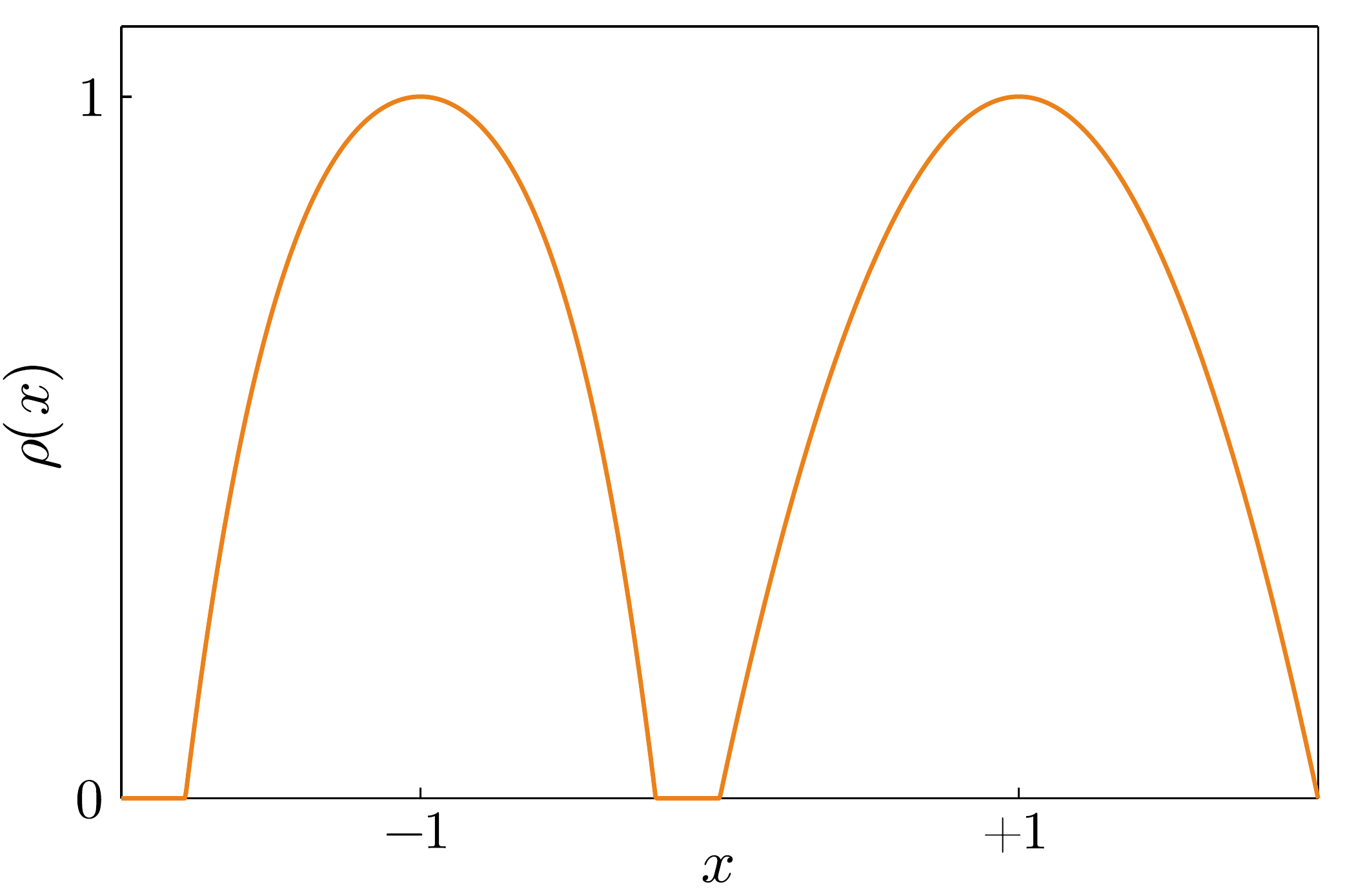}
		\subcaption{\raggedright Unnormalised density \eqref{eq:modes_limit_of_r-modes_density} of the measure in \Cref{eg:modes_limit_of_r-modes}.
		The point $+1$ is the unique radius-$r$ mode for all sufficiently small $r$.}
	\end{subfigure}
	\begin{subfigure}{0.49\textwidth}
		\centering
		\includegraphics[width=\textwidth]{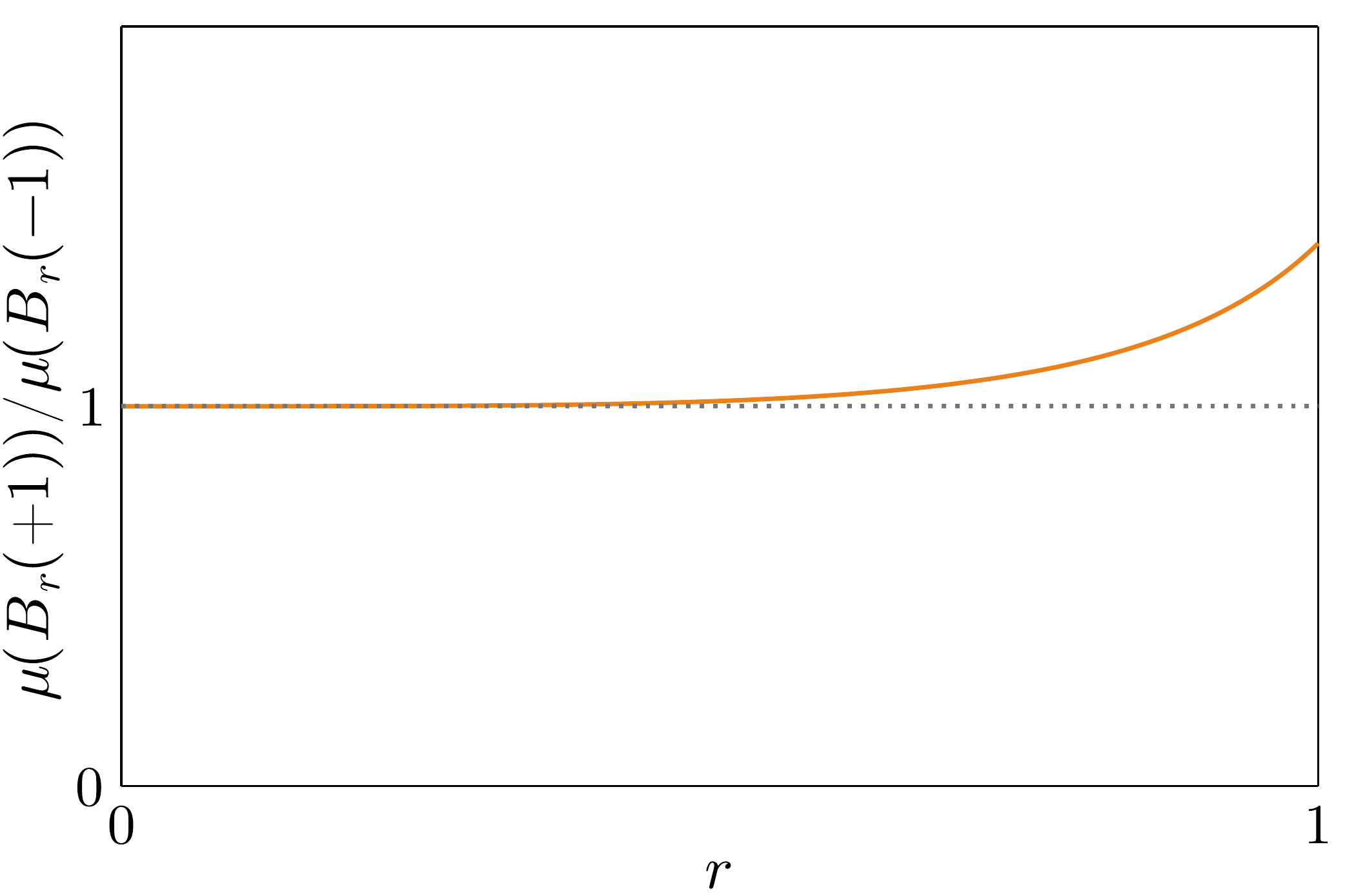}
		\subcaption{\raggedright The ratio $\nicefrac{\crcdf{+1}{r}}{\crcdf{-1}{r}}$
		converges to $1$ as $r \to 0$, but it is strictly greater than $1$ for any $r > 0$.}
	\end{subfigure}
	\caption{Not every strong mode is the limit of a sequence of radius-$r$ modes: in \Cref{eg:modes_limit_of_r-modes}, $-1$ is a strong mode but $+1$ is the unique radius-$r$ mode for all small $r$.}
	\label{fig:modes_limit_of_r-modes}
\end{figure}

\begin{example}
	\label{eg:modes_limit_of_r-modes}
	Define $\mu \in \prob{\Reals}$ by the Lebesgue density as shown in \Cref{fig:modes_limit_of_r-modes}, given by
	\begin{equation}
		\label{eq:modes_limit_of_r-modes_density}
		\rho(x) \propto \max \{ 1 - (x - 1)^2, 0 \} + \max \{ 1 - (x + 1)^2 - (x + 1)^4, 0 \}.
	\end{equation}
	When $r$ is sufficiently small, $\crcdf{+1}{r} = 2r - \frac{2}{3} r^3$ and
	$\crcdf{-1}{r} = 2r - \frac{2}{3} r^3 - \frac{2}{5} r^5$, so there is a unique
	radius-$r$ mode at $+1$.
	However, $+1$ and $-1$ are both strong modes: $+1$ is a radius-$r$ mode for all sufficiently small $r$, so it is a strong mode (\Cref{thm:limits_of_r-modes}) and $-1$ is a strong mode because
	\begin{equation*}
		\lim_{r \to 0} \frac{\crcdf{-1}{r}}{M_r} = \lim_{r \to 0} \frac{\crcdf{+1}{r}}{M_r}
		\lim_{r \to 0} \frac{\crcdf{-1}{r}}{\crcdf{+1}{r}} = 1.
	\end{equation*}
\end{example}

Instead of representing modes as limits of radius-$r$ modes --- which might not be possible --- one may consider families of points that are nearly greatest, which always exist, and try to take limits of such families.

\begin{definition}
	\label{defn:AMF}
	Let $X$ be a metric space and let $\mu \in \prob{X}$. 
	A net $(x_r)_{r > 0} \subseteq X$ is an \defterm{asymptotic maximising family} (AMF) if there exists a positive function $\varepsilon$ with $\lim_{r \to 0} \varepsilon(r) = 0$ and
	\begin{equation}
		\label{eq:AMF}
		\frac{\crcdf{x_r}{r}}{M_r} \geq 1 - \varepsilon(r)
		\quad
		\text{for all $r > 0$.}
	\end{equation}
\end{definition}

Note that every measure admits an AMF satisfying \eqref{eq:AMF}, even if the function $\varepsilon$ is specified in advance, which is sometimes advantageous.

The following results shed light on the subtleties involved in taking limits of radius-$r$ modes, or, more generally, AMFs.

\begin{theorem}
	\label{thm:limits_of_r-modes}
	Let $X$ be a separable metric space and let $\mu \in \prob{X}$.
	\begin{enumerate}[label=(\alph*)]
		\item
		\label{item:limits_of_r-modes_strong}
		If $(x^{\star})_{r > 0}$ is an AMF with $x^{\star} \in X$ fixed, then $x^{\star}$ is a strong mode.
		
		\item
		If $x^{\star}$ is a radius-$r$ mode for all small enough $r > 0$, then $x^{\star}$ is a strong mode.

		\item
		\label{item:limits_of_r-modes_gen_strong}
		If the AMF $(x_{r}^{\star})_{r > 0}$ converges to $x^{\star}$, then $x^{\star}$ is a generalised strong mode.

	\end{enumerate}
\end{theorem}

\begin{proof}
	\begin{enumerate}[label=(\alph*)]
		\item 
		As $\crcdf{x^{\star}}{r} \geq (1 - \varepsilon(r)) M_r$, it is immediate that $x^{\star}$ is a strong mode, because
		\begin{equation*}
			1\geq \lim_{r \to 0} \frac{\crcdf{x^{\star}}{r}}{M_r} \geq \lim_{r \to 0} ( 1 - \varepsilon(r) ) = 1 .
		\end{equation*}

		\item 
		This is immediate from \ref{item:limits_of_r-modes_strong} as $(x^{\star})_{r > 0}$ forms an AMF.

		\item 
		This is precisely \citet[Lemma~2.4]{Clason2019GeneralizedModes}.
		\qedhere

	\end{enumerate}
\end{proof}

The claim in \ref{item:limits_of_r-modes_gen_strong} --- which requires that the net $(x_{r}^{\star})_{r > 0}$ converges to $x^{\star}$ along every subsequence --- cannot be made stronger without additional hypotheses:
the limit $x^{\star}$ need not be a strong or weak mode (as can be seen in \Cref{eg:upward_closures_wrt_preceq_0}\ref{item:upward_closures_wrt_preceq_0_1}, for which $x^\star = 1$ is the limit of an AMF which is neither a strong mode nor a weak mode).
Furthermore, one cannot weaken the hypotheses of \ref{item:limits_of_r-modes_gen_strong} further:
the points $\pm 1$ in \Cref{thm:oscillation_example} are limit points of an AMF but they are not even generalised modes. 

The general question of classifying measures $\mu$ for which limits of radius-$r$ modes are strong modes is still open, although \citet{DashtiLawStuartVoss2013} show that reweightings of Gaussian measures on Hilbert spaces enjoy this property, and \cite{KlebanovWacker2022} show the same for some Gaussian measures on sequence spaces.

Under an additional nesting assumption, intersection arguments can be applied to AMFs to yield the existence of several kinds of modes.

\begin{theorem}[AMFs and strong modes]
	\label{thm:AMFs_and_strong_modes}
	Let $X$ be a complete and separable metric space and let $\mu \in \prob{X}$.
	Let $(x_{r})_{r > 0}$ be any AMF, i.e.\ any net satisfying \eqref{eq:AMF}, and let $I \defeq \bigcap_{r > 0} \upclosure_{r} x_{r}$.
	Then
	\begin{enumerate}[label=(\alph*)]
		\item
		\label{thm:AMFs_and_strong_modes_supp}
		$I \subseteq \supp(\mu)$;
		
		\item
		\label{thm:AMFs_and_strong_modes_strong_mode}
		every $x^{\star} \in I$ is a strong (and hence weak and generalised strong) mode for $\mu$;
		
		\item
		\label{thm:AMFs_and_strong_modes_nonemptiness}
		and if also
		\begin{equation}
			\label{eq:nesting_for_GWM}
			0 < r \leq s \implies \upclosure_{r} x_{r} \subseteq \upclosure_{s} x_{s},
		\end{equation}
		then $I$ is non-empty and compact.
	\end{enumerate}
\end{theorem}

\begin{proof}
	Let $x^{\star} \in I$.
	For all sufficiently small $r > 0$, it follows that $\crcdf{x^{\star}}{r} \geq \crcdf{x_{r}}{r} \geq M_{r} ( 1 - \varepsilon(r) ) > 0$, and so $x^{\star} \in \supp(\mu)$, which establishes \ref{thm:AMFs_and_strong_modes_supp}.
	Furthermore, since $x^{\star} \in \upclosure_{r} x_{r}$ for each $r$,
	\[
		1 \geq \frac{\crcdf{x^{\star}}{r}}{M_{r}} \geq \frac{\crcdf{x_{r}}{r}}{M_{r}} \geq 1 - \varepsilon(r).
	\]
	Taking the limit as $r \to 0$ throughout shows that $x^{\star}$ is a strong mode (and hence also a weak and generalised strong mode) for $\mu$, establishing \ref{thm:AMFs_and_strong_modes_strong_mode}.
	(Alternatively, one may observe that $(x^{\star})_{r > 0}$ is a constant AMF and appeal to  \Cref{thm:limits_of_r-modes}\ref{item:limits_of_r-modes_strong}.)

	For \ref{thm:AMFs_and_strong_modes_nonemptiness}, let $(r_{n})_{n \in \Naturals}$ be some null sequence of radii.
	The nesting hypothesis \eqref{eq:nesting_for_GWM} implies that $I = \bigcap_{n \in \Naturals} \upclosure_{r_{n}} x_{r_{n}}$.
	For each $n$, $\upclosure_{r_{n}} x_{r_{n}}$ is non-empty and, by \Cref{lem:upper_closure_is_closed_and_bounded}, is closed and bounded with $\gamma(\upclosure_{r_{n}} x_{r_{n}}) \leq 2 r_{n}$.
	This, together with the nesting hypothesis \eqref{eq:nesting_for_GWM} and Kuratowski's intersection theorem (\Cref{thm:intersection_theorem}\ref{item:intersection_theorem_Kuratowski}), ensures that $I$ is non-empty and compact.
\end{proof}

\begin{remark}
	\begin{enumerate}[label=(\alph*)]
		\item
		The nesting hypothesis \eqref{eq:nesting_for_GWM}, in conjunction with \Cref{lem:upper_closure_is_closed_and_bounded}, ensures that the AMF $(x_{r})_{r > 0}$ --- and indeed any family of greatest elements $(x_{r}^{\star})_{r > 0}$ --- must be bounded.
		This means that \Cref{thm:AMFs_and_strong_modes} does not apply to measures such as \Cref{eg:upward_closures_wrt_preceq_0}\ref{item:upward_closures_wrt_preceq_0_2}, for which the radius-$r$ modes ``escape to infinity'' as $r \to 0$.
		Hypothesis \eqref{eq:nesting_for_GWM} also fails for measures displaying oscillatory behaviour of the kind discussed in \Cref{thm:oscillation_example}.

		\item
		\Cref{thm:AMFs_and_strong_modes} is not sharp, in the sense that there can exist modes $x^{\star} \notin \bigcap_{r > 0} \upclosure_{r} x_{r}$.
		See \Cref{eg:modes_limit_of_r-modes} for an example of this situation with
		\begin{equation*}
			x_{r} \equiv {+1} ,
			\quad
			\upclosure_{r} x_{r} = \{ +1 \} ,
			\quad
			x^{\star} = {-1} \notin \bigcap_{r > 0} \upclosure_{r} x_{r} .
		\end{equation*}
	\end{enumerate}
\end{remark}

\section{Preorders in the small-radius limit}
\label{sec:limiting_preorders}

One would like to think of $x^{\star} \in X$ as a mode of $\mu \in \prob{X}$ if $x^{\star}$ is a greatest or maximal element of $X$ with respect to the preorder $\preceq_{r}$ ``in the limit as $r \to 0$'' in some sense.
However, is such a limiting preorder well defined?
Must this preorder have greatest or maximal elements?

In fact, there are several candidates for a small-radius limiting preorder and it appears that each of them has at least one undesirable feature.
This work will focus on the analytic small-radius limiting preorder $\preceq_{0}$, to be defined shortly (\Cref{defn:analytic_small-radius_preorder}).
This preorder has the advantage that its greatest elements are weak modes;
however, it has the disadvantage that it is not total, i.e.\ the existence of greatest elements is not guaranteed, and indeed the collection of incomparable elements may be rather large.
We claim that this is a small price to pay:
we show in \Cref{sec:alternative_small-radius_preorders} that the alternative definitions are even more ill behaved.

\subsection{Definition and basic properties}

\begin{definition}[Small-radius limiting preorder]
	\label{defn:analytic_small-radius_preorder}
	Let $X$ be a metric space and let $\mu \in \prob{X}$.
	Define a preorder $\preceq_{0}$ on $X$ by
	\begin{align}
		\label{eq:preceq_0}
		x \preceq_{0} x'
		& \iff \limsup_{r \to 0} \frac{\crcdf{x}{r}}{\crcdf{x'}{r}} \leq 1
		\iff \liminf_{r \to 0} \frac{\crcdf{x'}{r}}{\crcdf{x}{r}} \geq 1 ,
	\end{align}
	if both $x, x' \in \supp(\mu)$.
	Additionally, as exceptional cases, $x \preceq_{0} x'$ is defined to be false for $x \in \supp(\mu)$ and $x' \notin \supp(\mu)$, and $x \preceq_{0} x'$ is defined to be true for $x \notin \supp(\mu)$ and $x' \in X$.
\end{definition}

It is relatively straightforward to verify that $\preceq_{0}$, as defined above, is a preorder on $X$;
the only subtleties are correct handling of points outside the support, and the use of the upper bound (but not equality)
\begin{equation}
		\label{eq:limsup_product_rule}
		\limsup_{r \to 0} \frac{\crcdf{x}{r}}{\crcdf{y}{r}} \frac{\crcdf{y}{r}}{\crcdf{z}{r}}
		\leq
		\limsup_{r \to 0} \frac{\crcdf{x}{r}}{\crcdf{y}{r}} \limsup_{r \to 0} \frac{\crcdf{y}{r}}{\crcdf{z}{r}}
\end{equation}
when verifying transitivity.
As usual, we will write $x \asymp_{0} x'$ if both $x \preceq_{0} x'$ and $x \succeq_{0} x'$ hold, and $x \incomp_{0} x'$ if neither $x \preceq_{0} x'$ nor $x \succeq_{0} x'$ hold.

The appeal of the preorder $\preceq_{0}$ is that its greatest elements are exactly the weak modes of $\mu$, as defined in \eqref{eq:weak_mode}, as the next two results show.

\begin{lemma}[Properties of $\preceq_0$-maximal elements]
	\label{lem:properties_of_maximal}
	Let $X$ be separable and let $\mu \in \prob{X}$.
	\begin{enumerate}[label=(\alph*)]
			\item
			\label{item:0_maximal_implies_in_supp}
			If $x^{\star}$ is $\preceq_{0}$-maximal, then $x^{\star} \in \supp(\mu)$.
			\item
			\label{item:0_maximal_limits}
			The point $x^{\star} \in \supp(\mu)$ is $\preceq_{0}$-maximal if and only if any $x \in X$ satisfies either
			\begin{equation} \label{eq:liminf_of_small-ball_ratio}
						\liminf_{r \to 0} \frac{\crcdf{x}{r}}{\crcdf{x^{\star}}{r}} < 1
						\text{ or } \lim_{r \to 0} \frac{\crcdf{x}{r}}{\crcdf{x^{\star}}{r}} = 1.
			\end{equation}
	\end{enumerate}
\end{lemma}

\begin{proof}
	\begin{enumerate}[label=(\alph*)]
			\item
			Suppose that $x^{\star}$ is maximal but, for a contradiction, suppose also that $x^{\star} \notin \supp(\mu)$.
			As $X$ is separable, take $x \in \supp(\mu) \neq \varnothing$.
			By the exceptional cases of \Cref{defn:analytic_small-radius_preorder}, $x^{\star} \prec_0 x$, contradicting the assumption that $x^{\star}$ is maximal.

			\item
			Let $x, x' \in \supp(\mu)$. 
			It is straightforward to verify from the definitions that
			\begin{align}
					\label{eq:dominant_analytic_order}
					x \prec_0 x' &\iff \liminf_{r \to 0} \frac{\crcdf{x}{r}}{\crcdf{x'}{r}} < 1, \\
					\label{eq:equiv_analytic_order}
					x \asymp_0 x' &\iff \lim_{r \to 0} \frac{\crcdf{x}{r}}{\crcdf{x'}{r}} = 1.
			\end{align}
			Suppose first that $x^{\star}$ is maximal, and let $x \in X$ be arbitrary.
			If $x \notin \supp(\mu)$, then $\crcdf{x}{r} = 0$ for all sufficiently small $r$, so \eqref{eq:liminf_of_small-ball_ratio} holds.
			If $x \in \supp(\mu)$, then maximality of $x^{\star}$ implies that either $x \prec_0 x^{\star}$ or $x \asymp_0 x^{\star}$, from which \eqref{eq:liminf_of_small-ball_ratio} follows.

			Conversely, suppose that $x \in X$ satisfies $x^{\star} \preceq_0 x$.
			The exceptional cases in \Cref{defn:analytic_small-radius_preorder} imply that $x \in \supp(\mu)$.
			Hence, by \eqref{eq:equiv_analytic_order}, $x^{\star} \asymp_0 x$, proving that $x^{\star}$ is $\preceq_0$-maximal.
			\qedhere
	\end{enumerate}
\end{proof}

\begin{lemma}[Characterisation of weak modes]
	\label{lem:GWM_greatest_maximal}
	Let $X$ be separable and let $\mu \in \prob{X}$.
	Then the following are equivalent:
	\begin{enumerate}[label=(\alph*)]
		\item
		\label{item:GWM_greatest_maximal_GWM}
		$x^{\star} \in X$ is a weak mode for $\mu$;

		\item
		\label{item:GWM_greatest_maximal_greatest}
		$x^{\star} \in X$ is a $\preceq_{0}$-greatest element;

		\item
		\label{item:GWM_greatest_maximal_maximal}
		$x^{\star} \in X$ is a $\preceq_{0}$-maximal element that is comparable with every other $x' \in X$.
	\end{enumerate}
\end{lemma}

\begin{proof}
	(\ref{item:GWM_greatest_maximal_GWM}$\implies$\ref{item:GWM_greatest_maximal_greatest})$\quad$
	Suppose that $x^{\star}$ is a weak mode for $\mu$.
	Then, by definition (see \eqref{eq:weak_mode}), it follows that $x \preceq_{0} x^{\star}$ for each $x \in \supp(\mu)$.
	As $x^{\star} \in \supp(\mu)$, any point $x \notin \supp(\mu)$ satisfies
	$x \preceq_{0} x^{\star}$ by the special cases in the definition of $\preceq_{0}$.

	\noindent(\ref{item:GWM_greatest_maximal_greatest}$\implies$\ref{item:GWM_greatest_maximal_GWM})$\quad$
	Suppose that $x^{\star}$ is a $\preceq_{0}$-greatest element.
	Then $x^{\star} \in \supp(\mu)$ by \Cref{lem:properties_of_maximal}.
	Hence, for $x' \in \supp(\mu)$, \eqref{eq:weak_mode} holds because $x' \preceq_{0} x^{\star}$.
	For	$x' \notin \supp(\mu)$, we obtain
	\begin{equation*}
		\frac{\crcdf{x'}{r}}{\crcdf{x^{\star}}{r}} = 0 \text{ for sufficiently small } r,
	\end{equation*}
	proving that $x^{\star}$ is a weak mode.

	\noindent(\ref{item:GWM_greatest_maximal_greatest}$\iff$\ref{item:GWM_greatest_maximal_maximal})$\quad$
	This is obvious, since the defining property of being greatest is exactly the property of being maximal and globally comparable.
\end{proof}

The preorder $\preceq_{0}$ does have some shortcomings.
One is that, in contrast to $\preceq_{r}$ with $r > 0$ (\Cref{lem:upper_closure_is_closed_and_bounded}), upward closures under $\preceq_{0}$ need be neither closed nor bounded.

\begin{example}
	\label{eg:upward_closures_wrt_preceq_0}

	\begin{enumerate}[label=(\alph*)]
		\item 
		\label{item:upward_closures_wrt_preceq_0_1}
		For an example of a non-closed upward closure under $\preceq_{0}$, similar in spirit to the example of \citet{Clason2019GeneralizedModes} mentioned in \Cref{sec:related}, let $\mu \in \prob{\Reals}$ have the Lebesgue density $\rho \colon \Reals \to \Reals$, $\rho(x) \defeq 2 x \one [ 0 \leq x \leq 1 ]$, with $\supp(\mu) = [0, 1]$, and consider $y \in \Reals$.
		If $y < 0$ or $y > 1$, then $y \notin \supp(\mu)$ and $\upclosure_{0} y = \Reals$.
		For $0 \leq y \leq \nicefrac{1}{2}$, $\upclosure_{0} y = [y, 1]$, which is closed.
		However, for $\nicefrac{1}{2} < y < 1$,
		\begin{equation*}
			\lim_{r \to 0} \frac{\crcdf{1}{r}}{\crcdf{y}{r}}
			=
			\lim_{r \to 0} \frac{2 r - r^{2}}{4 y r}
			=
			\frac{1}{2 y}
			<
			1
		\end{equation*}
		and so $\upclosure_{0} y = [y, 1)$, which is not closed.
		Finally, $\upclosure_{0} 1 = [\nicefrac{1}{2}, 1]$, which is closed.

		\item
		\label{item:upward_closures_wrt_preceq_0_2}
		For an example of an unbounded upward closure under $\preceq_{0}$, let $\mu \in \prob{\Reals}$ have the unbounded Lebesgue density $\rho \colon \Reals \to \Reals$,		\begin{equation*}
			\rho(x) \defeq \sum_{n \in \Naturals} n \one \bigl[ n - \tfrac{ 2^{-n - 1} }{ n } \leq x \leq n + \tfrac{ 2^{-n - 1} }{ n } \bigr].
		\end{equation*}
		That is, $\rho$ consists of a sum of disjoint indicator functions centred on the natural numbers $n \in \Naturals$, each having mass $2^{-n}$ and height $n$.
		Then, for any $x, y \in \Naturals$ with $x > y$,
		\begin{equation*}
			\lim_{r \to 0} \frac{\crcdf{x}{r}}{\crcdf{y}{r}}
			=
			\lim_{r \to 0} \frac{2 x r}{2 y r}
			=
			\frac{x}{y}
			>
			1
		\end{equation*}
		and so $\upclosure_{0} y \supseteq \Naturals \cap [y, \infty)$.
	\end{enumerate}
\end{example}

\Cref{eg:upward_closures_wrt_preceq_0} furnishes two examples of measures with no $\preceq_{0}$-maximal element, let alone a $\preceq_{0}$-greatest element (weak mode), or strong mode.
However, these examples are relatively tame: there is no mode simply because any candidate mode $x^{\star}$ is dominated by some other point $x'$.
The \emph{real} shortcoming and subtlety of $\preceq_{0}$ is that it is not total --- that is, the order admits incomparable elements --- and we make this the topic of the next subsection.

\subsection{Criteria for incomparability and comparability}

For $r > 0$, totality of $\preceq_{r}$ followed immediately from \Cref{defn:positive_radius_preorder}.
This is certainly not so obvious for $\preceq_{0}$.
Indeed, what is immediate from \Cref{defn:analytic_small-radius_preorder} is that $\preceq_{0}$-incomparable elements can be characterised as follows:

\begin{lemma}[Incomparability in the limiting preorder]
	\label{lem:incomp_0}
	For $x, x' \in X$,
	\begin{equation}
		\label{eq:incomp_0}
		x \incomp_{0} x' \iff x, x' \in \supp(\mu) \text{ and } \liminf_{r \to 0} \frac{\crcdf{x}{r}}{\crcdf{x'}{r}} < 1 < \limsup_{r \to 0} \frac{\crcdf{x}{r}}{\crcdf{x'}{r}} .
	\end{equation}
\end{lemma}

In the other direction, we can give a (very strong) sufficient condition for two points to be comparable under $\preceq_{0}$:

\begin{lemma}
	Let $X$ be any metric space and let $\mu \in \prob{X}$.
	Suppose that on some interval $(0, r^{\star})$, the function $r \mapsto \nicefrac{\crcdf{x}{r}}{\crcdf{x'}{r}}$ is uniformly continuous for $x$, $x' \in \supp(\mu)$.
	Then $x$ and $x'$ are $\preceq_{0}$-comparable.
\end{lemma}

\begin{proof}
	The ratio function $\nicefrac{\crcdf{x}{r}}{\crcdf{x'}{r}}$ can be uniquely extended to a uniformly continuous function on $[0, r^{\star}]$ \citep[Lemma~3.11]{AliprantisBorder2006}.
	By continuity, the limit of the ratio function as $r \to 0$ must exist;
	the result follows by \Cref{lem:incomp_0}.
\end{proof}

The previous two lemmas hint at a way to construct concrete examples of measures with incomparable points under $\preceq_{0}$:
one must choose the masses around two points such that the ratio of the masses of balls around such points oscillates as $r \to 0$.

We now construct such a measure on $\Reals$ with a Lebesgue density and two $\preceq_{0}$-incomparable maximal points, neither of which is $\preceq_0$-greatest.
(\Cref{sec:dense_antichains} will supply even more extreme and general examples, but it is pedagogically useful to consider a simpler construction first.)
The idea is to construct a density so that the measure $\mu \in \prob{\Reals}$ induced by it has a specific behaviour around the points $x = \pm 1$.
In this case, the density is chosen so that $r \mapsto \crcdf{x}{r}$ piecewise linearly interpolates the function $r \mapsto \sqrt{r}$ through either the interpolation knots $r = a^{-n}$ with $n \in \Naturals$ even or the interpolation knots $r = a^{-n}$ with $n \in \Naturals$ odd, where $a > 1$ is chosen arbitrarily.
It turns out that these mild perturbations of the integrable singularity $\rho(x) \propto \absval{ x }^{-1/2}$ produce ``incomparable modes''.

\begin{figure}[t]
	\centering
	\begin{subfigure}[t]{0.49\textwidth}
		\centering
		\includegraphics[width=\textwidth]{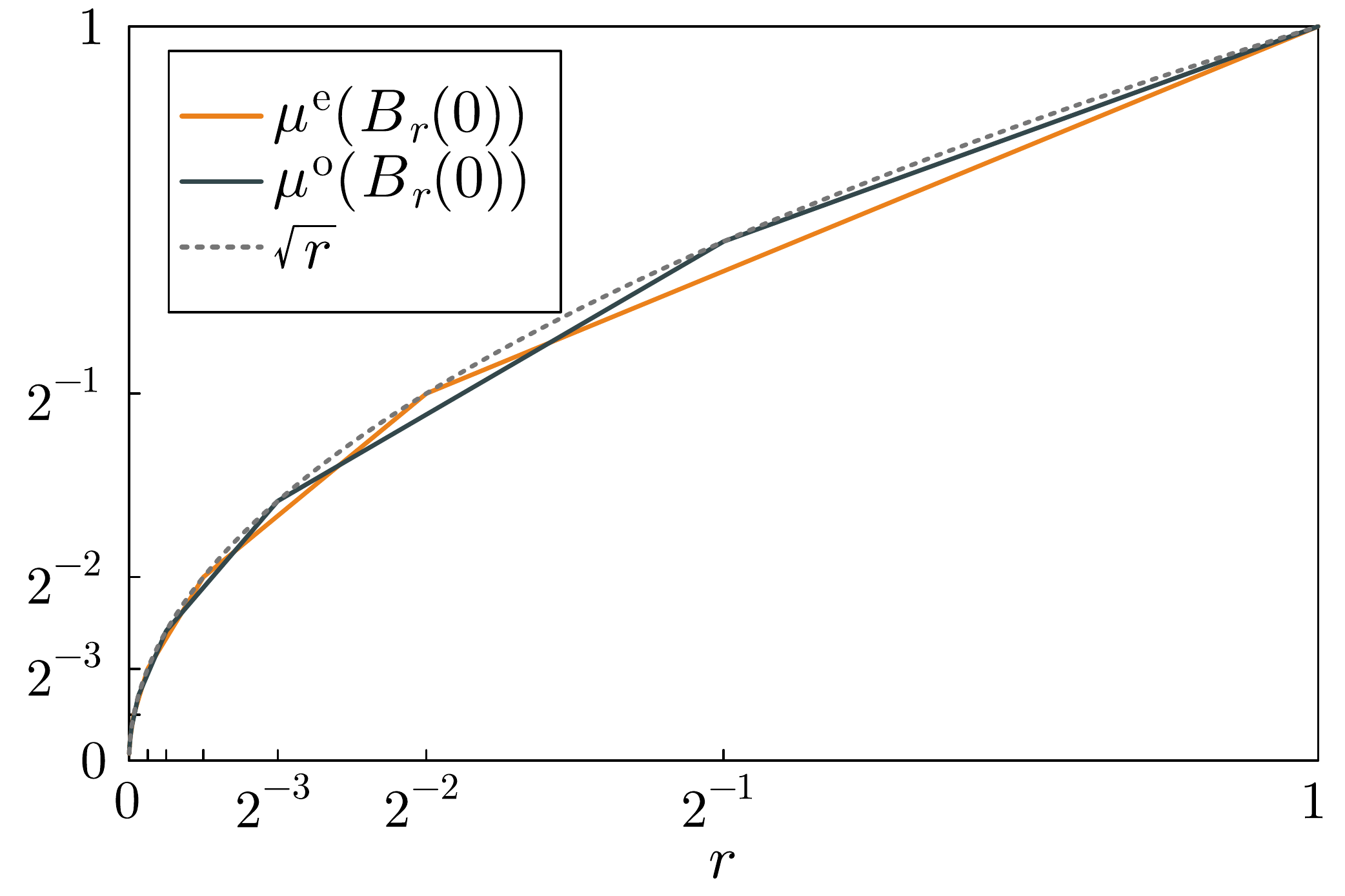}
		\subcaption{\raggedright The RCDFs $\mu^{\even}(\cball{0}{r})$ and $\mu^{\odd}(\cball{0}{r})$ (shown on a linear scale) interpolate between the knots $a^{-n}$ to give mild perturbations of the function $r \mapsto \sqrt{r}$.}
	\end{subfigure}
	\begin{subfigure}[t]{0.49\textwidth}
		\centering
		\includegraphics[width=\textwidth]{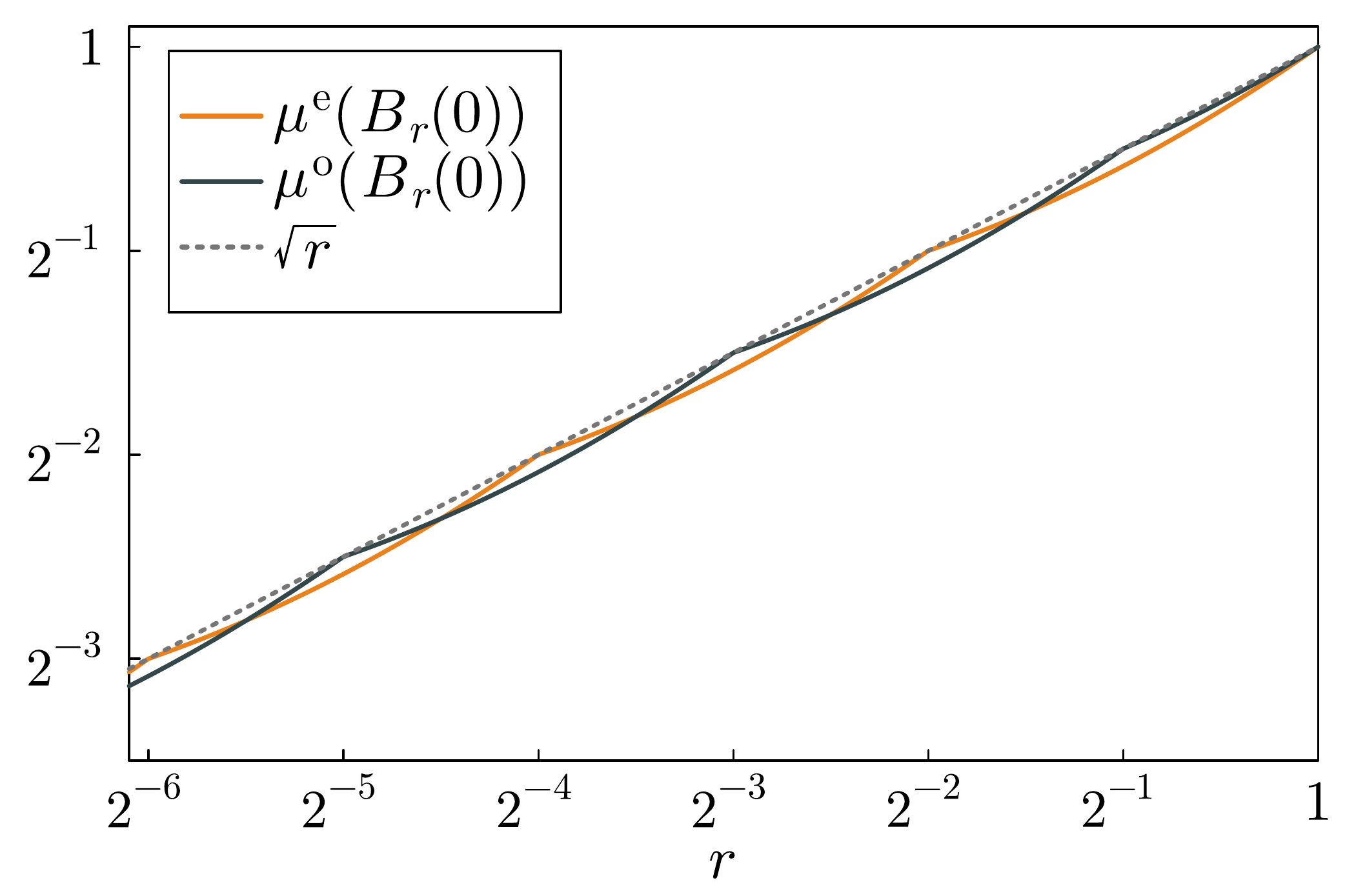}
		\subcaption{\raggedright The RCDFs $\mu^{\even}(\cball{0}{r})$ and $\mu^{\odd}(\cball{0}{r})$ (shown with a logarithmic scale for $r$) agree with the function $r \mapsto \sqrt{r}$ at the knots $a^{-n}$ for even $n$ and odd $n$ respectively.}
	\end{subfigure}
	\begin{subfigure}[t]{0.49\textwidth}
		\centering
		\includegraphics[width=\textwidth]{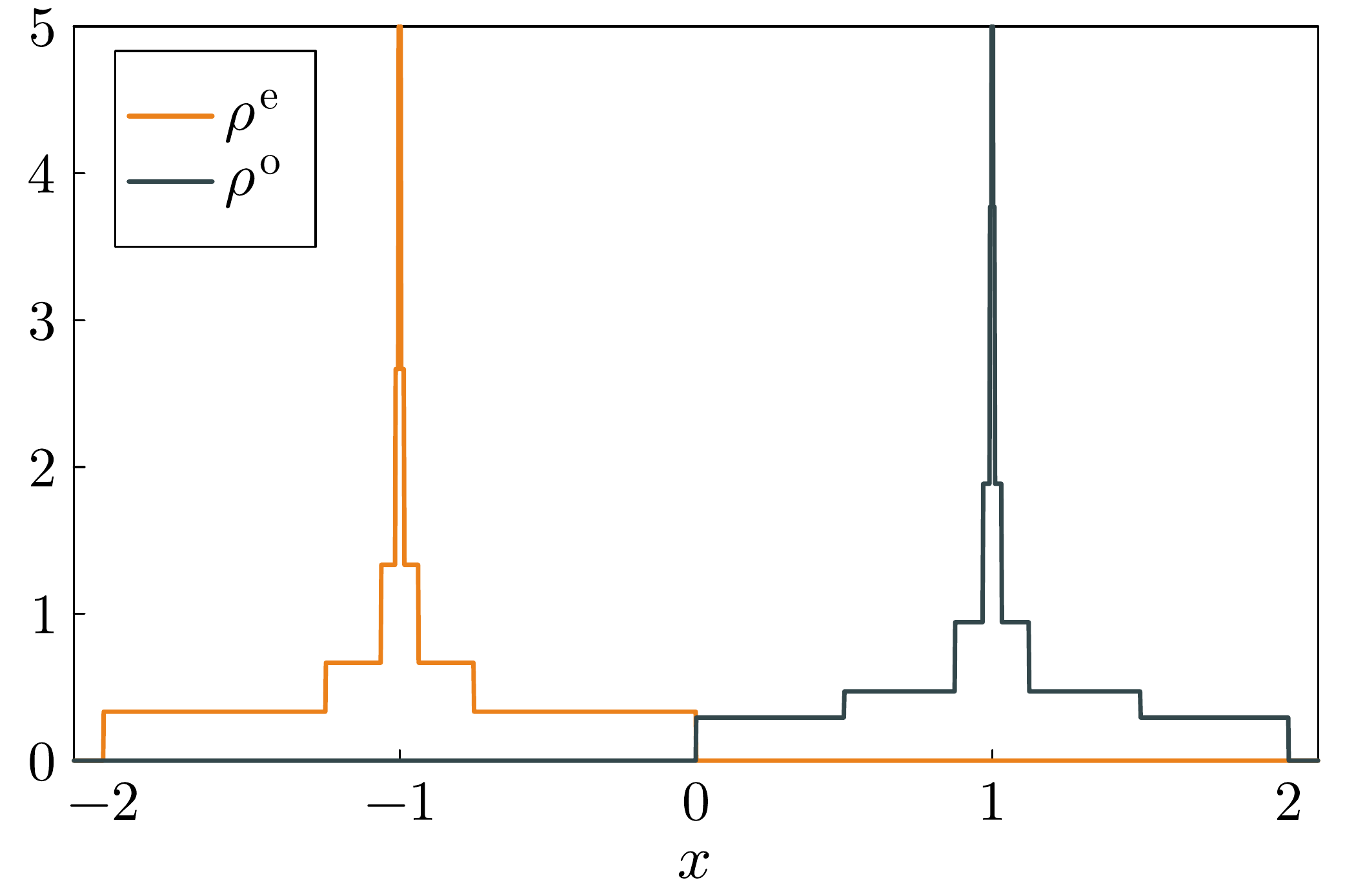}
		\subcaption{\raggedright The probability density functions $\rho^{\even} (\quark + 1)$ and $\rho^{\odd} (\quark - 1)$ have singularities which behave like $\absval{ \quark }^{-\nicefrac{1}{2}}$ at $-1$ and $+1$ respectively.}
	\end{subfigure}
	\begin{subfigure}[t]{0.49\textwidth}
		\centering
		\includegraphics[width=\textwidth]{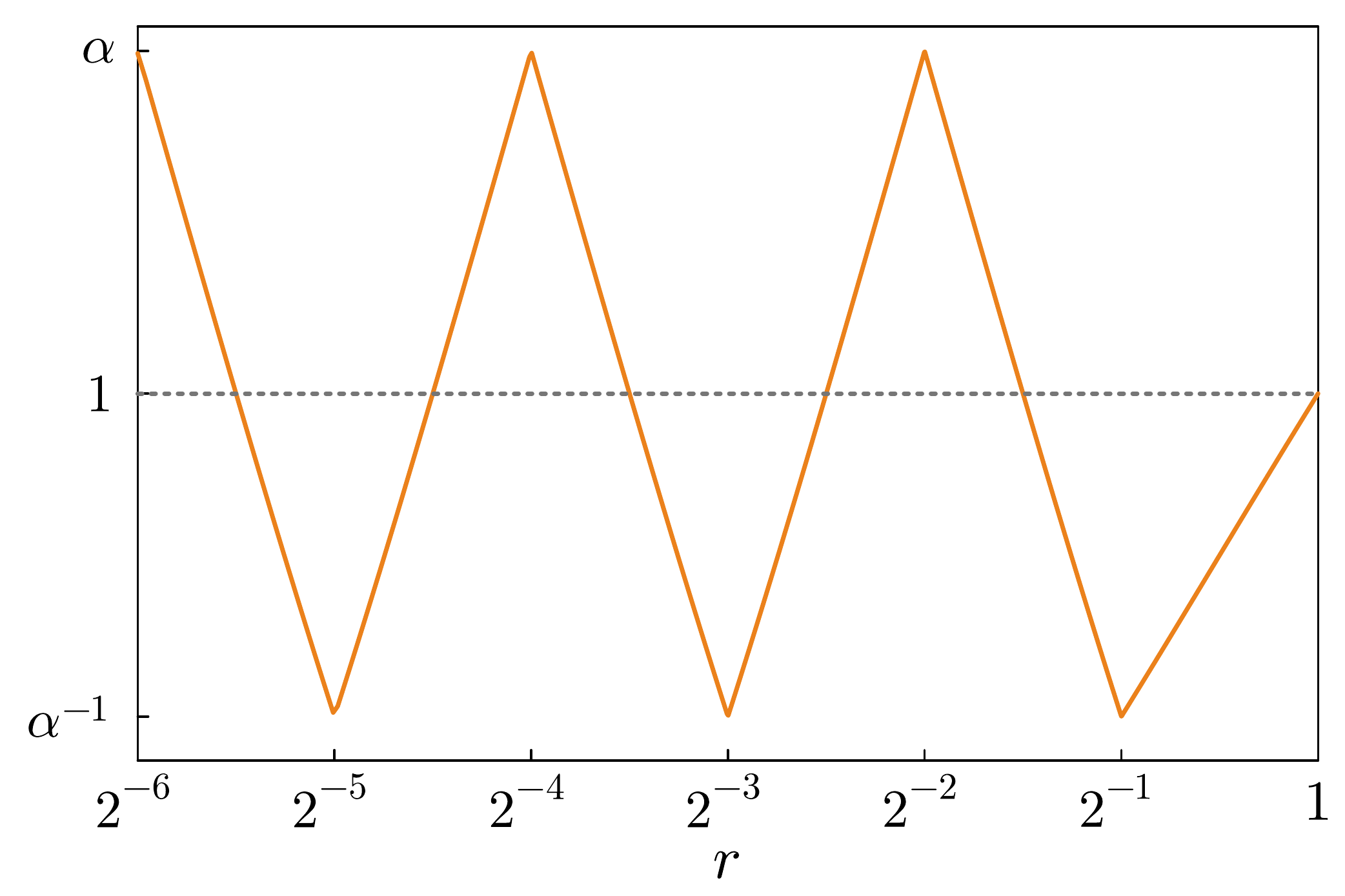}
		\subcaption{\raggedright The ratio $\nicefrac{ \crcdf{-1}{r} }{ \crcdf{+1}{r} }$  (shown with a logarithmic scale for $r$) oscillates between $\alpha$ and $\alpha^{-1}$ as $r \to 0$, so the $\liminf$ of the ratio is below $1$ and the $\limsup$ is above $1$.}
	\end{subfigure}
	\caption{Illustration of the measures defined in \Cref{thm:oscillation_example} for the parameter choice $a = 2$.}
	\label{fig:oscillation_example}
\end{figure}

\begin{example}[An absolutely continuous measure on $\Reals$ with incomparable maximal points and neither weak nor generalised modes; after an example of I.~Klebanov]
	\label{thm:oscillation_example}
	Let $X$ be any Borel-measurable subset of $\Reals$ containing $[-2, 2]$.
	Fix $a > 1$ and, as illustrated in \Cref{fig:oscillation_example}, define $\mu^{\even}, \mu^{\odd} \in \prob{X}$ via their Lebesgue densities $\rho^{\even}, \rho^{\odd} \colon X \to [0, \infty]$,
	\begin{equation*}
		\rho^{\even} (x)
		\defeq
		\begin{cases}
			0, & \text{if $\absval{ x } > 1$,} \\
			\dfrac{a^{\nicefrac{n}{2}} ( 1 - a^{-1} )}{2(1 - a^{-2})} , & \text{if $a^{-2 - n} \leq \absval{ x } \leq a^{-n}$ for even $n \in \Naturals_{0} \defeq \Naturals \cup \{ 0 \} $,} \\
			\infty, & \text{if $x = 0$,}
		\end{cases}
	\end{equation*}
	and
	\begin{equation*}
		\rho^{\odd} (x)
		\defeq
		\begin{cases}
			0, & \text{if $\absval{ x } > 1$,} \\
			\dfrac{1 - a^{-\nicefrac{1}{2}}}{2(1 - a^{- 1})} , & \text{if $a^{-1} \leq \absval{ x } \leq 1$,} \\
			\dfrac{a^{\nicefrac{n}{2}} ( 1 - a^{-1} )}{2(1 - a^{-2})} , & \text{if $a^{-2 - n} \leq \absval{ x } \leq a^{-n}$ for odd $n \in \Naturals$,} \\
			\infty, & \text{if $x = 0$,}
		\end{cases}
	\end{equation*}
	so that the RCDFs are
	\begin{equation*}
		\mu^{\even} (\cball{0}{r})
		=
		\begin{cases}
			1, & \text{if $r \geq 1$,} \\
			a^{- 1 -\nicefrac{n}{2}} + ( r - a^{-2 - n} ) \dfrac{a^{\nicefrac{n}{2}} ( 1 - a^{-1} )}{1 - a^{-2}} , & \text{if $a^{-2 - n} \leq r \leq a^{-n}$ for even $n \in \Naturals_{0}$,} \\
			0, & \text{if $r = 0$,}
		\end{cases}
	\end{equation*}
	and
	\begin{equation*}
		\mu^{\odd} (\cball{0}{r})
		=
		\begin{cases}
			1, & \text{if $r \geq 1$,} \\
			a^{-\nicefrac{1}{2}} + ( r - a^{-1} ) \dfrac{1 - a^{-\nicefrac{1}{2}}}{1 - a^{- 1}} , & \text{if $a^{-1} \leq r \leq 1$,} \\
			a^{- 1 -\nicefrac{n}{2}} + ( r - a^{-2 - n} ) \dfrac{a^{\nicefrac{n}{2}} ( 1 - a^{-1} )}{1 - a^{-2}} , & \text{if $a^{-2 - n} \leq r \leq a^{-n}$ for odd $n \in \Naturals$,} \\
			0, & \text{if $r = 0$.}
		\end{cases}
	\end{equation*}
	We now consider the probability measure $\mu \defeq \tfrac{1}{2} \mu^{\even} (\quark + 1) + \tfrac{1}{2} \mu^{\odd} (\quark - 1) \in \prob{X}$ with Lebesgue density $\rho \defeq \tfrac{1}{2} \rho^{\even} (\quark + 1) + \tfrac{1}{2} \rho^{\odd} (\quark - 1)$.
	
	We first observe that $\pm 1 \succeq_{0} x$ for any $x \neq \pm 1$.
	For sufficiently small $r > 0$, both $\rho^{\even}$ and $\rho^{\odd}$ are bounded above by a constant on $\cball{x}{r} = [x - r, x + r]$, so that $\crcdf{x}{r} \leq c r$ for some $c \geq 0$.
	On the other hand, by construction, both $\crcdf{-1}{r}$ and $\crcdf{+1}{r}$ are asymptotically equivalent to $\nicefrac{\sqrt{r}}{2}$ as $r \to 0$, from which it follows that ${\pm 1} \succeq_{0} x$.

	However, ${-1}$ and ${+1}$ are incomparable.
	Observe that, for $r = a^{-n}$ with $n \in \Naturals$ even,
	\begin{equation*}
		\frac{\crcdf{-1}{r}}{\crcdf{+1}{r}} = \alpha \defeq \frac{a + 1}{2 a^{1/2}} > 1 ,
	\end{equation*}
	whereas for $r = a^{-n}$ with $n \in \Naturals$ odd, this ratio of ball masses takes the value $\alpha^{-1}$, and, for all $r > 0$, it lies in the interval $[\alpha^{-1}, \alpha]$, all of which can be verified easily from the interpolation formulae for $\mu^{\even} (\cball{0}{r})$ and $\mu^{\odd} (\cball{0}{r})$.
	\Cref{lem:incomp_0} now implies that ${-1} \incomp_{0} {+1}$, since
	\begin{equation*}
		\alpha^{-1} = \liminf_{r \to 0} \frac{\crcdf{-1}{r}}{\crcdf{+1}{r}} < 1 < \limsup_{r \to 0} \frac{\crcdf{-1}{r}}{\crcdf{+1}{r}} = \alpha .
	\end{equation*}

	Thus, the preorder $\preceq_{0}$ induced by $\mu$ has two incomparable maximal elements, namely $\pm 1$, has no greatest elements, and hence $\mu$ has no weak modes (\Cref{lem:GWM_greatest_maximal}).

	We now check that $+1$ and $-1$ are not generalised modes.
	Let $r_n \defeq a^{-2n}$, and suppose that $x_n \to 1$ as $n \to \infty$.
	Choose $N$ large enough that, for all $n \geq N$, $\absval{ x_n - 1 } < \nicefrac{1}{2}$ and $r_n < \nicefrac{1}{2}$.
	As the density $\rho^{\odd}(\quark - 1)$ is a symmetric singularity around $+1$, it follows that $\crcdf{x_n}{r_n} \leq \crcdf{+1}{r_n}$.
	As $M_{r_n} = \crcdf{-1}{r_n}$, we obtain that
	\begin{equation*}
		\liminf_{n \to \infty} \frac{\crcdf{x_n}{r_n}}{M_{r_n}} \leq \liminf_{n \to \infty} \frac{\crcdf{+1}{r_n}}{M_{r_n}} = \alpha^{-1} < 1.
	\end{equation*}
	This proves that $+1$ is not a generalised mode; a similar argument with $(r_n)_{n \in \Naturals} = a^{-2n + 1}$ proves that $-1$ is not a generalised mode.

	Finally, suppose that $x \neq \pm 1$, and let $(r_n)_{n \in \Naturals}$ be any null sequence.
	Let $\varepsilon \defeq \min \bigl\{ \absval{ x - 1 }, \absval{ x + 1 } \bigr\}$.
	Suppose that $x_n \to x$ as $n \to \infty$.
	There must exist $N \in \Naturals$ such that, for all $n \geq N$, $\absval{ x_n - 1 } > \nicefrac{\varepsilon}{2}$ and $\absval{ x_n + 1 } > \nicefrac{\varepsilon}{2}$.
	The Lebesgue density of $\mu$ is bounded on $\Reals \setminus \bigl(\cball{+1}{\nicefrac{\varepsilon}{2}} \cup \cball{-1}{\nicefrac{\varepsilon}{2}}\bigr)$ by some constant $C > 0$, so $\crcdf{x_n}{r_n} \leq Cr_n$ for $n \geq N$.
	As $M_{r_n} \in \Theta\bigl(r_n^{\nicefrac{1}{2}}\bigr)$ as $n \to \infty$, it follows that
	\begin{equation*}
		\liminf_{n \to \infty} \frac{\crcdf{x_n}{r_n}}{M_{r_n}} = 0,
	\end{equation*}
	so $x$ is not a generalised mode.
\end{example}

\Cref{thm:oscillation_example} illustrates a difficulty with weak modes, and one whose cause can be traced to incomparability:
if the space $X$ is partitioned into disjoint positive-mass sets $A$ and $B$, existence of modes for $\mu$ restricted to (or conditioned upon) $A$ and $B$ individually cannot ensure existence of a mode for $\mu$, since the modes of $\mu|_{A}$ and $\mu|_{B}$ may be $\preceq_{0}$-incomparable.

Thus, while $\pm 1$ are intuitively modes and have Lebesgue density $+\infty$, the measure $\mu$ has no modes in any of the senses defined in \Cref{sec:related}.
We emphasise that one cannot simply declare all points with Lebesgue density $+\infty$ to be modes, since this would place all singularities of the density on the same footing, which is clearly undesirable if one singularity is genuinely ``smaller'' than the other in the sense that the RCDFs around these points are, say, $\sqrt{r}$ and $2 \sqrt{r}$, and so the smaller one ought not to be considered a mode.

As suggested in the introduction, this example could be interpreted as evidence that maximal --- rather than greatest --- elements of a preorder are good candidates for modes.
Indeed, from the order-theoretic perspective, maximal elements appear to be just as reasonable as greatest elements, and we hope that this encourages further study of whether maximal elements are sufficient for applications.

The extension theorems of Szpilrajn, Arrow, and Hansson \citep{Hansson1968,Szpilrajn1930} assert that any non-total preorder $\preceq$ can be extended to a total preorder $\preceq'$.
Thus, given the non-totality of $\preceq_{0}$, one might hope to resolve all these issues by defining a mode of $\mu$ to be a $\preceq'_{0}$-greatest element.
Unfortunately, such a total extended preorder is not uniquely determined and so such a definition of a mode would not be well defined:
for the measure $\mu$ of \Cref{thm:oscillation_example}, there are total extensions $\preceq'_{0}$ of $\preceq_{0}$ yielding each of the three situations
\begin{equation*}
	- 1 \preceq'_{0} +1 \not\preceq'_{0} -1 ,
	\quad
	+1 \preceq'_{0} -1 \not\preceq'_{0} +1 ,
	\quad
	\text{and } -1 \asymp'_{0} +1 .
\end{equation*}
That is, which (if any) of $\pm 1$ counts as a mode would seem to be a matter of personal choice.

Finally, we note that similar ideas could be used to construct incomparable points that are not $\preceq_{0}$-maximal, but such examples have less importance for the theory of modes.

\subsection{Absolutely continuous measures with dense antichains}
\label{sec:dense_antichains}

\Cref{thm:oscillation_example} can be easily extended to construct a measure $\mu \in \prob{\Reals}$ with any finite number of mutually incomparable $\preceq_{0}$-maximal elements, none of which are greatest elements.
Indeed, it is natural to wonder how bad the situation of incomparability can be, and in particular how large an antichain can be.
This section's main result, \Cref{thm:countable_dense_antichain}, shows that $\mu$ may have a topologically dense antichain consisting of maximal elements (and mutually incomparable would-be modes are ``nearly everywhere''), even when $\mu$ has a Lebesgue density;
from the perspective of geometric measure theory, the notable point here is that there is no need to resort to singular measures.

We begin with the following straightforward proposition:

\begin{proposition}
	\label{prop:finite_discrete_no_incomp}
	Let $X$ be a finite or discrete metric space and let $\mu \in \prob{X}$.
	Then $\preceq_{0}$ has no incomparable elements.
\end{proposition}

\begin{proof}
	Let $x, x' \in \supp(\mu)$.
	As $X$ is discrete, the measure $\mu$ must be atomic, so
	\begin{equation*}
		\lim_{r \to 0} \frac{\crcdf{x}{r}}{\crcdf{x'}{r}} = \frac{\mu(\{x\})}{\mu(\{x'\})}.
	\end{equation*}
	As the limit exists, the ratio does not oscillate on either side of unity as $r \to 0$, so comparability of $x$ and $x'$ follows from \Cref{lem:incomp_0}.
\end{proof}

\Cref{prop:finite_discrete_no_incomp} shows that any measure on a finite metric space induces a total order $\preceq_0$. 
We now show that incomparability can arise even in very simple settings, such as in a countable metric space or on the real line with a continuous, bounded Lebesgue density.

\begin{example}
	\label{eg:countable_space_antichain}

	\begin{enumerate}[label=(\alph*)]
		\item 
		\label{item:countable_space_antichain_1}
		Let $X$ be the closure of the set $\set{-1 + 2^{-n}}{n \in \Naturals} \cup \set{1 - 2^{-n}}{n \in \Naturals}$ with the Euclidean metric inherited from $\Reals$.
		Define the measure $\mu \in \prob{X}$ by
		\begin{equation*}
			\mu \defeq \frac{1}{Z} \sum_{k = 1}^{\infty} 2^{-4k+1} \delta_{1-2^{-4k+1}} + 2^{-4k-1}\delta_{-1 + 2^{-4k}},
		\end{equation*}
		where $Z > 0$ is a normalisation constant.
		Then $+1$ and $-1$ are incomparable because 
		\begin{align}
			\label{eq:countable_space_antichain_1}
			\frac{\crcdf{+1}{2^{-4k+1}}}{\crcdf{-1}{2^{-4k+1}}} &= \frac{(15Z)^{-1} \times 2^{-4k+1}}{4 \times (15Z)^{-1} \times 2^{-4k+1}} = \frac{1}{4}, \\
			\label{eq:countable_space_antichain_2}
			\frac{\crcdf{+1}{2^{-4k-1}}}{\crcdf{-1}{2^{-4k-1}}} &= \frac{\crcdf{+1}{2^{-4(k+1)+3}}}{\crcdf{-1}{2^{-4k-1}}} = \frac{4 \times (15Z)^{-1} \times 2^{-4(k+1)+3}}{(15Z)^{-1} \times 2^{-4k-1}} = 4.
		\end{align}

		\item
		\label{item:countable_space_antichain_2}
		Let $X = \Reals$ and define the densities
		\begin{equation*}
			\Delta_{w, h}(x) = \begin{cases}
				h\left(1 - \frac{|x|}{w} \right) & |x| \leq w\\
				0 & \text{otherwise,}
			\end{cases}
		\end{equation*}
		which have total mass $wh$ and are supported on the interval $[-w, w]$. 
		By analogy with part \ref{item:countable_space_antichain_1}, let $\mu \in \prob{\Reals}$ have the continuous, bounded Lebesgue density
		\begin{equation*}
			\rho(x) = \frac{1}{Z} \sum_{k \in \Naturals} \Delta_{2^{-4k+1},1}\bigl(x - (1 - 2^{-4k + 2})\bigr) + \Delta_{2^{-4k - 1}, 1}\bigl(x - (-1 + 2^{-4k})\bigr),
		\end{equation*}
		where $Z > 0$ is a normalisation constant.
		As \eqref{eq:countable_space_antichain_1} and \eqref{eq:countable_space_antichain_2} remain true for the measure $\mu$ in this example, the points $\pm 1$ are incomparable.
	\end{enumerate}
\end{example}

While \Cref{eg:countable_space_antichain}\ref{item:countable_space_antichain_2} shows that even a measure with a continuous, bounded Lebesgue density may have an antichain, we show in \Cref{prop:essentially_total_examples} that this antichain is never at the ``top'' of the order as in \Cref{thm:oscillation_example}.

Examples such as \Cref{thm:oscillation_example,eg:countable_space_antichain} can be extended to show that an antichain may be countably infinite.
To do so, we first introduce a family of ``coprime'' oscillatory RCDFs to generalise the RCDFs $\mu^{\even}$ and $\mu^{\odd}$ of \Cref{thm:oscillation_example}:

\begin{figure}[t]
	\centering
	\begin{subfigure}[t]{0.49\textwidth}
		\centering
		\includegraphics[width=\linewidth]{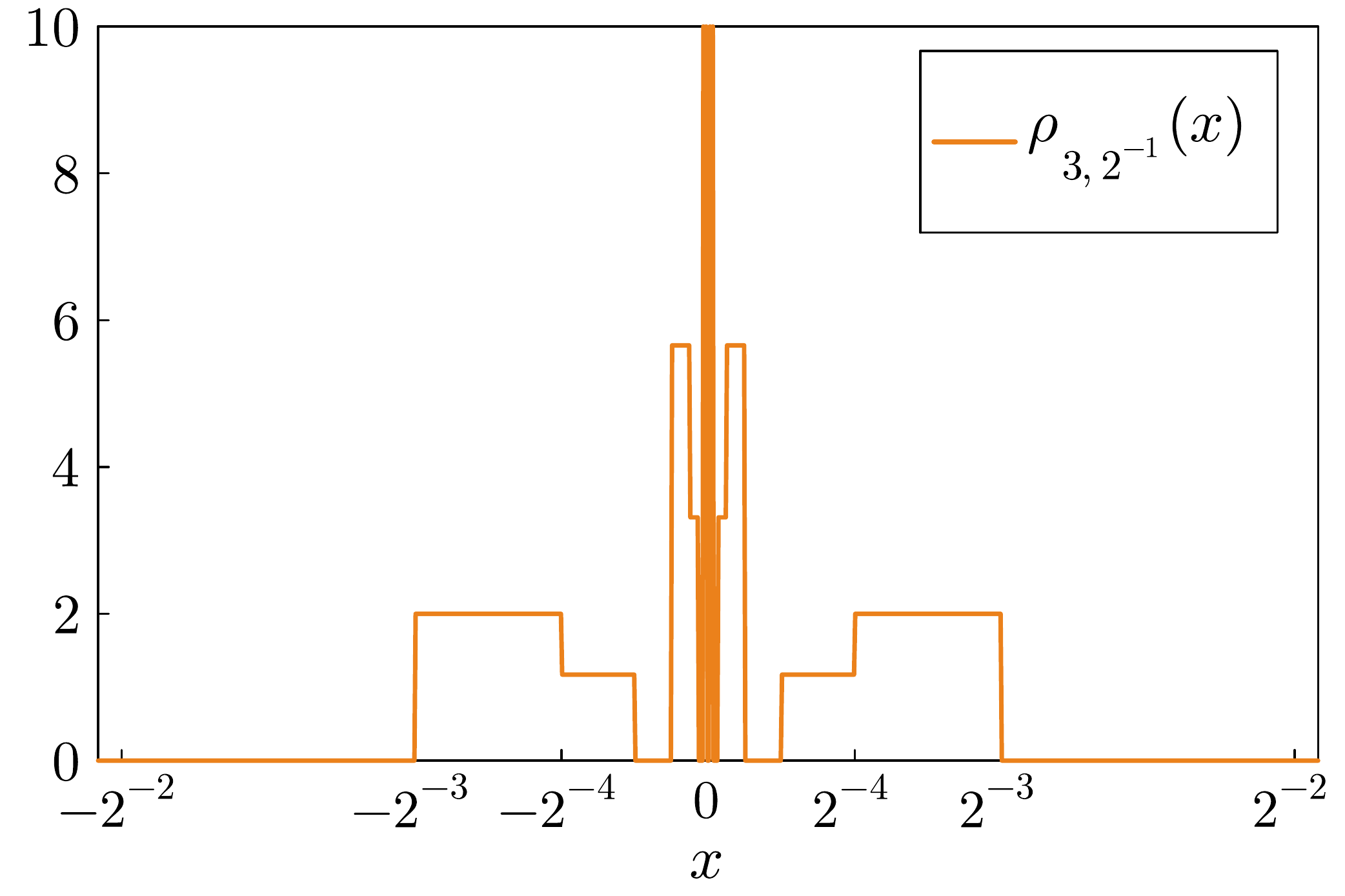}
		\subcaption{\raggedright As in \Cref{thm:oscillation_example}, the density $\rho_{k, m}$ is based on perturbations of the singularity $\absval{ \quark }^{-\nicefrac{1}{2}}$ and has mass $m$.}%
	\end{subfigure}
	\begin{subfigure}[t]{0.49\textwidth}
		\centering
		\includegraphics[width=\linewidth]{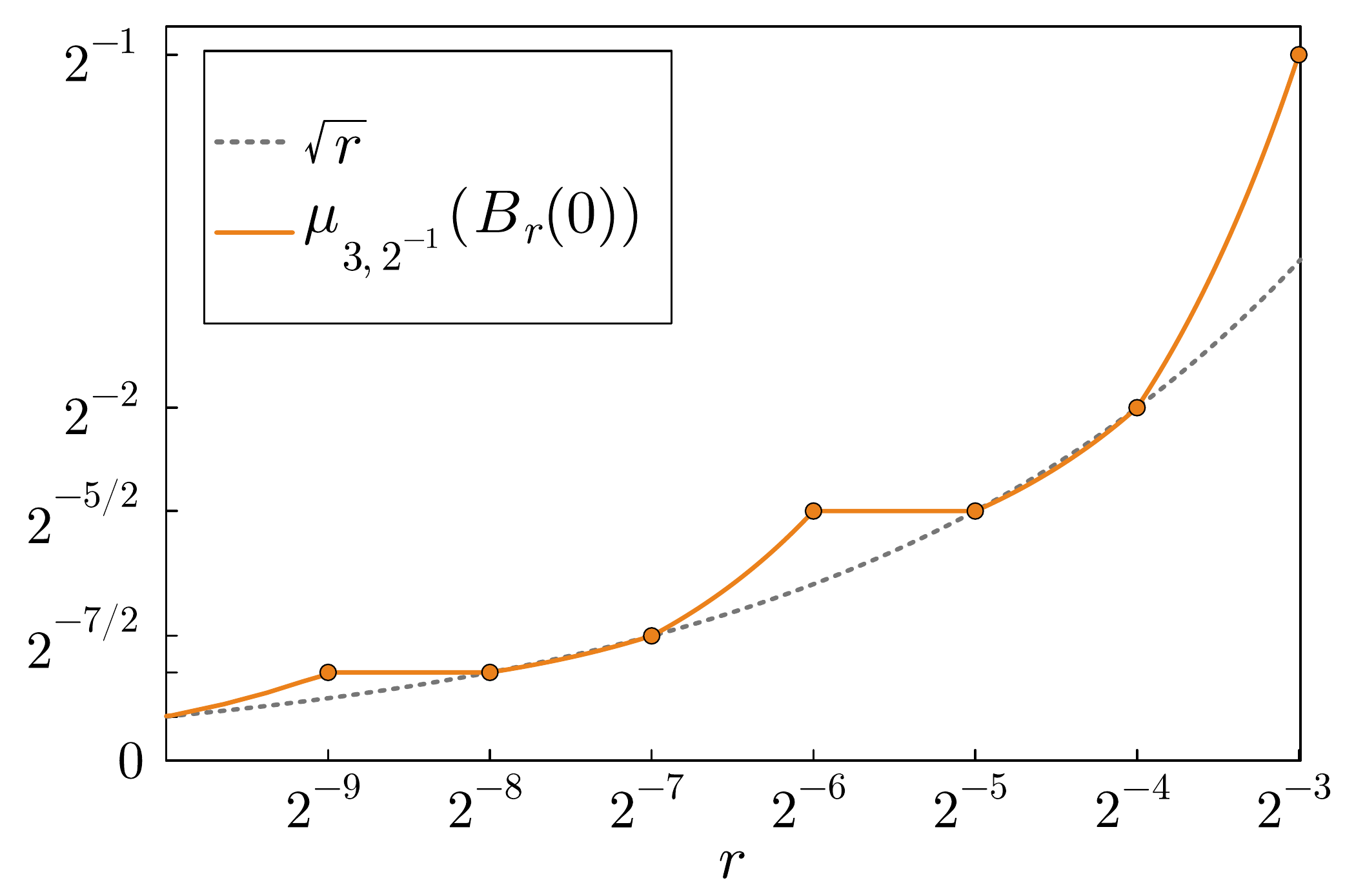}
		\subcaption{\raggedright The RCDF $\mu_{k, m}(\cball{0}{\quark})$ linearly interpolates between the knots $r = a^{-n}$, $n \in \Naturals$ (marked as circles) to obtain the desired perturbations of the ``growth rate'' $\sqrt{r}$.}%
	\end{subfigure}
	\caption{Example of a density $\rho_{k, m}$ and RCDF $\mu_{k,m}$ from \Cref{prop:rcdf_family} with $a = 2$.}
	\label{fig:rcdf_family}
\end{figure}

\begin{proposition}[A family of oscillatory RCDFs]
	\label{prop:rcdf_family}
	Fix $a > 1$ and a natural number $k \geq 2$.
	Construct the Lebesgue densities $\rho_{k} \colon \Reals \to [0, \infty]$ as in \Cref{fig:rcdf_family}\hyperref[fig:rcdf_family]{(a)}, defined by
	\begin{equation*}
		\rho_{k}(x) \defeq
		\begin{cases}
				0, & \text{ if } \absval{ x } > a^{-1},\\
				\frac{1}{2} a^{\frac{n + 1}{2}}, & \text{ if }
				a^{-n-1} < \absval{ x } \leq a^{-n} \text{ for } n \in \Naturals
				\text{ with } k \mid n, \\
				0, & \text{ if } a^{-n-1} < \absval{ x } \leq a^{-n} \text{ for }
				n \in \Naturals  \text{ with } k \mid n + 1, \\
				\frac{1}{2} a^{\frac{n}{2}} \left(\frac{1 - a^{-\nicefrac{1}{2}}}{1 - a^{-1}} \right), &
				\text{ if } a^{-n-1} < \absval{ x } \leq a^{-n} \text{ for } n \in \Naturals
				\text{ with }
				k \nmid n \text{ and } k\nmid n + 1, \\
				\infty, & \text{ if } x = 0,
		\end{cases}
	\end{equation*}
	and, given $m > 0$, define the corresponding truncated densities $\rho_{k, m}(x) \defeq \rho_k(x) \one \left[ |x| \leq r(m) \right]$, with the truncation radius $r(m)$ chosen such that
	\begin{equation*}
		r(m) \defeq \inf \Set{s}{\int_{\cball{0}{s}} \rho_k(t)\,\rd t = m}.
	\end{equation*}
	Write $\mu_{k,m}$ for the measure on the real line with $\mu_{k, m}(\Reals) = m$ and Lebesgue density $\rho_{k,m}$.
	\begin{enumerate}[label=(\alph*)]
		\item 
		\label{item:rcdf_family_1} 
		The RCDF $s \mapsto \mu_{k, m}(\cball{0}{s})$ linearly interpolates between the knots
		\begin{equation*}
			\bigset{(a^{-n}, a^{-\nicefrac{n}{2}})}{n \in \Naturals,~k\nmid n}
			\cup \bigset{(a^{-n}, a^{\nicefrac{1}{2} - \nicefrac{n}{2}})}{n \in \Naturals,~k \mid n} \cup \bigl\{ (0, 0) \bigr\},
		\end{equation*}
		until truncated at radius $r(m)$ (\Cref{fig:rcdf_family}\hyperref[fig:rcdf_family]{(b)}) and has formula
		\begin{equation*}
			\hspace{-2.5em}
			\mu_{k,m}(\cball{0}{s}) =
			\begin{cases}
					\mu_{k,m}(\cball{0}{r(m)}), & \text{if } s > r(m), \\
					a^{\frac{n + 1}{2}} s, & \text{if $a^{-n-1} < s \leq a^{-n}$ for $n \in \Naturals$ with $k \mid n$,} \\
					a^{-\frac{n}{2}}, & \text{if $a^{-n-1} < s \leq a^{-n}$ for $n \in \Naturals$ with $k \mid n + 1$,} \\
					\frac{1 - a^{-\nicefrac{1}{2}}}{1 - a^{-1}}  \left(a^{\frac{n}{2}} s + a^{-\frac{n}{2}}\right), &
					\text{if $a^{-n-1} < s \leq a^{-n}$ for $n \in \Naturals$ with $k \nmid n$ and $k \nmid n + 1$,} \\
					0,& \text{if $s = 0$.}
			\end{cases}
		\end{equation*}

		\item
		\label{item:rcdf_family_2} 
		In particular, if $a^{-n} \leq r(m)$, 
		\begin{equation*}
			\mu_{k, m}(\cball{0}{a^{-n}}) = \begin{cases}
					a^{-\nicefrac{n}{2}}, & \text{if $k \nmid n$,} \\
					a^{\nicefrac{1}{2} - \nicefrac{n}{2}}, & \text{if $k \mid n$.}
			\end{cases}
		\end{equation*}

		\item 
		\label{item:rcdf_family_3} 
		Given distinct coprime integers $k, k' \geq 2$ and arbitrary $m, m' > 0$,
		\begin{equation*}
			\liminf_{s \to 0} \frac{\mu_{k,m}(\cball{0}{s})}{\mu_{k',m'}(\cball{0}{s})} <
			1 <
			\limsup_{s \to 0} \frac{\mu_{k,m}(\cball{0}{s})}{\mu_{k',m'}(\cball{0}{s})}.
		\end{equation*}

		\item 
		\label{item:rcdf_family_4} 
		Provided $s \leq r(m)$, we have $\sqrt{\nicefrac{s}{a}} \leq \mu_{k, m}(\cball{0}{s}) \leq \sqrt{as}$.

		\item 
		\label{item:rcdf_family_5}
		The truncation radius satisfies $r(m) \leq am^2$.

		\item
		\label{item:rcdf_family_6}
		The density $\rho_k$ satisfies $\rho_k(t) \leq t^{-\nicefrac{1}{2}}$ for all $t \in \Reals$.
	\end{enumerate}
\end{proposition}

\begin{proof}
	\begin{enumerate}[label=(\alph*)]
		\item
		The formula for the RCDF follows by integrating the density $\rho_{k, m}$.

		\item 
		The value at the knots $a^{-n}$ follows from \ref{item:rcdf_family_1}.

		\item 
		We exploit the fact that $k$ and $k'$ are coprime, so the sequence $(n_i)_{i \in \Naturals} = (ik' - 1)k \nearrow \infty$ is divisible by $k$ but not $k'$, and the sequence $(m_i)_{i \in \Naturals} = (ik - 1)k' \nearrow \infty$ is divisible by $k'$ but not $k$.
		For sufficiently large $i$, $a^{-n_i} \leq \min \{ r(m), r(m') \}$, and hence by \ref{item:rcdf_family_2} we obtain
		\begin{equation*}
			\frac{\mu_{k, m}(\cball{0}{a^{-n_i}})}{\mu_{k', m'}(\cball{0}{a^{-n_i}})} = \frac{a^{\nicefrac{1}{2} - \nicefrac{n_i}{2}}}{a^{-\nicefrac{n_i}{2}}} = a^{\frac{1}{2}}.
		\end{equation*}
		Similarly, for $i$ sufficiently large such that $a^{-m_i} \leq \min \{ r(m), r(m') \}$,
		\begin{equation*}
			\frac{\mu_{k, m}(\cball{0}{a^{-m_i}})}{\mu_{k', m'}(\cball{0}{a^{-m_i}})} = \frac{a^{-\nicefrac{m_i}{2}}}{a^{\nicefrac{1}{2} - \nicefrac{m_i}{2}}} = a^{-\frac{1}{2}}.
		\end{equation*}
		As these hold for all $i$ sufficiently large, and $a^{-n_i}$ and $a^{-m_i}$ converge to zero, the desired inequality follows.

		\item
		The lower bound follows because, for $s \leq r(m)$,
		\begin{equation*}
			\mu_{k,m}(\cball{0}{s}) \geq \mu_{k,m}(\cball{0}{a^{-\lfloor -\log_a(s) \rfloor}})
			\geq a^{-\lfloor -\log_a(s) \rfloor/2} \geq \sqrt{\nicefrac{s}{a}},
		\end{equation*}
		where the penultimate inequality uses \ref{item:rcdf_family_2}; the upper bound is easily verified from the construction of $\mu_{k,m}$ as a linear interpolation of the knots.

		\item
		As $\int_{\cball{0}{s}} \rho_{k}(t) \,\rd t \geq \sqrt{\nicefrac{s}{a}}$, it follows that $\int_{\cball{0}{am^2}} \rho_{k}(t)\,\rd t \geq m$, and hence $r \leq am^2$.

		\item
		This is easily verified from the expression for $\rho_k$.
		\qedhere
	\end{enumerate}	
\end{proof}

We now use \Cref{prop:rcdf_family} to show that a maximal antichain of a measure can be topologically dense even in the apparently well-behaved case of an absolutely continuous probability measure on the real line.
Our example shows that the set of $\preceq_0$-maximal elements might be very different to the set of $\preceq_0$-greatest elements: the measure we construct has a dense set of maximal elements, yet it does not possess any greatest element because none of those maximal elements is globally comparable.

In spirit, the idea is much the same as \Cref{thm:oscillation_example}: centre mutually incomparable compactly supported singularities at a dense collection of points $\{ q_{k} \}_{k \in \Naturals}$.
This is much more subtle, however, as one must take care to ensure that the points $q_{k}$ are distant enough from one another that the singularities neither interfere with each other nor accumulate too much mass at a point outside of the dense set.
Here, this is achieved by taking the $q_{k}$ to be multiples of powers of two, a case that is easily analysed but quite sparse.
Indeed, we write $D$ for the set of dyadic rationals, which we write as the disjoint union over the \defterm{levels} $D_{\ell} \defeq \Set{(2i - 1)2^{-\ell}}{1 \leq i \leq 2^{\ell - 1}}$.
By a slight abuse of terminology, we also describe the sum of the densities centred at points in $D_{\ell}$ as the $\ell\textsuperscript{th}$ level of the measure.

\begin{theorem}[An absolutely continuous measure on $\Reals$ with a countable dense antichain]
	\label{thm:countable_dense_antichain}

	Let $\mu \in \prob{\Reals}$ have the Lebesgue density $\rho \colon \Reals \to \Reals$ as shown in \Cref{fig:countable_antichain}, defined by
	\begin{equation*}
		\rho(x) \defeq \sum_{\ell = 1}^{\infty} \sum_{i = 1}^{2^{\ell - 1}} \rho_{\indexof{\ell}{i}, \massof{\ell}}(x - q_{\ell, i}),
	\end{equation*}
	where $\rho_{k,m}$ is the density constructed in \Cref{prop:rcdf_family} with parameter $a = 2$; $\indexof{\ell}{i}$ is the $(2^{\ell - 1} + i - 1)\textsuperscript{th}$ prime; $\massof{\ell} \defeq 2^{-2\ell+1}$; and $q_{\ell,i} \defeq (2i - 1) 2^{-\ell} \in D_{\ell}$.
	Then:
	\begin{enumerate}[label=(\alph*)]
		\item 
		\label{item:countable_dense_antichain_is_prob_measure}
		the $\ell\textsuperscript{th}$ level of the measure $\mu$, consisting of all densities centred at points in $D_{\ell}$, has mass $2^{-\ell}$, and hence $\mu$ is a probability measure;
		\item 
		\label{item:countable_dense_antichain_incomp}
		the set of dyadic rationals $D = \Set{(2i-1)2^{-\ell}}{\ell \in \Naturals,~1 \leq i \leq 2^{\ell - 1}}$ is a $\preceq_0$-antichain;
		\item
		\label{item:countable_dense_antichain_maximal}
		every element of $D$ is $\preceq_0$-maximal.
	\end{enumerate}
\end{theorem}

\begin{proof}
	\begin{enumerate}[label=(\alph*)]
		\item 
		By construction, each density in level $\ell$ has mass $\massof{\ell} = 2^{-2\ell + 1}$, and there are $2^{\ell - 1}$ densities, giving a total mass of $2^{-\ell}$. 
		It follows that $\mu$ is a probability measure as $\int_{\Reals} \rho(x) \,\rd x = \sum_{\ell \in \Naturals} 2^{-\ell} = 1$.

		\item
		Take distinct elements $q_{\ell,i},q_{\ell', i'} \in D$. 
		It is sufficient to check that $q_{\ell,i} \not \preceq_0 q_{\ell',i'}$, as one can swap $q_{\ell,i}$ and $q_{\ell', i'}$ to obtain that $q_{\ell,i} \incomp_0 q_{\ell',i'}$. 
		Asymptotically, $\crcdf{q_{\ell,i}}{r} \sim \mu_{\indexof{\ell}{i}, \massof{\ell}}(\cball{0}{r})$ as $r \to 0$, and likewise $\crcdf{q_{\ell',i'}}{r} \sim \mu_{\indexof{\ell'}{i'},\massof{\ell'}}(\cball{0}{r})$ (\Cref{lem:rcdf_behaviour_countable_dense_antichain}\ref{item:rcdf_behaviour_dyadic_rationals}).
		Using the identity $\limsup_{r \to 0} f(r) g(r) = \limsup_{r \to 0} f(r) \lim_{r \to 0} g(r)$, we obtain that 
		\begin{align*}
			& \limsup_{r \to 0} \frac{\crcdf{q_{\ell,i}}{r}}{\crcdf{q_{\ell',i'}}{r}} \\
			& \quad = \limsup_{r \to 0} \frac{\mu_{\indexof{\ell}{i},\massof{\ell}}(\cball{0}{r})}{\mu_{\indexof{\ell'}{i'}, \massof{\ell'}}(\cball{0}{r})}
			\lim_{r \to 0} \frac{\crcdf{q_{\ell,i}}{r}}{\mu_{\indexof{\ell}{i},\massof{\ell}}(\cball{0}{r})}
			\frac{\mu_{\indexof{\ell'}{i'}, \massof{\ell'}}(\cball{0}{r})}{\crcdf{q_{\ell',i'}}{r}} \\
			&\quad = \limsup_{r \to 0} \frac{\mu_{\indexof{\ell}{i}, \massof{\ell}}(\cball{0}{r})}{\mu_{\indexof{\ell'}{i'}, \massof{\ell'}}(\cball{0}{r})} > 1	
		\end{align*}
		where the final line follows by the construction of the oscillatory RCDFs in \Cref{prop:rcdf_family}\ref{item:rcdf_family_3} as $\indexof{\ell}{i}$ and $\indexof{\ell'}{i'}$ are distinct primes.
		This proves that $q_{\ell, i} \not \preceq_0 q_{\ell',i'}$ as claimed, from which incomparability follows.
		
		\item
		To show that $q \in D$ is maximal, it suffices to check that $q \not \preceq_0 x$ for any $x \in [0, 1] \setminus D$; 
		part \ref{item:countable_dense_antichain_incomp} proves that $q \not \preceq_0 x$ when $x \in D$.
		To prove this, we must characterise the behaviour of the RCDF $\crcdf{x}{r}$; this depends on the properties of the binary representation of $x$ and in particular on a quantity we call the \defterm{dyadic irrationality exponent} $\dyadicirratexp{x} \in [1, \infty)$ (\Cref{defn:dyadic_irrationality_exponent}).
		If $\dyadicirratexp{x} < 4$, then $\crcdf{x}{r} \in o(r^{\nicefrac{1}{2}})$ (\Cref{lem:rcdf_behaviour_countable_dense_antichain}\ref{item:rcdf_behaviour_badly_approximated}); as $\crcdf{q}{r} \in \Theta(r^{\nicefrac{1}{2}})$ by the construction of the density centred at $q$, it follows that $q \not \preceq_0 x$ because
		\begin{equation*}
			\limsup_{r \to 0} \frac{\crcdf{x}{r}}{\crcdf{q}{r}} = 0.
		\end{equation*}
		If $\dyadicirratexp{x} > 4$, then $x$ is approximated particularly well by a sequence of dyadic rationals, so there exists a sequence of scales as $r \to 0$ such that the RCDF $\crcdf{x}{r}$ behaves much like its approximating dyadic rational. 
		In fact, this approximation is so good that $q \incomp_0 x$ (\Cref{lem:rcdf_behaviour_countable_dense_antichain}\ref{item:rcdf_behaviour_incomp}) for the same reason that two dyadic rationals are incomparable.
		In the critical case $\dyadicirratexp{x} = 4$, there exist examples with $q \incomp_0 x$ and examples where $x \prec_0 q$, but in either case we can still verify that $q \not \preceq_0 x$ (\Cref{lem:rcdf_behaviour_countable_dense_antichain}\ref{item:rcdf_behaviour_dyadic_rationals_not_dominated}) as required.
		This proves that no $x \in [0, 1] \setminus D$ can dominate any $q \in D$, completing the proof.
		\qedhere
	\end{enumerate}
\end{proof}

\begin{remark}
	\label{rem:countable_dense_antichain}
	\begin{enumerate}[label=(\alph*)]
		\item
		The proof shows that the dyadic rationals do not form a maximal antichain in the sense of setwise inclusion: points with $\dyadicirratexp{x} > 4$ are also incomparable with the dyadic rationals;
		thus, the cardinality of a maximal antichain is at least $\aleph_0$.
		On the other hand, the Lebesgue differentiation theorem implies that any antichain has Lebesgue measure zero (see also \Cref{prop:essentially_total_examples}\ref{item:essentially_total_examples_lebesgue}), so one cannot expect to find a larger antichain in a measure-theoretic sense.

		\item 
		Our construction is not limited to this specific dense set and enumeration, or even to absolutely continuous measures on the real line;
		for example, one can reweight a Gaussian measure on a separable Hilbert space $H$ to have a similar RCDF to our prototypical measures $\mu_{k, m}$ at the point $0$, then place such measures at points in a dense subset of $H$.
		Another possibility is to argue as in \Cref{thm:countable_dense_antichain} using $\Rationals \cap [0, 1]$ as the dense set; 
		the behaviour then depends on the usual number-theoretic irrationality exponent\footnote{For further details about the irrationality exponent, traditionally denoted $\mu(x)$, see e.g.\ \citet{FeldmanNesterenko1998}.} instead of the dyadic irrationality exponent $\dyadicirratexp{x}$, but one still obtains a dense antichain containing all rationals in $[0, 1]$.
		Some of the technical steps are described in more detail in \citet[Section~7.3 and Appendix~A]{Lambley2022}.
	\end{enumerate}
\end{remark}

\begin{figure}[t]
	\centering
	\begin{subfigure}[t]{0.49\textwidth}
		\centering
		\includegraphics[width=\linewidth]{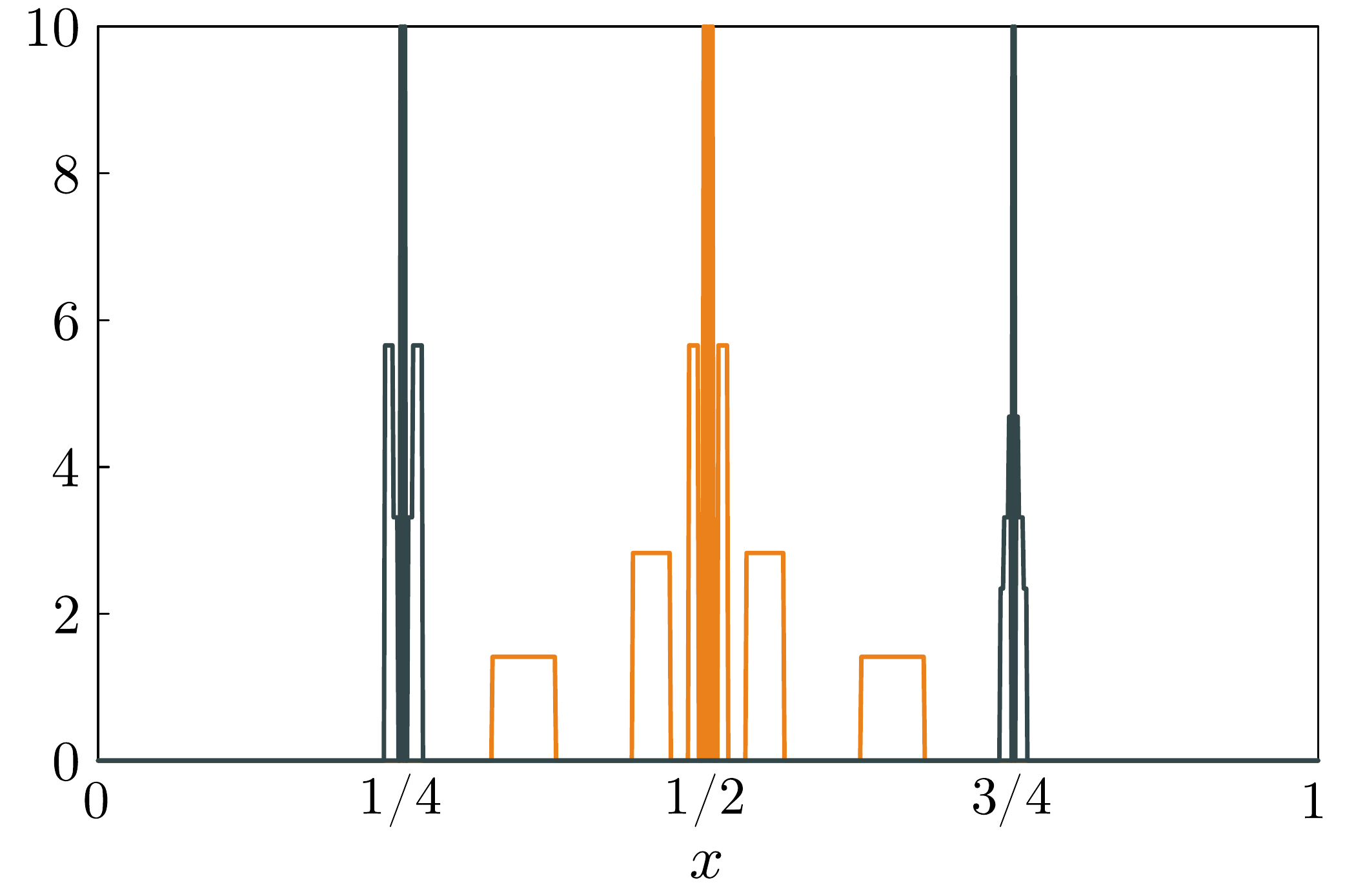}
		\subcaption{\raggedright The density $\rho$ is constructed as a sum of the prototype densities $\rho_{k,m}$. The orange density is $\rho_{2, 2^{-1}}(\quark - \nicefrac{1}{2})$, and the grey densities are $\rho_{3, 2^{-3}}(\quark - \nicefrac{1}{4})$ and $\rho_{5, 2^{-3}}(\quark - \nicefrac{3}{4})$.}%
	\end{subfigure}
	\begin{subfigure}[t]{0.49\textwidth}
		\centering
		\includegraphics[width=\linewidth]{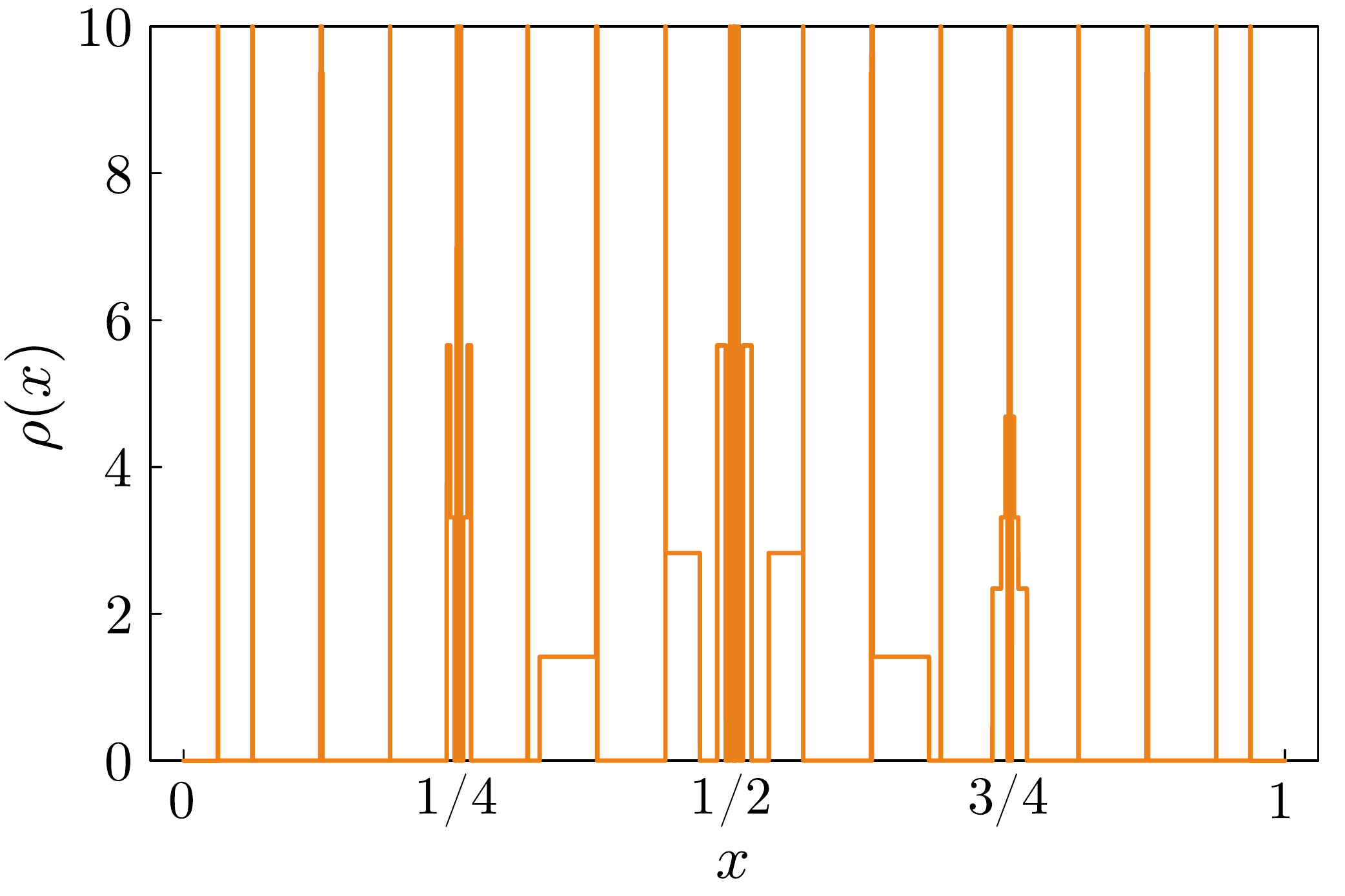}
		\subcaption{\raggedright Approximation of the density, truncated at the fifth level (i.e.\ with the densities centred at all dyadic rationals of the form $c2^{-n}$ with $n \leq 5$).}%

	\end{subfigure}
	\caption{The density $\rho$ from \Cref{thm:countable_dense_antichain} for which the dyadic rationals are an antichain.}
	\label{fig:countable_antichain}
\end{figure}

\subsection{Essential totality}
\label{sec:essential_totality}

The need for a $\preceq_0$-greatest element to be globally comparable is a non-trivial one, and it can fail rather dramatically, e.g.\ when the maximal elements form a dense antichain as in \Cref{thm:countable_dense_antichain}.
Such examples could be criticised as somewhat artificial, but we feel that they highlight the importance of checking for incomparability and developing technical conditions on the measure which prevent it.

One could rule out incomparability if $\preceq_0$ were total, but this is not true in general, and checking this condition is often difficult in practice.
We propose a somewhat weaker condition, where one can tolerate incomparability away from the ``top'' of the preorder, as long as any candidate for a maximal element is also globally comparable.

Our condition of \defterm{essential totality} can be interpreted as an order-theoretic generalisation of the $M$-property of \citet{AyanbayevKlebanovLieSullivan2022_I};
recall \eqref{eq:property_M}.
A motivating example is that of a Gaussian measure $\mu$ on an infinite-dimensional space $X$:
the Cameron--Martin space $H(\mu)$ is an essentially total subspace where a maximal element must lie, and any element of the Cameron--Martin space is globally comparable using the OM functional and property $M(\mu, H(\mu))$.

\begin{definition}
	\label{defn:essentially_total}
	Let $X$ be a metric space and let $\mu \in \prob{X}$.
	A non-empty subset $E \subseteq X$ is \defterm{$\mu$-essentially total} if:
	\begin{enumerate}[label=(\alph*)]
		\item
		\label{item:essentially_total_1}
		any two elements of $E$ are comparable (i.e.\ $E$ is a $\preceq_0$-chain);
		\item
		\label{item:essentially_total_2}
		for any $x \in E$ and $x' \in X \setminus E$, $x' \preceq_{0} x$; and
		\item
		\label{item:essentially_total_3}
		for any $x' \in X \setminus E$, there exists $x \in E$ such that $x' \prec_{0} x$.
	\end{enumerate}
\end{definition}

Condition \ref{item:essentially_total_2} says that if $x^{\star} \in E$ is an upper bound on $E$, then it is $\preceq_0$-greatest;
\ref{item:essentially_total_3} says that no element in $X \setminus E$ can be greatest.
We emphasise, though, that there is no need for $E$ to be a large set in any measure-theoretic or topological sense.

\begin{proposition}[Examples of essentially total subsets]
	\label{prop:essentially_total_examples}
	\begin{enumerate}[label=(\alph*)]
		\item
		\label{item:essentially_total_examples_lebesgue}
		Suppose that $X \subseteq \Reals^n$ is open and that $\mu \in \prob{X}$ has continuous density $\rho \colon X \to [0, \infty)$ with respect to $\lambda^{n}$.
		Then $E \defeq \set{x \in X}{\rho(x) > 0}$ is $\mu$-essentially total, and $I_{\mu}(x) \defeq -\log \rho(x)$ is an OM functional with domain $E$.

		\item
		\label{item:essentially_total_examples_om_and_m_prop}
		Suppose that $\mu \in \prob{X}$ has an OM functional $I_{\mu} \colon E \to \Reals$ and property $M(\mu, E)$ holds.
		Then $E$ is $\mu$-essentially total.

		\item
		\label{item:essentially_total_examples_reweighted}
		Suppose more generally that $\mu_0 \in \prob{X}$ has an OM functional $I_{\mu_0} \colon E \to \Reals$ and property $M(\mu_0, E)$ holds, and that $\mu \in \prob{X}$ has Radon--Nikodym derivative
		\begin{equation*}
			\frac{\rd \mu}{\rd \mu_0}(x) \propto \exp\bigl(-\Phi(x)\bigr)
		\end{equation*}
		for some locally uniformly continuous \defterm{potential} $\Phi \colon X \to \Reals$.
		Then $E$ is $\mu$-essentially total, and $I_{\mu}(x) \defeq I_{\mu_0}(x) + \Phi(x)$ is an OM functional for $\mu$.
	\end{enumerate}
\end{proposition}

\begin{proof}
	\begin{enumerate}[label=(\alph*)]
		\item
		The Lebesgue differentiation theorem implies that for any $x \in X$,
		\begin{equation*}
			\lim_{r \to 0} \frac{\crcdf{x}{r}}{\lambda^{n}(\cball{x}{r})} = \rho(x).
		\end{equation*}
		For any $x$ and $x' \in E$, one can pick $r$ sufficiently small such that $\cball{x}{r}$ and $\cball{x'}{r}$ lie in the open set $X$.
		This implies that $\lambda^{n}(\cball{x}{r}) = \lambda^{n}(\cball{x'}{r})$, and so
		\begin{equation}
			\label{eq:lebesgue_differentiation}
			\lim_{r \to 0} \frac{\crcdf{x}{r}}{\crcdf{x'}{r}} = \lim_{r \to 0} \frac{\crcdf{x}{r}}{\lambda^{n}(\cball{x}{r})} \lim_{r \to 0} \frac{\lambda^{n}(\cball{x'}{r})}{\crcdf{x'}{r}} = \frac{\rho(x)}{\rho(x')}.
		\end{equation}
		Hence, $E$ is a chain and $I_{\mu}$ is an OM functional on $E$.
		When $x' \in X \setminus E$, one can still apply the Lebesgue differentiation theorem to obtain
		\begin{equation*}
			\lim_{r \to 0} \frac{\crcdf{x'}{r}}{\lambda^{n}(\cball{x'}{r})} = 0,
		\end{equation*}
		so an argument similar to that in \eqref{eq:lebesgue_differentiation} proves that $x' \prec_0 x$ for any $x \in E$.
		\item
		The existence of an OM functional $I_{\mu}$ proves that $E$ is a chain.
		Using the $M$-property and \citet[Lemma~B.1]{AyanbayevKlebanovLieSullivan2022_I}, for $x' \in X \setminus E$ and $x \in E$, we must have $x' \prec_0 x$, because
		\begin{equation*}
			\lim_{r \to 0} \frac{\crcdf{x'}{r}}{\crcdf{x}{r}} = 0.
		\end{equation*}

		\item By \citet[Lemma~B.8]{AyanbayevKlebanovLieSullivan2022_I}, $I_{\mu}$ is an OM functional for $\mu$ and property $M(\mu, E)$ holds.
		The result follows by \ref{item:essentially_total_examples_om_and_m_prop}.
		\qedhere
	\end{enumerate}
\end{proof}

\begin{proposition}
	\label{prop:essentially_total_properties}
	Let $X$ be a metric space and let $\mu \in \prob{X}$.
	Suppose that $\varnothing \neq E \subseteq X$ is $\mu$-essentially total.
	\begin{enumerate}[label=(\alph*)]
		\item
		\label{item:essentially_total_properties_maximal_greatest_in_E}
		Any $\preceq_0$-maximal element must lie in $E$ and is $\preceq_0$-greatest.
		\item
		\label{item:essentially_total_properties_variational}
		If $\mu$ admits an OM functional $I_{\mu} \colon E \to \Reals$, then
		\begin{equation*}
			x^{\star} \text{ is $\preceq_0$-greatest} \iff x^{\star} \in E \text{ and } x^{\star} \text{ minimises } I_{\mu}.
		\end{equation*}
	\end{enumerate}
\end{proposition}

\begin{proof}
	\begin{enumerate}[label=(\alph*)]
		\item
		A maximal element $x^{\star}$ must lie in $E$, or else one could find $x \in E$ such that $x^{\star} \prec_0 x$ by essential totality, contradicting the maximality of $x^{\star}$.
		Conditions \ref{item:essentially_total_1} and \ref{item:essentially_total_2} of essential totality together imply that $x^{\star}$ is globally comparable, so it must be greatest (\Cref{lem:GWM_greatest_maximal}).

		\item
		Using the OM functional for $E$, one finds that
		\begin{align*}
			x^{\star} \in E \text{ is an upper bound for } E &\iff
			\lim_{r \to 0} \frac{\crcdf{x}{r}}{\crcdf{x^{\star}}{r}} \leq 1 \text{ for all } x \in E \\
			&\iff
			\frac{e^{-I_{\mu}(x)}}{e^{-I_{\mu}(x^{\star})}} \leq 1 \text{ for all } x \in E  \\
			&\iff
			x^{\star} \text{ minimises } I_{\mu}.
		\end{align*}
		If $x^{\star}$ is $\preceq_0$-greatest, then $x^{\star} \in E$ by \ref{item:essentially_total_properties_maximal_greatest_in_E}, and the previous implications prove that $x^{\star}$ minimises $I_{\mu}$.
		Conversely, the definition of essential totality ensures that an upper bound for $E$ is $\preceq_0$-greatest, proving the reverse implication.
		\qedhere
	\end{enumerate}
\end{proof}

The variational characterisation of weak modes as minimisers of the OM functional generalises the result of \citet[Proposition~4.1]{AyanbayevKlebanovLieSullivan2022_I} to essentially total subsets.
Specialising to the case of a continuous Lebesgue density on an open set (\Cref{prop:essentially_total_examples}\ref{item:essentially_total_examples_lebesgue}) recovers the intuitive result that $x^{\star}$ is a weak mode if and only if it is a global maximiser of $\rho$.
The situation is more subtle if $X$ is not open:
the measure in \Cref{eg:upward_closures_wrt_preceq_0}\ref{item:upward_closures_wrt_preceq_0_1} restricted to $X = [0, 1]$ has a continuous Lebesgue density maximised at $x^{\star} = 1$, but $x^{\star}$ is not a weak mode.

As a consequence of our result on reweightings of well-behaved measures (\Cref{prop:essentially_total_examples}\ref{item:essentially_total_examples_reweighted}), we obtain the significant corollary that maximal elements are always greatest when the measure is a Bayesian posterior as in \eqref{eq:Bayes} arising from a Gaussian prior.
This is highly reassuring from the perspective of applications: pathological examples in the style of \Cref{thm:countable_dense_antichain} with non-greatest maximal elements do not occur in Bayesian posteriors for well-behaved inverse problems.

\section{Closing remarks}
\label{sec:conclusion}

This article has proposed that modes of probability measures should be understood as greatest or maximal elements of preorders that are defined using the masses of metric balls.

At fixed radius $r > 0$, there is an obvious choice of total preorder, and the order-theoretic point of view opens up attractive proof techniques for the existence of maximal/greatest elements (radius-$r$ modes) (\Cref{thm:r_greatest}).
However, we have also seen that such radius-$r$ modes can fail to exist (\Cref{eg:no_radius_1_mode,eg:no_radius_r_mode}), which provides further justification for the use of asymptotic maximising families as proposed by \citet{KlebanovWacker2022}, and we are able to contribute to the convergence analysis of such families as $r \to 0$ (\Cref{thm:limits_of_r-modes,thm:AMFs_and_strong_modes}).

In the limit as $r \to 0$, there are several limiting preorders that one could consider.
The one on which we have focussed, whose greatest elements are weak modes, is a non-total preorder.
Indeed, we have shown that even absolutely continuous measures can admit topologically dense antichains (\Cref{thm:countable_dense_antichain}), indicating that a measure must satisfy stringent regularity conditions to be certain of having greatest elements, i.e.\ weak modes.

As remarked in the introduction, we hope that this article will stimulate further discussion in the community about the ``correct'' definition of a mode.
We argue that there is a tension between the order-theoretic desire for modes to be merely \emph{maximal} elements of some preorder and an application-driven desire for modes to be \emph{greatest} elements.
To some extent, this tension can be avoided if one works only with particularly nice measures that display no oscillatory properties or that satisfy criteria such as essential totality, thus keeping all pathologies away from the ``top'' of the preorder.

Further useful new definitions of modes may be introduced and one would hope that they correspond to preorders.
However, as explored in \Cref{sec:alternative_small-radius_preorders}, it may well be that such definitions only induce non-transitive \emph{relations}.
In such cases, the loss of transitivity is not necessarily fatal, so long as it is kept away from the ``top'' of the relation, so that maximal/greatest elements may be defined.

On a high level, it would be interesting to know whether or not there can exist a function assigning to every (sufficiently well-behaved) measure $\mu \in \prob{X}$ a total preorder $\preceq^{\mu}$ whose maximal or greatest elements are useful modes for $\mu$.
This would appear to be a major open question that will involve much further investigation.

\appendix

\section{Technical supporting results}
\label{sec:technical}

\subsection{Radial cumulative distribution functions}

\begin{lemma}[Properties of RCDFs]
	\label{lem:RCDF}
	Let $X$ be a metric space and let $\mu \in \prob{X}$.
	\begin{enumerate}[label=(\alph*)]
		\item \label{lem:RCDF_in_x}
		For each $r > 0$, $x \mapsto \crcdf{x}{r}$ is upper semicontinuous.
		\item \label{lem:RCDF_in_r}
		For each $x \in X$, $r \mapsto \crcdf{x}{r}$ is monotonically increasing, is continuous from the right, has limits from the left, and is upper semicontinuous.
		Furthermore, $r \mapsto \crcdf{x}{r}$ is differentiable $\lambda^{1}$-a.e.
	\end{enumerate}
\end{lemma}

\begin{proof}
	For \ref{lem:RCDF_in_x}, fix $r > 0$ and let $(x_{n})_{n \in \Naturals}$ converge in $X$ to some $x \in X$.
	Then
	\begin{align*}
		\crcdf{x}{r}
		& =	\lim_{n \to \infty} \crcdf{x}{r + d(x, x_{n})} & & \text{since } \cball{x}{r} = \bigcap_{n \in \Naturals} \cball{x}{r + d(x, x_{n})} \\
		& = \limsup_{n \to \infty} \crcdf{x}{r + d(x, x_{n})} \\
		& \geq \limsup_{n \to \infty} \crcdf{x_{n}}{r} & & \text{since $\cball{x}{r + d(x, x_{n})} \supseteq \cball{x_{n}}{r}$} .
	\end{align*}

	For \ref{lem:RCDF_in_r}, monotonicity follows from the monotonicity of probability.
	To examine continuity, fix $x \in X$ and let $(r_{n})_{n \in \Naturals}$ be a convergent sequence in $[0, \infty)$ with limit $r \geq 0$.
	If $(r_{n})_{n \in \Naturals}$ is decreasing, then $\bigcap_{n \in \Naturals} \cball{x}{r_{n}} =  \cball{x}{r}$ and so the continuity of probability along monotone sequences implies that $\crcdf{x}{r_{n}} \searrow \crcdf{x}{r}$, which establishes continuity from the right.
	If $(r_{n})_{n \in \Naturals}$ is increasing, then $\bigcup_{n \in \Naturals} \cball{x}{r_{n}} = \oball{x}{r}$, and continuity of probability implies that $\crcdf{x}{r_{n}} \nearrow \mu(\oball{x}{r}) \leq \crcdf{x}{r}$, and this establishes existence of a limit from the left.
	Now let $r_{n} \to r$, and make no assumption that this convergence is monotone.
	By the above,
	\begin{equation*}
		\limsup_{n \to \infty} \crcdf{x}{r_{n}} \in \{ \mu(\oball{x}{r}) , \crcdf{x}{r} \} ,
	\end{equation*}
	i.e.\ the $\limsup$ is at most $\crcdf{x}{r}$, which establishes upper semicontinuity.
	Finally, a.e.-differentiability of $r \mapsto \mu(\cball{x}{r})$ follows from monotonicity and Lebesgue's theorem on differentiability of monotone functions.
\end{proof}

\begin{corollary}
	\label{cor:no_unbounded_sequence_approximates_M_r}
	Let $X$ be a separable metric space, let $\mu \in \prob{X}$, and fix $r > 0$.
	Then
	\begin{equation*}
		M_{r} \defeq \sup_{x \in X} \crcdf{x}{r} > 0
	\end{equation*}
	and every sequence $(x_{n})_{n \in \Naturals}$ such that $\crcdf{x_{n}}{r} \to M_{r}$ as $n \to \infty$ is bounded.
\end{corollary}

\begin{proof}
	The separability of $X$ implies that $\supp(\mu) \neq \varnothing$ \citep[Theorem~12.14]{AliprantisBorder2006}, and so there must exist at least one $x \in X$ with $\crcdf{x}{r} > 0$.
	Hence, $M_{r} > 0$.

	Now let $(x_{n})_{n \in \Naturals}$ be any sequence such that $\crcdf{x_{n}}{r} \to M_{r}$ as $n \to \infty$.
	Then there must exist $N \in \Naturals$ such that
	\begin{equation*}
		n \geq N \implies \crcdf{x_{n}}{r} \geq \nicefrac{M_{r}}{2} > 0,
	\end{equation*}
	i.e.\ $x_{n}$ eventually lies in $\set{ x \in X }{ \crcdf{x_{n}}{r} \geq \nicefrac{M_{r}}{2} }$, which is a bounded set by \Cref{lem:upper_closure_is_closed_and_bounded}\ref{lem:upper_closure_is_closed_and_bounded_3}, and so $(x_{n})_{n \in \Naturals}$ is a bounded sequence.
\end{proof}

\begin{definition}
	\label{defn:spherically_non-atomic}

	Let $X$ be a metric space.
	A probability measure $\mu \in \prob{X}$ will be called \defterm{spherically non-atomic} if every metric sphere has zero $\mu$-mass, i.e., for all $r \geq 0$ and all $x \in X$, $\mu(\oball{x}{r}) = \crcdf{x}{r}$.
\end{definition}

\begin{corollary}[RCDFs of spherically non-atomic measures]
	\label{cor:RCDF_spherically_nonatomic}
	Let $X$ be a metric space and assume that $\mu \in \prob{X}$ is spherically non-atomic.
	\begin{enumerate}[label=(\alph*)]
		\item \label{cor:RCDF_spherically_nonatomic_in_x}
		For each $r > 0$, $x \mapsto \crcdf{x}{r}$ is continuous.
		\item \label{cor:RCDF_spherically_nonatomic_in_r}
		For each $x \in X$, $r \mapsto \crcdf{x}{r}$ is monotonically increasing and continuous.
		\item \label{cor:RCDF_spherically_nonatomic_joint}
		$(r, x) \mapsto \crcdf{x}{r}$ is continuous.
		\item \label{cor:RCDF_spherically_nonatomic_at_0}
		For each $x \in X$, $\mu(\{x\}) = \mu(\cball{x}{0}) = \lim_{r \to 0} \crcdf{x}{r} = 0$.
	\end{enumerate}
\end{corollary}

\begin{proof}
	Easy modification of the proof of \Cref{lem:RCDF} shows that
	\begin{itemize}
		\item for each $r > 0$, $x \mapsto \mu(\oball{x}{r})$ is lower semicontinuous;
		\item for each $x \in X$, $r \mapsto \mu(\oball{x}{r})$ is monotonically increasing, is continuous from the left, has limits from the right, and is lower semicontinuous.
	\end{itemize}
	For a spherically non-atomic measure $\mu$, each occurrence of $\mu(\oball{x}{r})$ can be replaced with $\crcdf{x}{r}$, and this together with the original statement of \Cref{lem:RCDF} proves parts \ref{cor:RCDF_spherically_nonatomic_in_x} and \ref{cor:RCDF_spherically_nonatomic_in_r}.

	An easy modification of the classical theorem of \cite{Young1910} on the joint continuity of separately continuous functions (see e.g.\ \citet[Theorem~3.1]{Grushka2019}) establishes \ref{cor:RCDF_spherically_nonatomic_joint}.

	Finally, \ref{cor:RCDF_spherically_nonatomic_at_0} follows from $\mu(\{x\}) = \mu(\cball{x}{0}) = \mu(\oball{x}{0}) = \mu(\varnothing) = 0$;
	the claim regarding the limit follows from the continuity of $r \mapsto \crcdf{x}{r}$, as proven in \ref{cor:RCDF_spherically_nonatomic_in_r}.
\end{proof}

\subsection{Radius-\texorpdfstring{$r$}{r} modes in sequence spaces}
\label{sec:radius_r_modes_sequence_spaces}

Given $p \in [1, \infty)$ and $\alpha \in \Reals^{\Naturals}_{>0}$, we define the corresponding weighted $\ell^{p}$ space and its norm by
\begin{equation*}
	\ell^{p}_{\alpha} \defeq \Set{ x = (x_n)_{n \in \Naturals} \in \Reals^{\Naturals}}{\norm{ x }_{\ell^{p}_{\alpha}} \defeq \left( \sum_{n \in \Naturals} \Absval{ \frac{ x_{n} }{ \alpha_{n} } }^{p} \right)^{\nicefrac{1}{p}} < \infty} .
\end{equation*}
We also equip $\Reals^{\Naturals}$ and its subspaces with the finite-dimensional projections
\begin{equation*}
	P_n \colon \Reals^{\Naturals} \to \Reals^n, \quad x = (x_k)_{k \in \Naturals} \mapsto (x_1, \dots, x_n),
\end{equation*}
and denote the ball of radius $r > 0$ centred at $x \in \Reals^n$ by
\begin{equation*}
	\cballn{x}{r}{n} \defeq \Set{y \in \Reals^n}{ \left( \sum_{k = 1}^n \frac{\absval{y_k - x_k}^p}{\alpha_k^p} \right)^{\nicefrac{1}{p}} \leq r}.
\end{equation*}

\begin{lemma}
	\label{lem:lp_projection_properties}
	Let $X = \ell^{p}_{\alpha}$ for some $p \in [1, \infty)$, $\alpha \in \Reals^{\Naturals}_{>0}$ and let $\mu \in \prob{X}$.
	Define the set function $\mu_n(A) \defeq (\mu \circ P_{n}^{-1})(P_n A)$
	(This function is not necessarily a measure.)
	\begin{enumerate}[label=(\alph*)]
		\item \label{item:lp_projection_properties_1} 
		For any $n \in \Naturals$, $x \in X$, and $r > 0$, the projection maps satisfy 
		$P_n \cball{x}{r} = \cballn{P_n x}{r}{n}$.

		\item \label{item:lp_projection_properties_2} For any $x \in X$ and $r > 0$,
		\begin{equation*}
			\bigcap_{n \in \Naturals} P_{n}^{-1}(P_n \cball{x}{r}) = \cball{x}{r}.
		\end{equation*}

		\item \label{item:lp_projection_properties_3} For any $n \in \Naturals$ and $A \in \Borel{X}$, the projection maps satisfy $P_{n + 1}^{-1}\bigl(P_{n+1} A\bigr) \subseteq P_{n}^{-1}\bigl(P_n A\bigr)$.

		\item \label{item:lp_projection_properties_4} For any $n \in \Naturals$ and $A \in \Borel{X}$, the set functions $\mu_n$ satisfy $\mu_n(A) \geq \mu(A)$.

		\item \label{item:lp_projection_properties_5} For any $x \in X$ and $r > 0$, one has $\lim_{n \to \infty} \mu_n(\cball{x}{r}) = \crcdf{x}{r}$.
	\end{enumerate}
\end{lemma}

\begin{proof}
	\begin{enumerate}[label=(\alph*)]
		\item
		Use that
		\begin{align*}
			P_n \cball{x}{r} &= \Set{y \in \Reals^n}{\text{there exists } \tilde{y} \in X \text{ such that } \norm{ \tilde{y} - x }_{\ell^{p}_{\alpha}} \leq r \text{ and } y = P_n \tilde{y}} \\
							&= \Set{y \in \Reals^n}{\text{there exists } \tilde{y} \in X \text{ such that } \sum_{k = 1}^{\infty} \frac{\absval{\tilde{y}_k - x_k}^p}{\alpha_k^p} \leq r^p \text{ and } y = P_n \tilde{y}} \\
							&= \Set{y \in \Reals^n}{\sum_{k = 1}^n \frac{\absval{y_k - x_k}^p}{\alpha_k^p} \leq r^p} = \cballn{P_n x}{r}{n}.
		\end{align*}

		\item
		Observe that, by \ref{item:lp_projection_properties_1},
		\begin{align*}
			\bigcap_{n \in \Naturals} P_{n}^{-1}(P_n \cball{x}{r}) &= \bigcap_{n \in \Naturals} P_{n}^{-1}(\cballn{P_n x}{r}{n}) \\
													&= \Set{y \in X}{P_n y \in \cballn{P_n x}{r}{n} \text{ for all } n \in \Naturals} \\
													&= \Set{y \in X}{\sum_{k = 1}^n \frac{\absval{y_k - x_k}^p}{\alpha_k^p} \leq r^p \text{ for all } n \in \Naturals} = \cball{x}{r}.
		\end{align*}

		\item
		This is a straightforward consequence of the definitions.

		\item
		This follows from the inclusion $A \subseteq P_{n}^{-1}(P_n A)$ and monotonicity of $\mu$.

		\item
		As $\bigl(P_{n}^{-1}(P_n \cball{x}{r})\bigr)_{n \in \Naturals}$ is a decreasing sequence of sets, it follows that
		\begin{equation*}
			\lim_{n \to \infty} \mu_n(A) = \lim_{n \to \infty} \mu\bigl(P_{n}^{-1}(P_n \cball{x}{r})\bigr) = \mu\left(\bigcap_{n \in \Naturals} P_{n}^{-1}(P_n \cball{x}{r}) \right) = \crcdf{x}{r}
		\end{equation*}
		by continuity of measure.
		\qedhere
	\end{enumerate}
\end{proof}

\begin{lemma}[Spherical non-atomicity and weak upper semicontinuity in sequence spaces]
	\label{lem:spherical_non-atomicity_implies_weak_usc}
	Let $X = \ell^{p}_{\alpha}$, $1 \leq p < \infty$, $\alpha \in \Reals^{\Naturals}_{>0}$, and let $\mu \in \prob{X}$.
	Suppose that $\mu \circ P_{n}^{-1} \in \prob{\Reals^{n}}$ is spherically non-atomic for each $n \in \Naturals$.
	Then, for each fixed $r > 0$, the map $x \mapsto \crcdf{x}{r}$ is weakly upper semicontinuous.
\end{lemma}

\begin{proof}
	Suppose that $x_k \rightharpoonup x^{\star}$ as $k \to \infty$, and let $\mu_n(A) \defeq (\mu \circ P_{n}^{-1})(P_n A)$.
	As $\mu_n(\cball{x^{\star}}{r}) \searrow \crcdf{x^{\star}}{r}$ (\Cref{lem:lp_projection_properties}), it follows that, for any $\varepsilon > 0$, there exists $N \in \Naturals$ such that, for all $n \geq N$, $\mu_n(\cball{x^{\star}}{r}) - \crcdf{x^{\star}}{r} < \varepsilon$.
	Using this and the inequality $\mu_n(\cball{x_k}{r}) \geq \mu(\cball{x_k}{r})$ for any $k \in \Naturals$, we obtain
	\begin{equation}
		\label{eq:approx_of_rcdf_difference_by_mu_n}
		\crcdf{x_k}{r} - \crcdf{x^{\star}}{r} \leq \mu_n(\cball{x_k}{r}) - \mu_n(\cball{x^{\star}}{r}) + \varepsilon.
	\end{equation}
	By hypothesis, $x_k \rightharpoonup x^{\star}$, so $P_n x_k \to P_n x^{\star}$ as $k \to \infty$.
	As $\mu \circ P_{n}^{-1}$ is assumed to be spherically non-atomic, $x \mapsto (\mu \circ P_{n}^{-1})(\cballn{x}{r}{n})$ is continuous (\Cref{cor:RCDF_spherically_nonatomic}).
	Hence,
	\begin{align*}
		\lim_{k \to \infty} (\mu \circ P_{n}^{-1})(P_n \cball{x_k}{r}) &= \lim_{k \to \infty} (\mu \circ P_{n}^{-1})(\cballn{P_n x_k}{r}{n}) &&\text{(\Cref{lem:lp_projection_properties}\ref{item:lp_projection_properties_1})} \\
		&= (\mu \circ P_{n}^{-1})(\cballn{P_n x^{\star}}{r}{n}) &&\text{(by continuity)} \\
		&= (\mu \circ P_{n}^{-1})(P_n \cball{x^{\star}}{r}) &&\text{(\Cref{lem:lp_projection_properties}\ref{item:lp_projection_properties_1}).}
	\end{align*}
	Hence, $\lim_{k \to \infty} \mu_n(\cball{x_k}{r}) = \mu_n(\cball{x^{\star}}{r})$.
	Taking limits as $k \to \infty$ in \eqref{eq:approx_of_rcdf_difference_by_mu_n} yields that
	\begin{equation*}
		\limsup_{k \to \infty} \crcdf{x_k}{r} - \crcdf{x^{\star}}{r}
		\leq
		\lim_{k \to \infty} \mu_n(\cball{x_k}{r}) - \mu_n(\cball{x^{\star}}{r}) + \varepsilon
		= \varepsilon.
	\end{equation*}
	As $\varepsilon > 0$ was arbitrary, this shows that $x \mapsto \crcdf{x}{r}$ is weakly upper semicontinuous.
\end{proof}

We now state an explicit version of Anderson's inequality following the inequalities of \citet[Lemma~3.6]{DashtiLawStuartVoss2013} for Gaussian measures and \citet[Lemma~6.2]{AgapiouBurgerDashtiHelin2018} for Besov measures with $p = 1$.

Fix parameters\footnote{In the original setting of real analysis, $s$ and $d$ were interpreted as smoothness and spatial dimension respectively, but for us only the ratio $\nicefrac{s}{d}$ is important.} $1 \leq p < \infty$, $s \in \Reals$, and $d \in \Naturals$;
the (\emph{sequence space}) \emph{Besov space} $X_{p}^{s}$ is defined to be $\ell^{p}_{\gamma}$ for the weighting sequence $\gamma_{k} \defeq k^{- (s/d + 1/2) + 1/p}$, and the (\emph{sequence space}) \emph{Besov measure} $B_{p}^{s}$ is defined to be the countable product measure $\bigotimes_{k \in \Naturals} \mu_{k}$, where $\mu_{k} \in \prob{\Reals}$ has Lebesgue density proportional to $\exp ( - \absval{ x_{k} / \gamma_{k} }^p )$.
It is known that $B_{p}^{s}$ charges $X_{p}^{t}$ with full mass when $t = s - (1 + \eta) \nicefrac{d}{p}$ and $\eta > 0$ \citep[e.g.][Lemma~2]{LassasSaksmanSiltanen2009}.

\begin{lemma}[Explicit Anderson inequality for Besov-$p$ priors, $1 \leq p \leq 2$]
	\label{lem:explicit_anderson_Besov}
	Let $s \in \Reals$, $d \in \Naturals$, $\eta > 0$ and let $t \defeq s - (1 + \eta)\nicefrac{d}{p}$.
	Suppose that $X = X_p^t$ and let $\mu = B_p^s \in \prob{X}$ be a sequence-space Besov measure.
	Then, for any $0 < r < \|x\|_{X_p^t}$ and $x \in X$,
	\begin{equation}
		\label{eq:explicit_anderson_Besov}
		\frac{\crcdf{x}{r}}{\crcdf{0}{r}} \leq \exp\left( -\frac{1}{2} \left(\norm{ x }_{X_p^t} - r\right)^p \right).
	\end{equation}
\end{lemma}

\begin{proof}
	The space $X_p^t$ can be written as the sequence space $\ell^{p}_\delta$ with the weighting sequence $\delta_{k} = k^{-(s/d+1/2) + (2 + \eta)/p} > \gamma_{k} = k^{-(s/d + 1/2) + 1/p}$.
	The formula for the unnormalised marginal density of the Besov measure then yields
	\begin{align*}
		& \frac{\mu_n(\cball{x}{r})}{\mu_n(\cball{0}{r})} \\
		& \quad = \frac{\int_{P_n \cball{x}{r}} \exp\left(-\sum_{i = 1}^n \absval{ \nicefrac{u_{i}}{\gamma_{i}} }^p \right) \,\rd u}{\int_{P_n \cball{0}{r}} \exp\left(-\sum_{i = 1}^n \absval{ \nicefrac{u_{i}}{\gamma_{i}} }^p \right) \,\rd u} \\
		& \quad \leq \frac{\sup_{y \in P_n \cball{x}{r}} \exp(-\frac{1}{2} \sum_{i = 1}^n \absval{ \nicefrac{y_{i}}{\delta_{i}} }^p) \int_{P_n \cball{x}{r}} \exp\left(-\sum_{i = 1}^n \absval{ \nicefrac{u_{i}}{\gamma_{i}} }^p + \frac{1}{2} \sum_{i = 1}^n \absval{ \nicefrac{u_{i}}{\delta_{i}} }^p \right) \,\rd u}{\int_{P_n \cball{0}{r}} \exp\left(-\sum_{i = 1}^n \absval{ \nicefrac{u_{i}}{\gamma_{i}} }^p + \frac{1}{2} \sum_{i = 1}^n \absval{ \nicefrac{u_{i}}{\delta_{i}} }^p \right) \,\rd u} \\
		& \quad \leq \sup_{y \in P_n \cball{x}{r}} \exp\left(-\frac{1}{2} \sum_{i = 1}^n \Absval{ \nicefrac{ y_i }{ \delta_i } }^p \right),
	\end{align*}
	where the ratio of integrals is bounded above by $1$ using Anderson's inequality \citep{Anderson1955}.
	Hence, as $\lim_{n \to \infty} \mu_n(\cball{x}{r}) = \crcdf{x}{r}$ (\Cref{lem:lp_projection_properties}),
	\begin{equation*}
		\frac{\crcdf{x}{r}}{\crcdf{0}{r}} = \lim_{n \to \infty} \frac{\mu_n(\cball{x}{r})}{\mu_n(\cball{0}{r})} \leq \lim_{n \to \infty} \sup_{y \in P_n \cball{x}{r}} \exp\left(-\frac{1}{2} \sum_{i = 1}^n \Absval{ \nicefrac{ y_i }{ \delta_i } }^p \right) = \exp\left(-\frac{1}{2} \left(\norm{ x }_{X_p^t} - r\right)^p\right) ,
	\end{equation*}
	which establishes \eqref{eq:explicit_anderson_Besov}.
\end{proof}

\begin{theorem}[Radius-$r$ modes for product measures on weighted $\ell^{p}$ spaces]
	\label{thm:lp_alpha_radius_r_modes}
	Let $X = \ell_{\alpha}^{p}$, $1 < p < \infty$, $\alpha \in \Reals^{\Naturals}_{> 0}$. 
	Let $\mu_0 = \bigotimes_{n \in \Naturals} \mu_n \in \prob{X}$ with each $\mu_n \ll \lambda^{1}$ on $\Reals$.
	If $\mu \ll \mu_0$, then $\mu$ has a radius-$r$ mode for any $r > 0$.
\end{theorem}

\begin{proof}
	As $\mu_0$ is a product of the measures $\mu_n$, which are all absolutely continuous with respect to $\lambda^{1}$, the pushforward measures $\mu_0 \circ P_{n}^{-1}$ are absolutely continuous with respect to $\lambda^{n}$.
	As $\mu \ll \mu_0$, it follows that $\mu \circ P_{n}^{-1} \ll \mu_0 \circ P_{n}^{-1}$% (\Cref{lem:lp_posterior_pushforwards_non-atomic})
	, so the pushforwards of $\mu$ are also absolutely continuous with respect to $\lambda^{n}$.
	Hence, the measure $\mu$ has spherically non-atomic pushforwards $\mu \circ P_{n}^{-1}$, and so the map $x \mapsto \crcdf{x}{r}$ is weakly upper semicontinuous for any $r > 0$ (\Cref{lem:spherical_non-atomicity_implies_weak_usc}).
	As any sequence $(x_n)_{n \in \Naturals}$ with $\crcdf{x_n}{r} \nearrow M_r$ is bounded (\Cref{cor:no_unbounded_sequence_approximates_M_r}), there must exist a weakly convergent subsequence $(x_{n_k})_{k \in \Naturals} \rightharpoonup x^{\star}$ by the reflexivity of $\ell^{p}_\alpha$, $p > 1$. 
	The weak upper semicontinuity of $x \mapsto \crcdf{x}{r}$ implies that $x^{\star}$ is a radius-$r$ mode, because $M_r = \lim_{k \to \infty} \crcdf{x_{n_k}}{r} \leq \crcdf{x^{\star}}{r}$.
\end{proof}

\begin{corollary}
	\label{cor:Gaussian_Besov_Cauchy_radius_r_modes}
	Suppose that $X = \ell^{p}_{\alpha}$, $1 < p < \infty$, $\alpha \in \Reals^{\Naturals}_{> 0}$.
	If $\mu \ll \mu_0$ and $\mu_0 = \bigotimes_{n \in \Naturals} \mu_n$ is
	\begin{enumerate}[label=(\alph*)]
		\item a Gaussian measure;
		\item a Besov measure; or
		\item a Cauchy measure,
	\end{enumerate}
	then $\mu$ has a radius-$r$ mode for any $r > 0$.
\end{corollary}

\subsection{Small-ball probabilities for the countable dense antichain}

The measure in \Cref{thm:countable_dense_antichain} places variants of the prototype densities $\rho_{k,m}$ at each dyadic rational. 
While a variety of constructions are possible (see \Cref{rem:countable_dense_antichain}), we choose to use the dyadic rationals in $[0, 1]$ as the dense set for simplicity.
The advantage of using the dyadic rationals is that one can exploit the natural ``level'' structure, writing $D_{\ell} \defeq \Set{(2i-1)2^{-\ell}}{1 \leq i \leq 2^{\ell - 1}}$ for those dyadic rationals which, in their simplest form, can be written as $c 2^{-\ell}$.
From this level structure, one can explicitly compute the distance between terms and bound the support of the densities $\rho_{\indexof{\ell}{i}, \massof{\ell}}$ centred at points in $D_{\ell}$.

As the dyadic rationals are precisely the points in $[0, 1]$ with a finite binary expansion, the behaviour of the RCDF $\crcdf{x}{r}$ at an arbitrary point $x \in [0, 1]$ depends on a quantity which we call the \defterm{dyadic irrationality exponent}, and denote $\dyadicirratexp{x}$, which can be thought of as a quantitative estimate on the length of runs of $0$s or $1$s in the binary expansion of $x$.
This quantity is very much analogous to the number-theoretic \defterm{irrationality measure} $\varphi(x, n) \defeq \min_{1 < p < q, q \leq n} |x - \nicefrac{p}{q}|$ and corresponding \defterm{irrationality exponent} $\irratexp{x}$ \citep{FeldmanNesterenko1998}.
We choose the notation $\irratexp{x}$ for the irrationality exponent and not the more usual $\mu(x)$ to avoid confusion with the measure $\mu$.

\begin{definition}
	\label{defn:dyadic_irrationality_exponent}

	\begin{enumerate}[label=(\alph*)]
		\item The \defterm{dyadic irrationality measure} of $x \in [0, 1]$ is given by $\dyadicmeas{x}{\ell} \defeq \min_{q \in \biguplus_{i = 1}^{\ell} D_i} \absval{ x - q }$.

		\item The \defterm{dyadic irrationality exponent} of $x \in [0, 1] \setminus D$ is given by
		\begin{equation*}
			\dyadicirratexp{x} \defeq \inf \Set{\beta \geq 1}{\liminf_{\ell \to \infty} \frac{\dyadicmeas{x}{\ell}}{2^{-\beta \ell}} > 0} = \sup \Set{\beta \geq 1}{\liminf_{\ell \to \infty} \frac{\dyadicmeas{x}{\ell}}{2^{-\beta \ell}} < \infty}.
		\end{equation*}
	\end{enumerate}
\end{definition}

The dyadic irrationality exponent $\dyadicirratexp{x}$ is well defined, and indeed
\begin{equation}
	\label{eq:dyadic_irrat_exp_zero_or_infty}
	\liminf_{\ell \to \infty} \frac{\dyadicmeas{x}{\ell}}{2^{-\beta \ell}} = \begin{cases}
		0, & \beta < \dyadicirratexp{x}, \\
		+\infty, & \beta > \dyadicirratexp{x}.
	\end{cases}
\end{equation}
In general, it is not possible to say anything about the limit in \eqref{eq:dyadic_irrat_exp_zero_or_infty} in the critical case $\beta = \dyadicirratexp{x}$; the value could be anything in the range $[0, +\infty]$.
Furthermore, as $\irratmeas{x}{2^{\ell}} \leq \dyadicmeas{x}{\ell}$, it immediately follows that $\dyadicirratexp{x} \leq \irratexp{x}$, but the quantities are not equal in general --- for example, any irrational number must satisfy $\irratexp{x} \geq 2$ by Dirichlet's approximation theorem, but one can construct irrational numbers with $\dyadicirratexp{x} = 1$.

\begin{lemma}[Properties of the measure in \Cref{thm:countable_dense_antichain}]
	\label{lem:distance_bounds_countable_dense_antichain}
	Let $\mu \in \prob{\Reals}$ be the measure in \Cref{thm:countable_dense_antichain} and fix $\ell \in \Naturals$.
	\begin{enumerate}[label=(\alph*)]
		\item 
		\label{item:distance_bounds_support_size}
		Given $q_{\ell,i} \in D_{\ell}$, the support of the density $\rho_{\indexof{\ell}{i}, \massof{\ell}}(\quark - q_{\ell,i})$ is contained in $\cball{q_{\ell,i}}{2^{-4\ell+3}}$. 

		\item 
		\label{item:distance_bounds_disjoint_support}
		For distinct $q_{\ell,i}, q_{\ell,i'} \in D_{\ell}$, the densities $\rho_{\indexof{\ell}{i},\massof{\ell}}(\quark - q_{\ell,i})$ and $\rho_{\indexof{\ell}{i'},\massof{\ell}}(\quark - q_{\ell,i'})$ have disjoint support, and the supports are a distance at least $2^{-\ell} - 2^{-4\ell + 4}$ apart.

		\item
		\label{item:distance_bounds_equidistant_implies_r_growth}
		Fix $\delta, r > 0$ and $x \in [0, 1]$, and suppose that $\inf_{q \in D_{\ell}} \absval{ x - q } > \delta + r$. 
		Then
		\begin{equation*}
			\sum_{i = 1}^{2^{\ell - 1}} \int_{x - r}^{x + r} \rho_{\indexof{\ell}{i}, \massof{\ell}}(t - q_{\ell, i})\,\rd t \leq 2r \delta^{-\nicefrac{1}{2}}.
		\end{equation*}
	\end{enumerate}
\end{lemma}

\begin{proof}
	\begin{enumerate}[label=(\alph*)]
		\item
		By construction, $\rho_{\indexof{\ell}{i}, \massof{\ell}}$ has mass $\massof{\ell} = 2^{-2\ell+1}$.
		Hence, the truncation radius of this singularity is at most $2\massof{\ell}^{2}$ (\Cref{prop:rcdf_family}\ref{item:rcdf_family_5}) and therefore the support is contained in a ball of radius $2 \times 2^{-4\ell+2} \leq 2^{-4\ell + 3}$.

		\item 
		Distinct points in $D_{\ell}$ must be a distance at least $2^{-\ell}$ apart, and by \ref{item:distance_bounds_support_size} the supports of the densities $\rho_{\indexof{\ell}{i}, \massof{\ell}}$ and $\rho_{\indexof{\ell}{i'}, \massof{\ell}}$ are contained in a ball of radius $2^{-4\ell+3}$.
		Hence, their supports must be at least a distance $2^{-\ell} - 2 \times 2^{-4\ell+3}$ apart.

		\item
		By \Cref{prop:rcdf_family}\ref{item:rcdf_family_6}, outside of $\cball{q_{\ell,i}}{\delta}$, the density $\rho_{\indexof{\ell}{i}, \massof{\ell}}$ is bounded above by $\delta^{-\nicefrac{1}{2}}$, and the supports of the densities are disjoint, so the upper bound follows immediately.
		\qedhere
	\end{enumerate}
\end{proof}

\begin{lemma}[Behaviour of RCDFs in \Cref{thm:countable_dense_antichain}]
	\label{lem:rcdf_behaviour_countable_dense_antichain}
	
	Let $\mu \in \prob{\Reals}$ be the measure in \Cref{thm:countable_dense_antichain}.
	\begin{enumerate}[label=(\alph*)]
		\item
		\label{item:rcdf_behaviour_dyadic_rationals}
		Suppose that $q_{\ell,i} \in D_{\ell}$.
		Then $\crcdf{q_{\ell,i}}{r} \sim \mu_{\indexof{\ell}{i}, \massof{\ell}}(\cball{0}{r})$ as $r \to 0$.

		\item
		\label{item:rcdf_behaviour_badly_approximated}
		Suppose that $x \in [0, 1] \setminus D$ and that $\dyadicirratexp{x} < 4$.
		Then, for any $\beta \in (\dyadicirratexp{x}, 4)$, it follows that $\crcdf{x}{r} \in O(r^{\min\{1, \nicefrac{2}{\beta}\}})$ as $r \to 0$, and in particular $\crcdf{x}{r} \in o(r^{\nicefrac{1}{2}})$.

		\item
		\label{item:rcdf_behaviour_dyadic_rationals_not_dominated}
		Suppose that $x \in [0, 1] \setminus D$.
		Then, for any $q \in D$, 
		\begin{equation*}
			\liminf_{r \to 0} \frac{\crcdf{x}{r}}{\crcdf{q}{r}} < 1.
		\end{equation*}

		\item
		\label{item:rcdf_behaviour_incomp}
		Suppose that $x \in [0, 1] \setminus D$ and that $\dyadicirratexp{x} > 4$.
		Then, for any $q \in D$, 
		\begin{equation*}
			\limsup_{r \to 0} \frac{\crcdf{x}{r}}{\crcdf{q}{r}} > 1,
		\end{equation*}
		and therefore $x \incomp_0 q$.
	\end{enumerate}
\end{lemma}

\begin{proof}
	\begin{enumerate}[label=(\alph*)]
		\item 
		For any $r < 2^{-\ell}$, the ball $\cball{q_{\ell,i}}{r}$ does not contain any element of $\biguplus_{i = 1}^{\ell} D_i$ except $q_{\ell,i}$.
		Furthermore, for $m \in \Naturals$, if $r < 2^{-m} - 2^{-4m+3}$, then $\cball{q_{\ell,i}}{r}$ does not intersect the support of any singularity centred at $q' \in D_m$. (\Cref{lem:distance_bounds_countable_dense_antichain}\ref{item:distance_bounds_support_size}).
		
		As $2^{-m} - 2^{-4m+3} \in \Omega(2^{-m})$ as $m \to \infty$, there exists $M \in \Naturals$ and $c > 0$ such that
		\begin{equation*}
			2^{-m} - 2^{-4m+3} \geq c2^{-m} \text{ for all $m \geq M$.}
		\end{equation*}
		Picking $\ell_1(r) \defeq \lfloor -\log_2(\nicefrac{r}{c}) \rfloor$, we observe that $\cball{q_{\ell,i}}{r}$ is disjoint from the supports of any singularities in $\biguplus_{i = M}^{\ell_1(r)} D_i$. 
		Hence, we bound the mass from the first $M$ levels using \Cref{lem:distance_bounds_countable_dense_antichain}\ref{item:distance_bounds_equidistant_implies_r_growth}, then note that there is no contribution from levels $M, \dots, \ell_1(r)$, and finally bound the total mass from level $\ell_1(r) + 1$ onwards crudely.
		Fix $\delta \defeq \inf_{q' \in \biguplus_{i = 1}^M D_i} |q_{\ell,i} - q'|$ and suppose that $r \leq \nicefrac{\delta}{2}$;
		then
		\begin{align*}
			\crcdf{q_{\ell,i}}{r} &\leq \mu_{\indexof{\ell}{i},\massof{\ell}}(\cball{0}{r}) \\
			 &\quad\quad+ \sum_{i = 1}^{M} \sum_{j = 1}^{2^{i - 1}} \int_{x - r}^{x + r} \rho_{\indexof{i}{j}, \massof{i}}(t - q_{i,j})\,\rd t + \sum_{i = \ell_1(r) + 1}^{\infty} 2^{-i} \\
			 &\leq \mu_{\indexof{\ell}{i},\massof{\ell}}(\cball{0}{r}) + 2M(\nicefrac{\delta}{2})^{-\nicefrac{1}{2}}r + \sum_{i = \ell_1(r) + 1}^{\infty} 2^{-i} &&\text{(\Cref{lem:distance_bounds_countable_dense_antichain}\ref{item:distance_bounds_equidistant_implies_r_growth})}\\
			 &= \mu_{\indexof{\ell}{i},\massof{\ell}}(\cball{0}{r}) + O(r) \text{ as $r \to 0$.}
		\end{align*}
		As $\mu_{\indexof{\ell}{i},\massof{\ell}}(\cball{0}{r}) \in \Theta(r^{\nicefrac{1}{2}})$ as $r \to 0$ (\Cref{prop:rcdf_family}\ref{item:rcdf_family_3}), the $O(r)$ term is negligible and hence $\crcdf{q_{\ell,i}}{r} \sim \mu_{\indexof{\ell}{i},\massof{\ell}}(\cball{0}{r})$.

		\item
		Take $\beta \in (\dyadicirratexp{x}, 4)$; 
		\eqref{eq:dyadic_irrat_exp_zero_or_infty} implies that
		\begin{equation*}
			\liminf_{\ell \to \infty} \frac{\inf_{q \in \biguplus_{i = 1}^{\ell} D_i} \absval{ x - q }}{2^{-\beta \ell}} = \infty.
		\end{equation*}
		Hence, for $r_{\ell} \defeq 2^{-4\ell+3}$, it follows that $\inf_{q \in \biguplus_{i = 1}^{\ell} D_i} \absval{ x - q } - r_{\ell} \in \Omega(2^{-\beta \ell})$ as $\ell \to \infty$.
		Furthermore, as the supports of the densities centred at distinct elements of $D_{\ell}$ are disjoint and at least a distance $2^{-\ell} - 2^{-4\ell+4} \in \Omega(2^{-\ell})$ apart (\Cref{lem:distance_bounds_countable_dense_antichain}\ref{item:distance_bounds_disjoint_support}), there must exist $L \in \Naturals$ and $c > 0$ such that, for all $\ell \geq L$, 
		\begin{align*}
			\inf_{q \in \biguplus_{i = 1}^{\ell} D_i} \absval{ x - q } - r_{\ell} &> c2^{-\beta \ell} \\
			2^{-\ell} - 2^{-4\ell+4} &> c2^{-\ell}.
		\end{align*}
		Defining $\ell_1(r) \defeq \lfloor -\frac{1}{\beta} \log_2(\nicefrac{r}{c}) \rfloor$ and $\ell_2(r) \defeq \lfloor - \log_2(\nicefrac{r}{c}) \rfloor$, we see that if $L \leq \ell \leq \ell_1(r)$, then $\cball{x}{r}$ is disjoint from the support of every density centred at a point of $D_{\ell}$, and if $\ell_1(r) < \ell \leq \ell_2(r)$, then $\cball{x}{r}$ intersects the support of at most one density centred at a point in $D_{\ell}$.
		For $\ell > \ell_2(r)$, it is sufficient to bound $\crcdf{x}{r}$ by counting the total mass added in the $\ell\textsuperscript{th}$ level.

		Hence, let $\delta \defeq \inf_{q \in \biguplus_{i = 1}^L D_i} \absval{ x - q } > 0$ and pick $r < \nicefrac{\delta}{2}$ so that we may bound the mass from the first $L$ levels using \Cref{lem:distance_bounds_countable_dense_antichain}\ref{item:distance_bounds_equidistant_implies_r_growth}.
		Using this and the claims above,
		\begin{align*}
			\crcdf{x}{r} &\leq \sum_{\ell = 1}^L \sum_{i = 1}^{2^{\ell - 1}} \int_{x - r}^{x + r} \rho_{\indexof{\ell}{i}, \massof{\ell}}(t - q_{\ell,i})\,\rd t + \sum_{\ell = \ell_1(r) + 1}^{\ell_2(r)} \massof{\ell} \\
			 &\quad\quad+ \sum_{\ell = \ell_2(r) + 1}^{\infty} 2^{-\ell} \\
			 &\leq 2L(\nicefrac{\delta}{2})^{-\nicefrac{1}{2}}r + \sum_{\ell = \ell_1(r) + 1}^{\ell_2(r)} 2^{-2\ell+1} + 2^{-\ell_2(r)} &&\text{(\Cref{lem:distance_bounds_countable_dense_antichain}\ref{item:distance_bounds_equidistant_implies_r_growth})}\\
			 &\leq 2L(\nicefrac{\delta}{2})^{-\nicefrac{1}{2}}r + \frac{8}{3} 2^{-2(\ell_1(r) + 1)} + \frac{2r}{c} \\
			 &\leq 2L(\nicefrac{\delta}{2})^{-\nicefrac{1}{2}} r + \frac{8}{3} \left(\frac{r}{c} \right)^{\nicefrac{2}{\beta}} + \frac{2r}{c} \in O(r^{\min\{1, \nicefrac{2}{\beta}\}}) \text{ as $r \to 0$.}
		\end{align*}

		\item
		The case $\dyadicirratexp{x} < 4$ follows from \ref{item:rcdf_behaviour_badly_approximated}.
		Hence, without loss of generality, suppose that $\dyadicirratexp{x} > 1$ and pick $\beta \in (1, \dyadicirratexp{x})$; \eqref{eq:dyadic_irrat_exp_zero_or_infty} implies that
		\begin{equation*}
			\liminf_{\ell \to \infty} \frac{\inf_{q \in \biguplus_{i = 1}^{\ell} D_i} \absval{ x - q }}{2^{-\beta \ell}} = 0.
		\end{equation*}
		Hence, there must exist a sequence $(\ell_k)_{k \in \Naturals} \nearrow \infty$ and a sequence $(q_{\ell_k})_{k \in \Naturals}$ with $q_{\ell_k} \in D_{\ell_k}$ such that $|x - q_{\ell_k}| < 2^{-\beta \ell_k - 1}$.
		This implies that any $q_{\ell_k} \neq q \in \biguplus_{i = 1}^{\ell_k} D_i$ must satisfy $\absval{ x - q } > 2^{-\ell_k} - 2^{-\beta \ell_k - 1}$.
		As it suffices to bound $\crcdf{x}{r}$ at the radii $s_k \defeq 2^{-\beta \ell_k - 1} \searrow 0$, we proceed by bounding the mass contributed by the first $\ell_k$ levels by the total mass from the density centred at $q_{\ell_k}$ plus a $\Theta(r)$ term given by \Cref{lem:distance_bounds_countable_dense_antichain}\ref{item:distance_bounds_equidistant_implies_r_growth}.

		For $\ell > \ell_k$, by a similar argument to that used above, any $q \in D_{\ell}$ satisfies $\absval{ x - q } > 2^{-\ell} - 2^{-\beta \ell_k - 1}$. 
		As the density centred at $q$ is truncated at a radius at most $r_{\ell} \defeq 2^{-4\ell+3}$, and $2^{-\ell} - 2^{-4\ell+3} \in \Omega(2^{-\ell})$ as $\ell \to \infty$, there must exist $L \in \Naturals$ and $c \in (0, 1)$ such that for $L \leq \ell \leq \beta\ell_k$,
		\begin{equation*}
			\absval{ x - q } - r_{\ell} > 2^{-\ell} - 2^{-\beta\ell_k - 1} - 2^{-4\ell+3} > c2^{-\ell} - 2^{-\beta \ell_k - 1}.
		\end{equation*}
		So, $\cball{x}{s_k}$ does not intersect the support of any density centred at a point of $D_{\ell}$ if $s_k < c 2^{-\ell} - 2^{-\beta \ell_k - 1}$;
		hence, if $L \leq \ell < \beta \ell_k + \log_2(c)$, then $\cball{x}{s_k}$ does not intersect the support of any density in the $\ell\textsuperscript{th}$ level.

		Combining these two claims and taking $k$ large enough that $\ell_k \geq L + 1$ yields the bound	
		\begin{align*}
			\crcdf{x}{s_k} &\leq \sum_{\ell = 1}^{\ell_k - 1} \sum_{i = 1}^{2^{\ell - 1}} \int_{x - r}^{x + r} \rho_{\indexof{\ell}{i}, \massof{\ell}}(t - q_{\ell,i})\,\rd t \\
		   &\quad\quad+ \crcdf{q_{\ell_k}}{s_k} + \sum_{\ell = \lfloor\beta\ell_k + \log_2(c)\rfloor}^{\infty} 2^{-\ell} \\
		   &\leq 2\ell_k \left(2^{-\ell_k} - 2^{-\beta \ell_k}\right)^{-\nicefrac{1}{2}} s_k + \crcdf{q_{\ell_k}}{s_k} + \frac{8s_k}{c} &&\text{(\Cref{lem:distance_bounds_countable_dense_antichain}\ref{item:distance_bounds_equidistant_implies_r_growth})} \\
		   &\leq 2\ell_k \left((2s_k)^{\nicefrac{1}{\beta}} - 2s_k\right)^{-\nicefrac{1}{2}}s_k + \crcdf{q_{\ell_k}}{s_k} + \frac{8s_k}{c}.
		\end{align*}
		As $(2s_k)^{\nicefrac{1}{\beta}} - 2s_k \in \Omega(s_k^{\nicefrac{1}{\beta}})$ as $k \to \infty$, we may pick $k$ sufficiently large that $(2s_k)^{\nicefrac{1}{\beta}} - 2s_k \geq Cs_k^{\nicefrac{1}{\beta}}$ for some $C>0$.
		Hence, using that $\ell_k = -\frac{1}{\beta} \left(\log_2(s_k) - 1\right)$,
		\begin{align*}
			\crcdf{x}{s_k} &\leq \crcdf{q_{\ell_k}}{s_k} + 2\ell_k (Cs_k^{\nicefrac{1}{\beta}})^{-\frac{1}{2}} s_k + \frac{8s_k}{c} \\
		   	&\leq  \crcdf{q_{\ell_k}}{s_k} - \frac{2C^{-\nicefrac{1}{2}}}{\beta}\bigl(\log_2(s_k) - 1\bigr) s_k^{1-\nicefrac{1}{2\beta}} + \frac{8s_k}{c} \\
			&= \crcdf{q_{\ell_k}}{s_k} + o(s_k^{\nicefrac{1}{2}}) \text{ as $k \to \infty$.}
		\end{align*}
		The claim follows because
		\begin{equation*}
			\liminf_{r \to 0} \frac{\crcdf{x}{r}}{\crcdf{q}{r}} \leq \liminf_{k \to \infty} \frac{\crcdf{q_{\ell_k}}{s_k}}{\crcdf{q}{s_k}} < 1,
		\end{equation*}
		and the RCDFs at distinct dyadic rationals are chosen so that their ratio oscillates on either side of unity.

		\item
		Take $\beta \in (4, \dyadicirratexp{x})$. 
		Then, by \eqref{eq:dyadic_irrat_exp_zero_or_infty}, there exists a sequence of levels $(\ell_k)_{k \in \Naturals} \nearrow \infty$ and a sequence $(q_{\ell_k})_{k \in \Naturals}$ with $q_{\ell_k} = q_{\ell_k, i_k} \in D_{\ell_k}$ with $|x - q_{\ell_k}| < c 2^{-\beta \ell_k}$.
		Ignoring the contribution from densities centred at points other than $q_{\ell_k}$, we observe that
		\begin{equation*}
			\crcdf{x}{r} \geq \int_{x - r}^{x + r} \rho_{\indexof{\ell_k}{i_k}, \massof{\ell_k}}(t - q_{\ell_k})\,\rd t.
		\end{equation*}
		Either $x > q_{\ell_k}$ or $x < q_{\ell_k}$; we deal with the first case as the second is almost identical.
		Fix $s_k \defeq 2^{-4\ell_k + 3}$.
		By translating the density, we see that
		\begin{align*}
			\crcdf{x}{s_k}
			&\geq \int_{-s_k + (x - q_{\ell_k})}^{s_k} \rho_{\indexof{\ell_k}{i_k}, \massof{\ell_k}}(t)\,\rd t \\
			&= \int_{-s_k}^{s_k} \rho_{\indexof{\ell_k}{i_k}, \massof{\ell_k}}(t)\,\rd t - \int_{-s_k}^{-s_k + (x - q_{\ell_k})} \rho_{\indexof{\ell_k}{i_k}, \massof{\ell_k}}(t)\,\rd t.
		\end{align*}
		Indeed, as the density $\rho_{\indexof{\ell_k}{i_k}, \massof{\ell_k}}$ is truncated at a radius at most $2^{-4\ell+3}$ (\Cref{prop:rcdf_family}\ref{item:rcdf_family_5}), and as $|x - q_{\ell_k}| \in o(s_k)$, one sees that the ball mass around $x$ asymptotically approaches the ball mass around the approximant $q_{\ell_k}$. 
		By a similar argument to \ref{item:rcdf_behaviour_dyadic_rationals_not_dominated},
		\begin{equation*}
			\limsup_{r \to 0} \frac{\crcdf{x}{r}}{\crcdf{q}{r}} \geq \limsup_{k \to \infty} \frac{\crcdf{q_{\ell_k}}{s_k}}{\crcdf{q}{s_k}} > 1,
		\end{equation*}
		i.e.\ $x \not \preceq_0 q$, because the ratio of RCDFs at two distinct dyadic rationals oscillates on either side of one.
		The claim on incomparability then follows because $q \not \preceq_0 x$ (\Cref{lem:rcdf_behaviour_countable_dense_antichain}\ref{item:rcdf_behaviour_dyadic_rationals_not_dominated}).
		\qedhere
	\end{enumerate}
\end{proof}

\section{Alternative small-radius preorders}
\label{sec:alternative_small-radius_preorders}

This section briefly outlines some alternatives to \Cref{defn:analytic_small-radius_preorder} of $\preceq_{0}$ and their shortcomings.

The main difficulty that one encounters with alternative definitions is that the corresponding relation may not be transitive.
We claim that transitivity is an essential property for any small-radius relation: without transitivity, it is not meaningful to talk about maximal and greatest elements, and so the characterisation of modes as greatest elements of an order fails.

Of course, for a small-radius preorder to be relevant to us, its greatest elements must have some natural interpretation as ``points of maximum probability''.
In some sense, determining what characterises a point of maximum probability is the main challenge, but, motivated by the examples considered throughout the paper, we believe that none of the alternative small-radius preorders are a significant improvement on preorder $\preceq_{0}$.

It seems natural to define an ordering on $X$ by taking limits of the positive-radius preorders $\preceq_r$ as $r \to 0$.
As any binary relation can be viewed as a subset of the Cartesian product $X \times X$, where $(x, x') \in \mathord{\preceq_r}$ precisely when $x \preceq_r x'$, we define some candidate limiting orderings using set-theoretic limits of the net $(\preceq_r)_{r > 0}$.
The corresponding limit set need not be a preorder in general, but we show that certain set-theoretic limits do always yield a preorder.

Indeed, the \defterm{set-theoretic limits inferior and superior} of a net $(A_r)_{r > 0}$ of subsets of $X$ are defined by
\begin{align*}
	\liminf_{r \to 0} A_r &\defeq \bigcup_{R > 0} \bigcap_{r < R} A_r = \Set{y \in X}{y \in A_r \text{ for all $r < R(y)$}}, \\
	\limsup_{r \to 0} A_r &\defeq \bigcap_{R > 0} \bigcup_{r < R} A_r = \Set{y \in X}{\text{for some null sequence $(r_n)_{n \in \Naturals}$, } y \in A_{r_n} \text{ for all $n \in \Naturals$}},
\end{align*}
and the \defterm{Kuratowski lower and upper limits} of $(A_r)_{r > 0}$ are defined by
\begin{align*}
	\Li_{r \to 0} A_r &\defeq \Set{y \in X}{y \text{ is a limit point of the net } (A_r)_{r > 0}}, \\
	\Ls_{r \to 0} A_r &\defeq \Set{y \in X}{y \text{ is a cluster point of the net } (A_r)_{r > 0}}. 
\end{align*}

The following is a useful equivalent characterisation of the Kuratowski limits:

\begin{lemma}[{\citealp[Lemmas~5.2.7 and 5.2.8]{Beer1993}}]
	\label{lem:kuratowski_limits_nets}
	Let $X$ be any metric space and let $(A_r)_{r > 0}$ be a net of subsets of $X$.
	\begin{enumerate}[label=(\alph*)]
		\item $x \in \Li_{r \to 0} A_r$ if and only if there exists a net $(x_r)_{r > 0}$ converging to $x$ with $x_r \in A_r$.
		\item $x \in \Ls_{r \to 0} A_r$ if and only if there exists a decreasing null sequence $(r_n)_{n \in \Naturals}$ and a sequence $(x_{r_n})_{n \in \Naturals}$ converging to $x$ with $x_{r_n} \in A_{r_n}$.
	\end{enumerate}
\end{lemma}

The set limits described above give four different approaches to taking the limit of the sets $(\preceq_r)_{r > 0}$, which we denote
\begin{align*}
	\mathord{\liprec} &\defeq \liminf_{r \to 0} \mathord{\preceq_r}, &\mathord{\lsprec} &\defeq \limsup_{r \to 0} \mathord{\preceq_r}, \\
	\mathord{\Liprec} &\defeq \Li_{r \to 0} \mathord{\preceq_r}, &\mathord{\Lsprec} &\defeq \Ls_{r \to 0} \mathord{\preceq_r}.
\end{align*}

\begin{proposition}
	\label{prop:limits_of_preceq_r}
	Let $X$ be a metric space and let $\mu \in \prob{X}$.
	\begin{enumerate}[label=(\alph*)]
		\item \label{item:limits_of_preceq_r_liprec_preorder} $\liprec$ is a preorder;
		\item \label{item:limits_of_preceq_r_liprec_subset} $\mathord{\liprec}$ is a subset of $\mathord{\preceq_0}$;
		\item \label{item:limits_of_preceq_r_liprec_subset_2} $x \liprec x' \implies x \preceq_0 x'$.
	\end{enumerate}
\end{proposition}

\begin{proof}
	For \ref{item:limits_of_preceq_r_liprec_preorder}, it is routine to check that $\liprec$ is a preorder:
	reflexivity is obvious, and if $x \liprec y$ and $y \liprec z$ then there exists $R > 0$ such that, for all $r < R$, $x \preceq_r y$ and $y \preceq_r z$, giving $x \preceq_r z$ by transitivity of $\preceq_r$.

	For \ref{item:limits_of_preceq_r_liprec_subset}, $(x, x') \in \mathord{\liprec}$ implies that, for some $R > 0$ and all $r < R$, $x \preceq_r x'$.
	Hence,
	\begin{equation*}
		\limsup_{r \to 0} \frac{\crcdf{x}{r}}{\crcdf{x'}{r}} \leq 1,
	\end{equation*}
	so $x \preceq_{0} x'$ by definition.
	Claim \ref{item:limits_of_preceq_r_liprec_subset_2} follows immediately from \ref{item:limits_of_preceq_r_liprec_subset}: if $(x, x') \in \mathord{\liprec}$, then $(x, x') \in \mathord{\preceq_0}$, so $x \preceq_0 x'$.
\end{proof}

As a consequence of \ref{item:limits_of_preceq_r_liprec_subset}, any $\preceq_{0}$-antichain is also a $\liprec$-antichain.
Hence, \Cref{thm:countable_dense_antichain} gives an example of a countable dense $\liprec$-antichain;
this demonstrates that $\liprec$ does not have better behaviour in this regard than $\preceq_0$.

The set-theoretic ordering $\liprec$ can also be criticised as unnecessarily strict in cases where $x' \prec_0 x$ for any $r > 0$, but
\begin{equation*}
	\limsup_{r \to 0} \frac{\crcdf{x}{r}}{\crcdf{x'}{r}} = 1.
\end{equation*}
\Cref{eg:modes_limit_of_r-modes} gives a measure where the $\liprec$-greatest elements and the $\preceq_{0}$-greatest elements differ:
under $\liprec$ only $+1$ is greatest, whereas both $-1$ and $+1$ are $\preceq_{0}$-greatest.
While $\liprec$-greatest elements are reasonable candidates for modes, they do not seem to correspond exactly to any of the established definitions of modes.
To be more precise, while \Cref{prop:limits_of_preceq_r}\ref{item:limits_of_preceq_r_liprec_subset} implies that they are always weak modes, it is not clear whether or not they are strong modes, and not all weak modes are $\liprec$-greatest.

\begin{example}[$\lsprec$ is not necessarily transitive]
	The essential idea is even if $x \preceq_{r_n} y$ for some null sequence $(r_n)_{n \in \Naturals}$, and $y \preceq_{r_n'} z$ for some null sequence $(r_n')_{n \in \Naturals}$, it is possible that $x \not \preceq_r z$ for any $r > 0$.
	For a concrete example of this situation, let
	\begin{equation*}
		X \defeq \bigset{-2 \pm 2^{-3n + 2}}{n \in \Naturals} \cup \bigset{2 \pm 2^{-3n + 2}}{n \in \Naturals} \cup \bigset{\pm 2^{-3n + 2}}{n \in \Naturals}
	\end{equation*}
	with its usual Euclidean metric.
	Define the ``target RCDFs''
	\begin{align*}
		f(2^{-3n + 2}) &\defeq 2^{-3n + 2},~~~~~
		g(2^{-3n + 2}) \defeq 2^{-3n + 1}, \\
		h(2^{-3n + 2}) &\defeq
		\begin{cases}
			2^{-3n + 3}, & \text{if $n$ is odd,}\\
			2^{-3n}, & \text{if $n$ is even,}
		\end{cases}
	\end{align*}
	and let
	\begin{align*}
		\mu & \defeq \frac{1}{Z} \sum_{n \in \Naturals} \frac{1}{2}\bigl(f(2^{-3n + 2}) - f(2^{-3n - 1})\bigr)\bigl(\delta_{-2 - 2^{-3n + 2}} + \delta_{-2 + 2^{-3n + 2}}\bigr)
		\\
		& \phantom{\defeq} \quad + \frac{1}{Z} \sum_{n \in \Naturals} \frac{1}{2}\bigl(g(2^{-3n + 2}) - g(2^{-3n - 1})\bigr)\bigl(\delta_{2 - 2^{-3n + 2}} + \delta_{2 + 2^{-3n - 1}}\bigr) \\
		& \phantom{\defeq} \quad + \frac{1}{Z} \sum_{n \in \Naturals} \frac{1}{2}\bigl(h(2^{-3n + 2}) - h(2^{-3n - 1})\bigr)\bigl(\delta_{-2^{-3n + 2}} + \delta_{2^{-3n - 1}}\bigr),
	\end{align*}
	where $Z$ is a normalisation constant chosen to ensure that $\mu \in \prob{X}$.

	The construction of $\mu$ ensures that the RCDFs at $-2$, $+2$ and $0$ are $\frac{1}{Z} f(r)$, $\frac{1}{Z} g(r)$ and $\frac{1}{Z} h(r)$ for $r \leq 2^{-1}$.
	Then
	\begin{align*}
		\frac{\crcdf{-2}{2^{-3n+2}}}{\crcdf{0}{2^{-3n+2}}} &=
		\begin{cases}
				2^{-1}, & \text{if $n$ is odd,}\\
				2^2, & \text{if $n$ is even,}
		\end{cases} \\
		\frac{\crcdf{0}{2^{-3n + 2}}}{\crcdf{+2}{2^{-3n + 2}}} &=
		\begin{cases}
				2^2 & \text{if $n$ is odd}, \\
				2^{-1} & \text{if $n$ is even,}
		\end{cases} \\
		\frac{\crcdf{-2}{2^{-3n + 2}}}{\crcdf{+2}{2^{-3n + 2}}} &= 2.
	\end{align*}
	It follows that $-2 \lsasymp 0$, because there are null sequences $(r_n)_{n \in \Naturals}$ such that $-2 \preceq_{r_n} 0$ and vice versa; the same argument shows that $0 \lsasymp +2$.
	But $-2 \prec_r 2$ for any $r > 0$, and hence $-2 \not \lsprec 2$.
	This violates transitivity.
\end{example}

\begin{example}[$\Liprec$ and $\Lsprec$ are not necessarily transitive]
	\label{eg:Liprec_Lsprec_not_transitive}
	Let $\mu \in \prob{\Reals}$ be the measure with Lebesgue density $\rho(x) \defeq \one \big[ x \in [0, 1] \big]$.
	We first verify that:
	\begin{enumerate}[label=(\alph*)]
		\item \label{item:Liprec_Lsprec_not_transitive_1} $x \Liprec 1$ for any $x \in \Reals$;
		\item \label{item:Liprec_Lsprec_not_transitive_2} $1 \Liprec x$ for any $x \in \Reals$;
		\item \label{item:Liprec_Lsprec_not_transitive_3} $\nicefrac{1}{2} \not \Lsprec x$ for any $x \in \Reals \setminus [0, 1]$; and
		\item \label{item:Liprec_Lsprec_not_transitive_4} $\nicefrac{1}{2} \not \Liprec x$ for any $x \in \Reals \setminus [0, 1]$.
	\end{enumerate}
	For \ref{item:Liprec_Lsprec_not_transitive_1}, observe that $x \preceq_r 1 - r$ for all small $r$, so $(x, 1 - r) \in \mathord{\preceq_r}$.
	Hence, $(x, 1) \in \mathord{\Liprec}$ by \Cref{lem:kuratowski_limits_nets}.

	For \ref{item:Liprec_Lsprec_not_transitive_2}, use that $1 + r \preceq_r x$ for all small $r$, so $(1 + r, x) \in \mathord{\preceq_r}$.
	This implies that $(1, x) \in \mathord{\Liprec}$.

	For \ref{item:Liprec_Lsprec_not_transitive_3}, suppose that $(x_{r_n}', x_{r_n}) \to (\nicefrac{1}{2}, x)$, and $(x_{r_n}', x_{r_n}) \in \mathord{\preceq_{r_n}}$.
	Let $\varepsilon = \min \{ \absval{ x }, \absval{ x - 1 } \} > 0$.
	There exists $N_1 \in \Naturals$ such that, for all $n \geq N_1$, $\min \{ \absval{ x_{r_n} }, \absval{ x_{r_n} - 1 } \} > \nicefrac{\varepsilon}{2}$.
	As $(r_n)_{n \in \Naturals}$ is a decreasing null sequence, there exists $N_2 \in \Naturals$ such that, for all $n \geq N_2$, $r_n < \nicefrac{\varepsilon}{2}$.
	Picking $N \defeq \max \{ N_1, N_2 \}$, we have $\crcdf{x_{r_n}}{r_n} = 0$ for $n \geq N$, because $\cball{x_{r_n}}{r_n} \cap [0, 1] = \varnothing$.
	It is easy to see that if $x_{r_n}' \to \nicefrac{1}{2}$, then for sufficiently large $n$ $\crcdf{x_{r_n}'}{r_n} > 0$.
	Hence, for all sufficiently large $n$, $x_{r_n}' \not \preceq_{r_n} x_{r_n}$.
	This is a contradiction.

	Claim \ref{item:Liprec_Lsprec_not_transitive_4} is a corollary of \ref{item:Liprec_Lsprec_not_transitive_3}, because $\mathord{\Liprec} \subseteq \mathord{\Lsprec}$.

	Now we prove that $\Liprec$ and $\Lsprec$ are not transitive.
	Suppose for contradiction that they are:
	then \ref{item:Liprec_Lsprec_not_transitive_1} and \ref{item:Liprec_Lsprec_not_transitive_2} imply that every point $x \in \Reals$ is equivalent to $1$, and so all points in $\Reals$ are equivalent by transitivity.
	As $\mathord{\Liprec} \subseteq \mathord{\Lsprec}$, this implies that all points are also $\Lsprec$-equivalent.
	However, \ref{item:Liprec_Lsprec_not_transitive_3} and \ref{item:Liprec_Lsprec_not_transitive_4} show that not all points in $\Reals$ are $\Lsprec$-equivalent or $\Liprec$-equivalent.
\end{example}

\section*{Acknowledgements}
\addcontentsline{toc}{section}{Acknowledgements}

This work has been partially supported by the Deutsche Forschungsgemeinschaft through project \href{https://gepris.dfg.de/gepris/projekt/415980428}{415980428}.
HL is supported by the Warwick Mathematics Institute Centre for Doctoral Training and gratefully acknowledges funding from the University of Warwick and the UK Engineering and Physical Sciences Research Council (Grant number: EP/W524645/1).
The authors would like to thank David Bate, Adam Epstein, Ilja Klebanov, Florian Theil, and Philipp Wacker for helpful discussions.

\bibliographystyle{abbrvnat}
\bibliography{references}

\begin{thebibliography}{33}
\providecommand{\natexlab}[1]{#1}
\providecommand{\url}[1]{\texttt{#1}}
\expandafter\ifx\csname urlstyle\endcsname\relax
  \providecommand{\doi}[1]{doi: #1}\else
  \providecommand{\doi}{doi: \begingroup \urlstyle{rm}\Url}\fi

\bibitem[Agapiou et~al.(2018)Agapiou, Burger, Dashti, and
  Helin]{AgapiouBurgerDashtiHelin2018}
S.~Agapiou, M.~Burger, M.~Dashti, and T.~Helin.
\newblock Sparsity-promoting and edge-preserving maximum a posteriori
  estimators in non-parametric {Bayesian} inverse problems.
\newblock \emph{Inverse Probl.}, 34\penalty0 (4):\penalty0 045002, 37pp., 2018.
\newblock \doi{10.1088/1361-6420/aaacac}.

\bibitem[Aliprantis and Border(2006)]{AliprantisBorder2006}
C.~D. Aliprantis and K.~C. Border.
\newblock \emph{Infinite {Dimensional} {Analysis}: {A} {Hitchhiker's} {Guide}}.
\newblock Springer, Berlin, third edition, 2006.
\newblock \doi{10.1007/3-540-29587-9}.

\bibitem[Anderson(1955)]{Anderson1955}
T.~W. Anderson.
\newblock The integral of a symmetric unimodal function over a symmetric convex
  set and some probability inequalities.
\newblock \emph{Proc. Amer. Math. Soc.}, 6:\penalty0 170--176, 1955.
\newblock \doi{10.2307/2032333}.

\bibitem[Ayanbayev et~al.(2022{\natexlab{a}})Ayanbayev, Klebanov, Lie, and
  Sullivan]{AyanbayevKlebanovLieSullivan2022_I}
B.~Ayanbayev, I.~Klebanov, H.~C. Lie, and T.~J. Sullivan.
\newblock {$\Gamma$}-convergence of {Onsager}--{Machlup} functionals: {I}.
  {With} applications to maximum a posteriori estimation in {Bayesian} inverse
  problems.
\newblock \emph{Inverse Probl.}, 38\penalty0 (2):\penalty0 025005, 32pp.,
  2022{\natexlab{a}}.
\newblock \doi{10.1088/1361-6420/ac3f81}.

\bibitem[Ayanbayev et~al.(2022{\natexlab{b}})Ayanbayev, Klebanov, Lie, and
  Sullivan]{AyanbayevKlebanovLieSullivan2022_II}
B.~Ayanbayev, I.~Klebanov, H.~C. Lie, and T.~J. Sullivan.
\newblock {$\Gamma$}-convergence of {Onsager}--{Machlup} functionals: {II}.
  {Infinite} product measures on {Banach} spaces.
\newblock \emph{Inverse Probl.}, 38\penalty0 (2):\penalty0 025006, 35pp.,
  2022{\natexlab{b}}.
\newblock \doi{10.1088/1361-6420/ac3f82}.

\bibitem[Beer(1993)]{Beer1993}
G.~Beer.
\newblock \emph{Topologies on {Closed} and {Closed} {Convex} {Sets}}, volume
  268 of \emph{Mathematics and {Its} {Applications}}.
\newblock Springer, Dordrecht, 1993.
\newblock \doi{10.1007/978-94-015-8149-3}.

\bibitem[Bj\"{o}rn and Bj\"{o}rn(2011)]{BjoernBjoern2011}
A.~Bj\"{o}rn and J.~Bj\"{o}rn.
\newblock \emph{Nonlinear {Potential} {Theory} on {Metric} {Spaces}}, volume~17
  of \emph{EMS Tracts in Mathematics}.
\newblock European Mathematical Society (EMS), Z\"{u}rich, 2011.
\newblock \doi{10.4171/099}.

\bibitem[Chang and Pollard(1997)]{ChangPollard1997}
J.~T. Chang and D.~Pollard.
\newblock Conditioning as disintegration.
\newblock \emph{Statist. Neerlandica}, 51\penalty0 (3):\penalty0 287--317,
  1997.
\newblock \doi{10.1111/1467-9574.00056}.

\bibitem[Clason et~al.(2019)Clason, Helin, Kretschmann, and
  Piiroinen]{Clason2019GeneralizedModes}
C.~Clason, T.~Helin, R.~Kretschmann, and P.~Piiroinen.
\newblock Generalized modes in {Bayesian} inverse problems.
\newblock \emph{SIAM/ASA J. Uncertain. Quantif.}, 7\penalty0 (2):\penalty0
  652--684, 2019.
\newblock \doi{10.1137/18M1191804}.

\bibitem[Dashti et~al.(2012)Dashti, Harris, and Stuart]{DashtiHarrisStuart2012}
M.~Dashti, S.~Harris, and A.~M. Stuart.
\newblock Besov priors for {Bayesian} inverse problems.
\newblock \emph{Inverse Probl. Imaging}, 6\penalty0 (2):\penalty0 183--200,
  2012.
\newblock \doi{10.3934/ipi.2012.6.183}.

\bibitem[Dashti et~al.(2013)Dashti, Law, Stuart, and
  Voss]{DashtiLawStuartVoss2013}
M.~Dashti, K.~J.~H. Law, A.~M. Stuart, and J.~Voss.
\newblock {MAP} estimators and their consistency in {Bayesian} nonparametric
  inverse problems.
\newblock \emph{Inverse Probl.}, 29\penalty0 (9):\penalty0 095017, 27pp., 2013.
\newblock \doi{10.1088/0266-5611/29/9/095017}.

\bibitem[Davey and Priestley(2002)]{Davey2022Introduction}
B.~A. Davey and H.~A. Priestley.
\newblock \emph{Introduction to {Lattices} and {Order}}.
\newblock Cambridge University Press, New York, second edition, 2002.
\newblock \doi{10.1017/CBO9780511809088}.

\bibitem[Dembo and Zeitouni(1998)]{DemboZeitouni1998}
A.~Dembo and O.~Zeitouni.
\newblock \emph{Large {Deviations} {Techniques} and {Applications}}, volume~38
  of \emph{Applications of Mathematics (New York)}.
\newblock Springer-Verlag, New York, second edition, 1998.
\newblock \doi{10.1007/978-3-642-03311-7}.

\bibitem[D\"{u}rr and Bach(1978)]{DuerrBach1978}
D.~D\"{u}rr and A.~Bach.
\newblock The {Onsager}--{Machlup} function as {Lagrangian} for the most
  probable path of a diffusion process.
\newblock \emph{Comm. Math. Phys.}, 60\penalty0 (2):\penalty0 153--170, 1978.
\newblock \doi{10.1007/BF01609446}.

\bibitem[E et~al.(2004)E, Ren, and Vanden-Eijnden]{ERenVandenEijnden2004}
W.~E, W.~Ren, and E.~Vanden-Eijnden.
\newblock Minimum action method for the study of rare events.
\newblock \emph{Comm. Pure Appl. Math.}, 57\penalty0 (5):\penalty0 637--656,
  2004.
\newblock \doi{10.1002/cpa.20005}.

\bibitem[Fel{\cprime}dman and Nesterenko(1998)]{FeldmanNesterenko1998}
N.~I. Fel{\cprime}dman and {\relax Yu}.~V. Nesterenko.
\newblock \emph{Number {Theory} {IV}}, volume~44 of \emph{Encyclopaedia of
  {Mathematical} {Sciences}}.
\newblock Springer, Berlin, Heidelberg, 1998.
\newblock \doi{10.1007/978-3-662-03644-0}.

\bibitem[Freidlin and Wentzell(1998)]{FreidlinWentzell1998}
M.~I. Freidlin and A.~D. Wentzell.
\newblock \emph{Random {Perturbations} of {Dynamical} {Systems}}, volume 260 of
  \emph{Grundlehren der mathematischen Wissenschaften [Fundamental Principles
  of Mathematical Sciences]}.
\newblock Springer-Verlag, New York, second edition, 1998.
\newblock \doi{10.1007/978-1-4612-0611-8}.
\newblock Translated from the 1979 Russian original by Joseph Sz\"{u}cs.

\bibitem[Grushka(2019)]{Grushka2019}
{\relax Ya}.~I. Grushka.
\newblock On monotonous separately continuous functions.
\newblock \emph{Appl. Gen. Topol.}, 20\penalty0 (1):\penalty0 75--79, 2019.
\newblock \doi{10.4995/agt.2019.9817}.

\bibitem[Hansson(1968)]{Hansson1968}
B.~Hansson.
\newblock Choice structures and preference relations.
\newblock \emph{Synthese}, 18\penalty0 (4):\penalty0 443--458, 1968.
\newblock \doi{10.1007/BF00484979}.

\bibitem[Helin and Burger(2015)]{HelinBurger2015}
T.~Helin and M.~Burger.
\newblock Maximum a posteriori probability estimates in infinite-dimensional
  {Bayesian} inverse problems.
\newblock \emph{Inverse Probl.}, 31\penalty0 (8):\penalty0 085009, 22pp., 2015.
\newblock \doi{10.1088/0266-5611/31/8/085009}.

\bibitem[Kaipio and Somersalo(2005)]{KaipioSomersalo2005}
J.~Kaipio and E.~Somersalo.
\newblock \emph{Statistical and {Computational} {Inverse} {Problems}}, volume
  160 of \emph{Applied Mathematical Sciences}.
\newblock Springer, New York, 2005.
\newblock \doi{10.1007/b138659}.

\bibitem[Klebanov and Wacker(2022)]{KlebanovWacker2022}
I.~Klebanov and P.~Wacker.
\newblock Maximum a posteriori estimators in $\ell^{p}$ are well-defined for
  diagonal {Gaussian} priors, 2022.
\newblock \arXiv{2207.00640}.

\bibitem[Kretschmann(2019)]{Kretschmann2019}
R.~Kretschmann.
\newblock \emph{Nonparametric {Bayesian} Inverse Problems with {Laplacian}
  Noise}.
\newblock PhD thesis, Universit\"at Duisburg-Essen, 2019.
\newblock \doi{10.17185/duepublico/70452}.

\bibitem[Lambley(2022)]{Lambley2022}
H.~Lambley.
\newblock An order-theoretic perspective on modes and maximum a posteriori
  estimation in {Bayesian} inverse problems.
\newblock MA4K9 Research Project (Master's Dissertation), University of
  Warwick, 2022.

\bibitem[Lasanen(2012)]{Lasanen2012_I}
S.~Lasanen.
\newblock Non-{Gaussian} statistical inverse problems. {Part} {I}: {Posterior}
  distributions.
\newblock \emph{Inverse Probl. Imaging}, 6\penalty0 (2):\penalty0 215--266,
  2012.
\newblock \doi{10.3934/ipi.2012.6.215}.

\bibitem[Lassas et~al.(2009)Lassas, Saksman, and
  Siltanen]{LassasSaksmanSiltanen2009}
M.~Lassas, E.~Saksman, and S.~Siltanen.
\newblock Discretization-invariant {Bayesian} inversion and {Besov} space
  priors.
\newblock \emph{Inverse Probl. Imaging}, 3\penalty0 (1):\penalty0 87--122,
  2009.
\newblock \doi{10.3934/ipi.2009.3.87}.

\bibitem[Lie and Sullivan(2018)]{LieSullivan2018}
H.~C. Lie and T.~J. Sullivan.
\newblock Equivalence of weak and strong modes of measures on topological
  vector spaces.
\newblock \emph{Inverse Probl.}, 34\penalty0 (11):\penalty0 115013, 22pp.,
  2018.
\newblock \doi{10.1088/1361-6420/aadef2}.

\bibitem[Malkowsky and Rako\v{c}evi\'{c}(2019)]{MalkowskyRakocevic2019}
E.~Malkowsky and V.~Rako\v{c}evi\'{c}.
\newblock \emph{Advanced {Functional} {Analysis}}.
\newblock CRC Press, Boca Raton, FL, 2019.
\newblock \doi{10.1201/9780429442599}.

\bibitem[Stuart(2010)]{Stuart2010}
A.~M. Stuart.
\newblock Inverse problems: {A} {Bayesian} perspective.
\newblock \emph{Acta Numer.}, 19:\penalty0 451--559, 2010.
\newblock \doi{10.1017/S0962492910000061}.

\bibitem[Sudakov(1959)]{Sudakov1959}
V.~N. Sudakov.
\newblock Linear sets with quasi-invariant measure.
\newblock \emph{Dokl. Akad. Nauk SSSR}, 127:\penalty0 524--525, 1959.

\bibitem[Sullivan(2017)]{Sullivan2017}
T.~J. Sullivan.
\newblock Well-posed {Bayesian} inverse problems and heavy-tailed stable
  quasi-{Banach} space priors.
\newblock \emph{Inverse Probl. Imaging}, 11\penalty0 (5):\penalty0 857--874,
  2017.
\newblock \doi{10.3934/ipi.2017040}.

\bibitem[Szpilrajn(1930)]{Szpilrajn1930}
E.~Szpilrajn.
\newblock Sur l'extension de l'ordre partiel.
\newblock \emph{Fund. Math.}, 16:\penalty0 386--389, 1930.
\newblock \doi{10.4064/fm-16-1-386-389}.

\bibitem[Young(1910)]{Young1910}
W.~Young.
\newblock A note on monotone functions.
\newblock \emph{Q. J. Pure Appl. Math. (Oxf.)}, 41:\penalty0 79--87, 1910.

\end{thebibliography}
\addcontentsline{toc}{section}{References}

\end{document}